\documentclass[10pt]{article}

\usepackage{url}
\usepackage{mathtools}
\usepackage{amssymb}
\usepackage{amsthm}
\usepackage{empheq}
\usepackage{latexsym}
\usepackage{enumitem}
\usepackage{eurosym}
\usepackage{dsfont}
\usepackage{appendix}
\usepackage{color} 
\usepackage[unicode]{hyperref}
\usepackage{frcursive}
\usepackage[utf8]{inputenc}
\usepackage[T1]{fontenc}
\usepackage{geometry}
\usepackage{multirow}
\usepackage{todonotes}
\usepackage{lmodern}
\usepackage{anyfontsize}
\usepackage{stmaryrd}
\usepackage{natbib}
\usepackage{cleveref}

\setcitestyle{numbers,open={[},close={]}}

\definecolor{red}{rgb}{0.7,0.15,0.15}
\definecolor{green}{rgb}{0,0.5,0}
\definecolor{blue}{rgb}{0,0,0.7}
\hypersetup{colorlinks, linkcolor={red},citecolor={green}, urlcolor={blue}}
			
\makeatletter \@addtoreset{equation}{section}

\newtheorem{theorem}{Theorem}[section]
\newtheorem{assumption}[theorem]{Assumption}
\newtheorem{corollary}[theorem]{Corollary}

\newtheorem{lemma}[theorem]{Lemma}
\newtheorem{proposition}[theorem]{Proposition}

\newtheorem{remark}[theorem]{Remark}

\DeclareUnicodeCharacter{014D}{\=o}
\def \E{\mathbb{E}}
\def \F{\mathbb{F}}

\def \H{\mathbb{H}}

\def \N{\mathbb{N}}

\def \P{\mathbb{P}}

\def \R{\mathbb{R}}
\def \S{\mathbb{S}}

\def\Ec{{\cal E}}
\def\Fc{{\cal F}}

\def\Lc{{\cal L}}

\def\Pc{{\cal P}}

\def\Rc{{\cal R}}
\def\Sc{{\cal S}}

\def\Wc{{\cal W}}

\def\Tb{{\bar T}}

\def\x{\times}

\def\Om{\Omega}

\def\Omb{\overline{\Om}}

\def\0{\mathbf{0}}
\def \Ec{\mathcal{E}}
\def \ab{\mathbf{a}}
\def \bb{\mathbf{b}}

\def \xb{\mathbf{x}}

\def \yb{\mathbf{y}}

\def\normeL2#1{\left\|{#1}\right\|_{L^2}}

\def \Lim{\displaystyle\lim}
\def \Liminf{\displaystyle\liminf}

\def\esup{{\rm ess \, sup}}

\def \ab {\boldsymbol{a}}
\def \bb {\boldsymbol{b}}
\def \fb {\boldsymbol{f}}

\def \xbb {\boldsymbol{x}}
\def \yb {\boldsymbol{y}}

\def \Xbb{\mathbf{X}}

\def \Ybb{\mathbf{Y}}

\def \Tb{\overline{T}}

\def \wr{\mathrm{w}}

\def \1{\mathds{1}}

\DeclareMathOperator*{\esssup}{ess\,sup}

 \title{Non--regular McKean--Vlasov equations and calibration problem in local stochastic volatility models\footnote{The author thanks Daniel Lacker and Nizar Touzi for their helpful suggestions and interesting discussions.}}

\author{
    Mao Fabrice {\sc Djete}\footnote{\'Ecole Polytechnique Paris, Centre de Math\'ematiques Appliqu\'ees, mao-fabrice.djete@polytechnique.edu. This work benefits from the financial support of the Chairs {\it Financial Risk} and {\it Finance and Sustainable Development}
    } 
    }
             \date{\today}

\begin{document}

\maketitle
 
\begin{abstract}
    In order to deal with the question of the existence of a calibrated local stochastic volatility model in finance, we investigate a class of McKean--Vlasov equations where a minimal continuity assumption is imposed on the coefficients. Namely, the drift coefficient and, in particular, the volatility coefficient are not necessarily continuous in the measure variable for the Wasserstein topology. In this paper, we provide an existence result and show an approximation by $N$--particle system or propagation of chaos for this type of McKean--Vlasov equations. As a direct result, we are able to deduce the existence of a calibrated local stochastic volatility model for an appropriate choice of stochastic volatility parameters. The associated propagation of chaos result is also proved.
\end{abstract}
\section{Introduction}

Motivated by the local stochastic volatility (LSV) model in finance that we will detail in \Cref{sec:application:vol}, we are interested in this paper in the question of the existence of a process $S$ satisfying
\begin{align} \label{eq:general_LSV}
    \mathrm{d}S_t
    =
    \frac{ S_t\sigma_D(t,S_t)\sigma_t}{\sqrt{\E[\sigma_t^2|S_t]}} \mathrm{d}W_t\;\;\;\mbox{or equivalently}\;\;\;\mathrm{d}X_t=-\frac{1}{2}\frac{\sigma^2_t \sigma_D(t,e^{X_t})^2}{\E[\sigma_t^2|X_t]}\mathrm{d}t
    +
    \frac{\sigma_t \sigma_D(t,e^{X_t})}{\sqrt{\E[\sigma_t^2|X_t]}} \mathrm{d}W_t\;\;\mbox{for}\;\;X_t=\log(S_t)
\end{align}
where $(W_t)_{t\ge 0}$ is an $\R$--valued $\F$--Brownian motion and $(\sigma_t)_{t \ge 0}$ is an $\R$--valued $\F$--predictable process on a filtered probability space $(\Om,\F,\P)$. This question is commonly accepted as a very difficult problem. Indeed, because of the conditional expectation of $\sigma^2_t$ given $S_t$ (or $X_t$),  the process $S$ (or $X$) is a McKean--Vlasov process with a dependence on the distribution which is non--regular (i.e., non--Lipschitz/non--continue) for the Wasserstein distance (see the discussions in \citeauthor*{guyon2013nonlinear} \cite[Chapter 11, Section 11.3]{guyon2013nonlinear} and in \Cref{remark:explanation_first}). For making this question more tractable, the process $(\sigma_t)_{t \ge 0}$ is usually taken as $\sigma_t^2=\overline{v}(t,X_t,Y_t)$ where $Y$ is an Itô process driven by another Brownian motion $B$ potentially correlated to $W$. Even with this simplification, the literature on this subject has remained limited so far. Notice that, if the couple $(X_t,Y_t)$ has a density $p(t,x,y)$ w.r.t. the Lebesgue measure, the expression $\E[\overline{v}(t,X_t,Y_t)|X_t]$ can be rewritten $\E[\overline{v}(t,X_t,Y_t)|X_t]
    =
    \frac{\int_{\R} \overline{v}(t,X_t,y) p(t,X_t,y)\;\mathrm{d}y}{\int_{\R}\; p(t,X_t,y)\;\mathrm{d}y}$.
This observation leads us to treat the variables $\int_{\R} \overline{v}(t,x,y) p(t,x,y)\;\mathrm{d}y$ and $\int_{\R} \; p(t,x,y)\;\mathrm{d}y$ separately. This allows us to prove some general results, one of which is the following. Let $\theta  \in (-1,1)$ s.t. $\mathrm{d}\langle W,B \rangle_t=\theta \mathrm{d}t$. Let us assume that $v$ and $\sigma_D$ are Borel maps bounded above and below by positive constants, and $\R_+ \x \R^2 \ni (t,x,y) \mapsto v(t,x,y)\sigma_D(t,e^x)^2 \in \R$ is Lipschitz in $(x,y)$  uniformly in $t$. Also, there are bounded Borel maps $\R_+ \x \R \ni (t,y) \mapsto \lambda(t,y) \in \R$ and $\R_+ \x \R \ni (t,y) \mapsto \beta(t,y) \in \R$ s.t. $\beta$ is Lipschitz in $y$ uniformly in $t$ and $\beta^2$ bounded below by a positive constant.
\begin{theorem}
    Under some conditions on the initial density $p_0:\R^2 \to \R_+$ $($see {\rm \Cref{ass:density_initial}}$)$, for any $T>0$, there exists an $\R^2$--valued $\F$--adapted continuous process $(X,Y)$ verifying: $\Lc(X_0,Y_0)(\mathrm{d}{x},\mathrm{d}{y})=p_0(x,y)\;\mathrm{d}{x}\;\mathrm{d}{y}$, for each $t \le T$,
    \begin{align} \label{eq:LSV_intro_1}
        \mathrm{d}X_t=-\frac{1}{2}\sigma^2_D(t,e^{X_t})\frac{c + p_X(t,X_t)\;v(t,X_t,Y_t) }{c+p_X(t,X_t)\;\E[v(t,X_t,Y_t)|X_t]}\mathrm{d}t
        +
        \sigma_D(t,e^{X_t})\sqrt{\frac{c + p_X(t,X_t)\;v(t,X_t,Y_t) }{c+p_X(t,X_t)\;\E[v(t,X_t,Y_t)|X_t]}} \mathrm{d}W_t
    \end{align}
    and
    \begin{align} \label{eq:LSV_intro_2}
        \mathrm{d}Y_t=\lambda(t,Y_t)\mathrm{d}t+\beta(t,Y_t) \mathrm{d}B_t
    \end{align}
    where $c$ is a positive constant and $p_X(t,\cdot)$ is the density of $\Lc(X_t).$
\end{theorem}
This result guarantees the existence of a calibrated local stochastic volatility model for the choice of stochastic volatility parameters $\sigma^2_t=\overline{v}(t,X_t,Y_t)=c + p_X(t,X_t)\;v(t,X_t,Y_t)$. Although the existence of a calibrated LSV model is commonly accepted in the mathematical finance community, there are not many papers that have rigorously investigated this issue. We can evoke the paper of \citeauthor*{AbergelTachet} \cite{AbergelTachet}. In \cite{AbergelTachet}, by considering a discretized version (in time and space) of the Fokker--Planck equation associated to \eqref{eq:general_LSV}, the authors use a fixed point argument for establishing an existence result in short--time. Let us cite also \citeauthor{jourdainZhou} \cite{jourdainZhou} who investigate this question in the case where $Y$ only takes a finite number of values. These discretizations are important in the proof of \cite{AbergelTachet} and \cite{jourdainZhou}. In a setting allowing continuous space and time, when all the coefficients involved are homogeneous in time $t$ (in particular $\overline{v}(t,s,y)=\overline{v}(y)$), \citeauthor*{LackerZhang19} \cite{LackerZhang19} prove the existence of a stationary solution of \eqref{eq:system_intro}. The homogeneous assumption is crucial for establishing their result. Moreover, they need the two Brownian motions $W$ and $B$ to be independent. Our theorem seems to be among the most general results regarding the existence of a LSV model calibrating European prices. We are able to consider continuous time and space, to allow the Brownian motions $W$ and $B$ to be correlated, we can have non--homogeneous coefficients and, it is possible to consider any maturity $T$ and start from any initial data (as long as it is smooth enough). In addition to this new existence result, in this paper (see \Cref{sec:apprx_particle}), we provide the approximation of $(X,Y)$ solution of \eqref{eq:LSV_intro_1} + \eqref{eq:LSV_intro_2} by a particle system. To the best of our knowledge, this is the first result of this kind for a system of type \eqref{eq:LSV_intro_1} + \eqref{eq:LSV_intro_2}. This approximation by particle system leads naturally to a numerical scheme allowing to compute a solution $(X,Y)$.

\medskip
The approach used to deal with the LSV model is sufficiently general that it allows us to consider more general framework (see \Cref{thm:main}). A typical example that can be considered is the question of the existence of a couple $\Xbb:=(X^1,X^2)$ verifying:
\begin{align} \label{eq:system_intro}
\begin{cases}
    \mathrm{d}X^{1}_t
    =
    b \big(t,\Xbb_t,p(t,\Xbb_t), (p{h})_{\mathbf{1}}(t,X^1_t)\big) \mathrm{d}t
    +
    {\sigma} \big(t,\Xbb_t,p(t,\Xbb_t), (p{h})_{\mathbf{1}}(t,X^1_t) \big) \mathrm{d}W_t
    \\
    \\
    \mathrm{d}X^{2}_t=\lambda(t,X^2_t)\mathrm{d}t +
    \beta(t,X^{2}_t)\mathrm{d}B_t~~\mbox{with}~~(p{h})_{\mathbf{1}}(t,X^1_t):=\int_{\R} {h}(t,X^1_t,x_2)p(t,X^1_t,x_2)\;\mathrm{d}x_2
\end{cases}
\end{align}
with $p(t,x_1,x_2)$ the density of $\Lc(\Xbb_t)$ w.r.t. the Lebesgue measure. We refer to \Cref{paragr:example} at \Cref{eq:example1} for an application of our main result (\Cref{thm:main}) to the system \eqref{eq:system_intro}.
These types of equations fall into the category of McKean--Vlasov equations. The applied mathematics researchers were studied equations of this type due to its involvement in fields like Physics, Biology, finance, $\cdots$. 

\medskip
For the study of moderately interacting particle systems, in the absence of $(ph)_{\mathbf{1}}(t,X^1_t)$, \eqref{eq:system_intro} has been investigated.  \citeauthor*{Oelschlager1985} \cite{Oelschlager1985} shows existence and uniqueness of this kind of equation when the diffusion coefficient $\sigma$ is constant. \citeauthor*{JOURDAIN1998727} \cite{JOURDAIN1998727} extend the work of \cite{Oelschlager1985} in the case where the drift and diffusion coefficients depend on the densities of the marginal law of the solution $\Xbb_t$. In the framework considered by \cite{JOURDAIN1998727} but without drift coefficients, \citeauthor*{mireilleJabir} \cite{mireilleJabir} weakened the assumptions necessary for the result of \cite{JOURDAIN1998727}.  
In the presence of $(ph)_{\mathbf{1}}(t,X^1_t)$, a common situation is the consideration of the conditional distribution of $X^2_t$ given $X^1_t$ as it is the case in the LSV model previously mentioned. Our existence result of a calibrated LSV model is actually a direct application of the existence of a couple $(X^1,X^2)$ satisfying \eqref{eq:system_intro}. 

\medskip
In this article, we will establish the existence of solution of \eqref{eq:system_intro} in a more general setting under suitable assumptions over the coefficients. We will show then how this result can be used in establishing an existence result for the calibration of local stochastic volatility model. Our approach relies on the use of a fixed point theorem after proving some Sobolev estimates for the density $p$ (see \Cref{sec:leadingproof} for the idea leading the proof). We would like to mention that most of the techniques used in this article are inspired by \citeauthor*{FK-PL-equations} \cite[Chapter 6]{FK-PL-equations}, where the authors present interesting techniques for establishing estimates of the gradient of a solution of the Fokker--Planck equation under mild assumptions.

\medskip
After introducing some notations just below, in \Cref{sec:set-resutls}, we present the framework considered in this paper and state the main results. Namely, an existence result, an approximation result by particle system and the application to the local stochastic volatility model. Most of the technical proofs are completed in \Cref{sec:proofs}.

\medskip




\vspace{1.5mm}

\medskip
{\bf \large Notations}.
	$(i)$
	Given a metric space $(E,\Delta)$, $p \ge 1,$ we denote by $\Pc(E)$ the collection of all Borel probability measures on $E$,
	and by $\Pc_p(E)$ the subset of Borel probability measures $\mu$ 
	such that $\int_E \Delta(e, e_0)^p  \mu({\mathrm{d}}e) < \infty$ for some $e_0 \in E$.
	We equip $\Pc_p(E)$ with the Wasserstein metric $\Wc_p$ defined by
	\[
		\Wc_p(\mu , \mu') 
		~:=~
		\bigg(
			\inf_{\lambda \in \Lambda(\mu, \mu')}  \int_{E \x E} \Delta(e, e')^p ~\lambda( \mathrm{d}e, \mathrm{d}e') 
		\bigg)^{1/p},
	\]
	where $\Lambda(\mu, \mu')$ denote the collection of all probability measures $\lambda$ on $E \x E$ 
	such that $\lambda( \mathrm{d}e, E) = \mu$ and $\lambda(E,  \mathrm{d}e') = \mu'( \mathrm{d}e')$. Equipped with $\Wc_p,$ $\Pc_p(E)$ is a Polish space (see \cite[Theorem 6.18]{villani2008optimal}). For any  $\mu \in \Pc(E)$ and $\mu$--integrable function $\varphi,$ we define $\langle \varphi, \mu \rangle=\langle \mu, \varphi \rangle:=\int_E \varphi(e) \mu(\mathrm{d}e)$.

\medskip
	\noindent $(ii)$	
	Let $\N^*$ denote the set of positive integers. Given non--negative integers $m$ and $n$, we denote by $\S^{m \x n}$ the collection of all $m \x n$--dimensional matrices with real entries, equipped with the standard Euclidean norm, which we denote by $|\cdot|$ regardless of the dimensions. 
	We also denote $\S^n:=\S^{n \times n}$, and denote by $0_{m \times n}$ the element in $\S^{m \times n}$ whose entries are all $0$, and by $\mathrm{I}_n$ the identity matrix in $\S^n$. Let $d \in \N^*,$ for any open set $E \subset  \R^d$, ${\rm diam}(E)$ denotes the diameter of $E$ i.e. $\sup_{\xbb,\yb \in E} |\xbb-\yb|$. Given an open set $\Om_T=(0,T) \x \Om$ in $\R \x \R^d$, where $\Om \subset \R^d$ is an open set and $T>0$, $C^\infty_c(\Om_T)$ denotes the class of infinitely differentiable functions with compact support in $\Om_T$ and  $C^{1,2}(\Om_T)$ denotes the class of functions on $\Om_T$ with continuous derivatives up to the second order in $\xbb$ and a continuous derivative in $t$. When the derivatives have continuous extensions to the closure of $\Om_T$ i.e. $\Omb_T$, we write $C^\infty_c(\Omb_T)$ and $C^{1,2}(\Omb_T)$. The notation $\partial_t$ indicates the continuous derivative in $t$, and the symbols $\nabla^2$ and $\nabla$ are the second and first order derivatives in $\xbb$. 
	Given an open set $U \subset \R^\ell$ for $\ell \ge 1$, for any Banach space $(E,||_E)$, the space $C^{0,\delta}(U,E)$ consists of H\"older continuous of order $\delta \in (0,1)$ functions $f$ on $U$ with value in $E$ with finite norm
    $$
        \|f\|_{C^{0,\delta}(U,E)}
        :=
        \sup_{\xbb \in U} |f(\xbb)|_E
        +
        \sup_{\xbb,\yb \in U,\;\xbb \neq \yb} \frac{|f(\xbb)-f(\yb)|_E}{|\xbb-\yb|^\delta}.
    $$
	The support of a function $f$, i.e., the closure of the set $\{ f \neq 0\}$, is denoted by {\rm supp} $f$.

\section{Setup and Main results} \label{sec:set-resutls}

The general assumptions used throughout this paper are now formulated. {\color{black}We set $\theta \in (-1,1)$ and the probability space $(\Om,\F:=(\Fc_t)_{t \ge 0},\Fc,\P)$ supporting $W$ and $B$ two $\R$--valued $\F$--Brownian motions satisfying $\mathrm{d}\langle W,B \rangle_t=\theta \mathrm{d}t$ .} We are given the numbers $q:=3+1/4$ and $\frac{1}{2} > \eta > \frac{1}{q}$ with $(1-2\eta) q > 2$. We refer to \Cref{sec_functions} for the definition of $L^1_{\ell oc}$, $L^q$ and the Sobolev spaces $H^{q,1}$ and $\H^{q,1}$. We consider the following Borel measurable functions
	\[
		\big[b, \sigma\big]:L^1_{\ell oc}(\R_+;L^1(\R^2;\R)) \x \R_+ \x \R^2 \longrightarrow \R \x \R
		\;\mbox{and}\;\big[\lambda,\beta \big]:\R_+ \x \R \longrightarrow \R \x \R.
	\]
For any $T \in \R_+$ and Borel map $f:\R_+ \x \R \to \R$, if we define $f^T(s,\xbb):=f(s \wedge T,\xbb)$ for all $(s,\xbb) \in \R_+ \x \R^2$, 
when $f \in L^1([0,T] \x \R^2)$, we will consider its extension $f^T$ over $ L^1_{\ell oc}(\R_+;L^1(\R^2;\R))$ and write $[b,\sigma](f)$ instead of $[b,\sigma](f^T)$.


    \begin{assumption} \label{assum:main1} 
        Let $T>0$. There exist $m$ and $M$ positive numbers satisfying: for any $f\in L^1([0,T] \x \R^2))$ s.t. $f \ge 0$, for each $t \in [0,T]$ and $\xbb=(x_1,x_2) \in \R^2,$ if we introduce the map $\bb:=(b^1,b^2)$ and the symmetric matrix $\ab:=(a^{i,j})_{1 \le i ,j \le 2}$ by: $b^1(f)(t,\xbb):=b(f)(t,\xbb),\;\;b^2(f)(t,\xbb):=\lambda(t,x_2)$,
\begin{align*}
    a^{1,1}(f)(t,\xbb)
    :=
    \frac{1}{2}\sigma(f) (t,\xbb)^2,\;a^{2,1}(f)(t,\xbb)
    :=
    a^{1,2}(f)(t,\xbb)
    :=
    \frac{1}{2}{\sigma} (f)(t,\xbb){\beta}(t,x_2)\theta\;\;\mbox{{\rm and}}\;\;a^{2,2}(f)(t,\xbb)
    :=
    \frac{1}{2}{\beta} (t,x_2)^2,
\end{align*}
one has that:

\medskip
{\small $\bullet$ $\mathbf{\underline{Growth\;assumption}}$} For any $(t,\xbb) \in [0,T] \x \R^2$, $|\bb(f)(t,\xbb)| \le M$ and  $m\;\mathrm{I}_2 \le \ab(f)(t,\xbb) \le M\; \mathrm{I}_2.$ In addition, for any open set $E\subset \R^2$ verifying ${\rm diam}(E) \le 1$ 
\begin{align*}
    \|\nabla \ab(f)\|_{L^{q}([0,T] \x E)}^q \le M \bigg[ 1 + \|f\|_{\H^{q,1}([0,T] \x \R^2)}^q + \||\partial_{x_1} f|_{\mathbf{1}}\|_{L^q([0,T] \x \R)}^q +  \| f_{\mathbf{1}}\|_{L^q([0,T] \x \R)}^q \bigg]
\end{align*}
and
\begin{align*}
    \sup_{t \in [0,T]}\|\ab(f)(t,\cdot)\|_{C^{0,1-2{\eta}-2/q}(E)} \le  M \bigg[1+ \sup_{t \in [0,T]}\| f(t,\cdot) \|_{C^{0,1-2{\eta}-2/q}(\R^2)} + \sup_{t \in [0,T]}\|f_{\mathbf{1}}(t,\cdot)\|_{C^{0,1-2{\eta}-1/q}(\R)}
    \bigg]
\end{align*}
where
\begin{align*}
    |f|_{\mathbf{1}}(t,x_1)
    =
    f_{\mathbf{1}}(t,x_1)
    :=
    \int_{\R}f(t,x_1,x_2)\; \mathrm{d}x_2\;\;\mbox{and}\;\;|\partial_{x_1} f|_{\mathbf{1}}(t,x_1)
    :=
    \int_{\R}|\partial_{x_1}f(t,x_1,x_2)|\; \mathrm{d}x_2;
\end{align*}

{\small$\bullet$ $\mathbf{\underline{Continuity\;assumption}}$} The map $[0,T] \x \R \ni (t,x_2) \to \beta(t,x_2) \in \R$ is Lipschitz in $x_2$ uniformly in $t$. Besides, whenever $\lim_{n \to \infty} \| f^n-f\|_{L^1([0,T] \x \R^2)}=0,$ for any $\varphi \in C_c^\infty([0,T] \x \R^2)$
\begin{align*}
    \lim_{n \to \infty} \int_{[0,T] \x \R^2} b^i(f^n)(t,\xbb) \varphi(t,\xbb)\;\mathrm{d}\xbb\;\mathrm{d}t
    =
    \int_{[0,T] \x \R^2} b^i(f)(t,\xbb) \varphi(t,\xbb)\;\mathrm{d}\xbb\;\mathrm{d}t\;\;\mbox{for any}\;1 \le i \le 2
\end{align*}
and
\begin{align*} 
    \lim_{n \to \infty} \int_{[0,T] \x \R^2} a^{i,j}(f^n)(t,\xbb) \varphi(t,\xbb)\;\mathrm{d}\xbb\;\mathrm{d}t
    =
    \int_{[0,T] \x \R^2} a^{i,j}(f)(t,\xbb) \varphi(t,\xbb)\;\mathrm{d}\xbb\;\mathrm{d}t\;\;\mbox{for any}\;1 \le i,j \le 2.
\end{align*}

\end{assumption}

\begin{remark} \label{remark:assump}
    It is worth mentioning that, the constants $m$ and $M$ may depend on $T$ but are independent of the choice of the map $f$. Also notice that when $f$ does not belong to the Sobolev space or to the H\"older space, the preceding inequalities are trivially true since the upper bound is equal to infinity. The value of $q$ i.e. $3+1/4$ follows from the use of some Sobolev embedding theorems in the proofs.
\end{remark}

Let $p_0:\R^2 \to \R_+$ be a density i.e., $\int_{\R^2} p_0(x_1,x_2)\;\mathrm{d}x_1\;\mathrm{d}x_2=1$ s.t. (again see \Cref{sec_functions} for the definition of the Sobolev spaces $H^{q,1}$ and $\H^{q,1}$)
\begin{align} \label{ass:density_initial}
    p_0 \in H^{q,1}(\R^2),\;\;(p_0)_{\mathbf{1}} \in H^{q,1}(\R)\;\;\mbox{and, there is}\;\;\alpha>0,\;\int_{\R} e^{\alpha\;|x_2|^2}\bigg|\int_{\R}p_0(x_1\;x_2)\;\mathrm{d}x_1\bigg|^q\;\mathrm{d}x_2<\infty.
\end{align}
\begin{theorem} \label{thm:main}
    Let $\kappa $ s.t.  $\frac{1}{2} > \eta > \kappa > \frac{1}{q}$ and $T>0$. Under {\rm \Cref{assum:main1}}, there exists an $\R^2$--valued $\F$--adapted continuous process $\Xbb:=(X^1,X^2)$ satisfying: $\Lc(\Xbb_0)(\mathrm{d}\xbb)=p_0(\xbb)\;\mathrm{d}\xbb,$ 
\begin{align} \label{eq:thm:states_1}
    \mathrm{d}X^{1}_t
    =
    b(p) (t,\Xbb_t) \mathrm{d}t
    +
    {\sigma} (p) (t,\Xbb_t) \mathrm{d}W_t\;\;\mbox{and}\;\;\mathrm{d}X^{2}_t
    =
    \lambda (t,X^2_t) \mathrm{d}t
    +
    {\beta} (t,X^2_t) \mathrm{d}B_t,\;\;t \le T
\end{align}
where $\Lc(\Xbb_t)(\mathrm{d}\xb)=p(t,x)\;\mathrm{d}\xb.$
    In addition, the density $p$ satisfies
\begin{align*}
    \|p_{\mathbf{1}} \|_{L^{q}([0,T] \x \R)}^q + \||\partial_{x_1}p|_{\mathbf{1}} \|_{L^{q}([0,T] \x \R)}^q + \|\;p\; \|_{\H^{q,1}([0,T] \x \R^2)}^q <\infty
\end{align*}
and
\begin{align*}
    \| \;p\; \|_{C^{0,\kappa-1/q}\big([0,T],\;C^{0,1-2\eta-2/q}(\R^2)\big)}
        +)(d
        \|\;p_{\mathbf{1}}\;\|_{C^{0,{\kappa}-1/q}\big([0,T],\;C^{0,1-2{\eta}-1/q}(\R)\big)} <\infty.
\end{align*}
\end{theorem}

\begin{remark} \label{remark:explanation_first}
    $(i)$ We would like to point out that finding a process $\Xbb$ satisfying the system mentioned in {\rm \Cref{thm:main}} is not an easy task. Indeed, in this system of equations, the coefficients $($especially the diffusion coefficient $\sigma$ $)$ depend on the marginal distribution of the process i.e. $\Lc(\Xbb_t)$. Given the assumption considered, this dependence can be particularly discontinuous for the Wasserstein distance, see the examples of $b$ and $\sigma$ given in the section just below.  Consequently, the classical results of existence for McKean--Vlasov processes fail in this situation. Besides this discontinuity, the involvement of the distribution generates another technical difficulty. As we will see in the examples, $b$ and $\sigma$ can depend on $(pv)_{\mathbf{1}}(t,x_1)$ where $v:[0,T] \x \R^2 \to \R^k$ is a bounded Borel map. This type of dependence through $p$ creates a {\rm non--local+ \rm local} aspect which makes things even more difficult. It is {\rm non--local}  through the second variable as there is an integration in the second variable $x_2$ in $(pv)_{\mathbf{1}}(t,x_1)$, but it is {\rm local} in the sense that the first variable $x_1$ is fixed in the density $p$ and in the map $v$ $($in the definition of $(ph)_{\mathbf{1}}(t,x_1)$ $)$. It is worth emphasizing that the role of $x_1$ and $x_2$ has its importance in the sense that if the {\rm non--local} part was $x_1$ $($ and not $x_2$ $)$ and the {\rm local} part was $x_2$ $($and not $x_1$ $)$, the difficulty would have been quite different. This case turns out to be less complex and has been investigated in the literature $($see  for instance {\rm \cite{mireilleJabir}}$)$. 
    
    \medskip
    $(ii)$ Let us mention that the shape of the coefficients makes it possible to take into consideration the conditional distribution of $X^2$ given $X^1$ in a non--linear and specific way as we will see in the application to the local stochastic volatility model. Also, the two Brownian motions $W$ and $B$ can be correlated .i.e. $\mathrm{d} \langle W,B \rangle_t=\theta \mathrm{d}t$.

    \medskip
    $(iii)$ Our result contains some known results in the literature. 
    But, the considered framework and assumptions are more general compared to the literature. Especially, the dependence of the coefficients on the marginal distribution are weaker than {\rm \cite{Oelschlager1985}, \cite{JOURDAIN1998727} } and {\rm \cite{mireilleJabir}}. In addition, regarding the application to the LSV model, compared to {\rm \cite{LackerZhang19}}, we provide a non--stationary solution which is more in the spirit of the model and we allow the Brownian motions to be correlated. Also, unlike {\rm \cite{jourdainZhou}} which consider a volatility process with a finite number of values, we are able to take into account a continuum of values namely $\R$. However, we are unable to have a more general dependency for the coefficients $\lambda$ and $\beta$. Also, we cannot provide a uniqueness result for this system.
\end{remark}

\subsection{Example of functions satisfying \Cref{assum:main1}} \label{paragr:example} 

Here, we give some examples of map $(b,\lambda,\sigma,\beta)$ for the application of \Cref{thm:main}. The maps $\lambda$ and $\beta$ need to be bounded, and $\beta$ must be Lipschitz in $x_2$ uniformly in $t$. Let $(h,v):\R_+\x \R^2 \to \R^k \x \R^k$. The maps $h$ and $v$ satisfy: for each $T>0$, $\;\mbox{for all } (t,x_1,x_2) \in [0,T] \x \R^2,$
$$
     \kappa^i_m \le h^i(t,x_1,x_2) \le \kappa^i_M\;\;\mbox{and}\;\;c^i_m \le v^i(t,x_1,x_2) \le c^i_M,\;\mbox{for each }1 \le i \le k,
$$
where $\{\kappa^i_m,\kappa^i_M,c^i_m,c^i_M ,\;1 \le i \le k\}$ are some real numbers.
Let us introduce
\begin{align} \label{eq:set_h}
    \Ec_h^T
    :=
    \big\{ (t,x_1,x_2,e_0,e_1,(e^1_2,\cdots,e^k_2)) \in [0,T] \x \R^2 \x \R_+ \x \R_+ \x \R^k:\;\; e_1 \kappa^i_m \le e^i_2 \le e_1 \kappa^i_M\;\mbox{for each }1 \le i \le k
    \big\}
\end{align}
and
\begin{align} \label{eq:set_v}
    \Ec_v^T
    :=
    \big\{ (t,x_1,x_2,e_0,e_1,(e^1_2,\cdots,e^k_2)) \in [0,T] \x \R^2 \x \R_+ \x \R_+ \x \R^k:\;\; e_1 c^i_m \le e^i_2 \le e_1 c^i_M\;\mbox{for each }1 \le i \le k
    \big\}.
\end{align}
We take the maps $[{b}^\circ,\sigma^\circ]:\R_+\x \R^2 \x \R_+ \x \R_+ \x \R^{k} \to \R^2$ satisfying:  $\Ec_h^T \ni (t,x_1,x_2,e_0,e_1,e_2) \to {b}^\circ(t,x_1,x_2,e_0,e_1,e_2) \in \R$ and $\Ec_v^T \ni (t,x_1,x_2,e_0,e_1,e_2) \to \sigma^\circ(t,x_1,x_2,e_0,e_1,e_2) \in \R$ are bounded and continuous in $(e_0,e_1,e_2)$ uniformly in $(t,x_1,x_2)$. In addition, the matrix $\ab$ defines as $a^{1,1}:=\frac{1}{2} (\sigma^\circ)^2$, $a^{2,2}:=\frac{1}{2} \beta^2$, $a^{1,2}=a^{2,1}:=\frac{1}{2}\theta \beta \sigma^\circ$ verifies $0<\inf_{\Ec^T_v} \inf_{\boldsymbol{z} \neq 0} \frac{\boldsymbol{z}^\top\;a\;\boldsymbol{z}}{|\boldsymbol{z}|^2}$.

\medskip
$\bullet$ A typical example is the following
\begin{align} \label{eq:example1}
    b(f)(t,\xbb):={b}^\circ\big(t,\xbb,f(t,\xbb),f_{\mathbf{1}}(t,x_1),(fh)_{\mathbf{1}}(t,x_1) \big)\;\mbox{and}\;\sigma(f)(t,\xbb):={\sigma}^\circ\big(t,\xbb,f(t,\xbb),f_{\mathbf{1}}(t,x_1),(fv)_{\mathbf{1}}(t,x_1) \big),
\end{align}
$ |\partial_{x_1}v^i(t,\xbb)| \le c^i_M,\;\mbox{for each }1 \le i \le k$, and for each $T>0$,  $\Ec_v^T \ni (t,x_1,x_2,e_0,e_1,e_2) \to \sigma^\circ(t,x_1,x_2,e_0,e_1,e_2) \in \R$ is Lipschitz in $(x_1,x_2,e_0,e_1,e_2)$ uniformly in $t$ i.e. $\esup_{\Ec_v^T}|\nabla \sigma^\circ|<\infty$.
We refer to \Cref{sub_sec_checkExample} for the checking of the \Cref{assum:main1}.

\medskip
$\bullet$ Let us give another example. We take 
\begin{align*}
    b(f)(t,\xbb):={b}^\circ\Big(t,\xbb,\;f(t,\xbb)\;,\int_{\R^2} f(t,\xbb')\;\mathrm{d}\xbb'\;,\; \int_{\R^2} h(t,\xbb')\;f(t,\xbb')\;\mathrm{d}\xbb'\Big)
\end{align*}
and
\begin{align*}
    \sigma(f)(t,\xbb):={\sigma}^\circ\Big(t,\xbb,\;f(t,\xbb)\;,\int_{\R^2} f(t,\xbb')\;\mathrm{d}\xbb'\;,\; \int_{\R^2} v(t,\xbb')\;f(t,\xbb')\;\mathrm{d}\xbb'\Big),
\end{align*}
and  for each $T>0$,  $\Ec_v^T \ni (t,x_1,x_2,e_0,e_1,e_2) \to \sigma^\circ(t,x_1,x_2,e_0,e_1,e_2) \in \R$ is Lipschitz in $(x_1,x_2,e_0)$ uniformly in $(t,e_1,e_2)$  i.e. $\esup_{\Ec_v^T}|\partial_{x_1} \sigma^\circ| + |\partial_{x_2} \sigma^\circ| + |\partial_{e_0} \sigma^\circ| < \infty$.

\subsection{Approximation by particle system} \label{sec:apprx_particle}

In this section, we provide a method to approximate a solution of the system \eqref{eq:thm:states_1} by using interacting processes. 

\medskip
We say $f \in L^1_{prob}(\R_+ \x \R^2)$ if $f \ge 0$, and for each $t \in \R_+,$ $\int_{\R^2} f(t,\xbb)\;\mathrm{d}\xbb=1$. Let $e \ge 1$.  In addition, we assume that
\begin{assumption}
\label{assm_particles}
     For any $T>0$, $f$, $f' \in L^1_{prob}(\R_+ \x \R^2)$ and $(t,\xbb,\xbb') \in [0,T] \x \R^2 \x \R^2$, we have
     \begin{align*}
         |[b,\sigma](f)(t,\xbb)-[b,\sigma](f')(t,\xbb')| 
         \le M \Big[|\xbb-\xbb'| + |f(t,\xbb)-f'(t,\xbb')| + |f_{\mathbf{1}}(t,x_1)-f'_{\mathbf{1}}(t,x'_1)| + \|f-f'\|_{L^1([0,T] \x \R^2)} \Big].
     \end{align*}
\end{assumption}

Let $\delta>0$ and $G_\delta:\R \to \R_+$ be a kernel satisfying: $\int_{\R} G_\delta(x) \mathrm{d}x=1$, for any $\alpha \ge1$,
\begin{align*}
    \Lim_{\delta \to 0}G_\delta * \varphi=\varphi,\;\mbox{a.e.},\;\;\|G_\delta * \varphi\|_{L^\alpha(\R)} \le C\|\varphi\|_{L^\alpha(\R)}\;\mbox{and}\;\|\nabla( G_\delta * \varphi)\|_{L^\alpha(\R)} \le C\|\nabla \varphi\|_{L^\alpha(\R)},\;\mbox{for all }\varphi \in C^\infty_c(\R)
\end{align*}
where $C$ is independent of $\delta$ and $``$ $*$ $"$ denotes the convolution product. We set $\widehat{G}_\delta(\xbb):=G_\delta(x_1)G_\delta(x_2)$.
For any process $\nu:=(\nu_t)_{t \in \R_+} \subset \Pc(\R^2),$ we introduce the map 
$$
    \widehat{G}_\delta *\nu(t,\xbb):=\int_{\R^2} \widehat{G}_\delta(x_1-x_1',x_2-x_2')\nu_t(\mathrm{d}x_1',\mathrm{d}x_2').
$$
Since \Cref{assm_particles} holds, it is easy to check that (see \Cref{sec_approx}), for each $T>0$, $C([0,T];\Pc_e(\R^2)) \x [0,T] \x \R^2 \ni (\nu,t,\xbb) \to [b,\sigma](\widehat{G}_\delta *\nu)(t,\xbb) \in \R^2$ is Lipschitz in $(\nu,\xbb)$ uniformly in $t$ where $C([0,T];\Pc_e(\R^2))$ is the set of continuous functions over $[0,T]$ with values in $\Pc_e(\R^2)$ . Let $\Xbb^\delta:=(X^{1,\delta},X^{2,\delta})$ be an $\R^2$--valued $\F$--adapted continuous process satisfying $p(0,\cdot)=p_0(\cdot)$, for $t \in [0,T]$,
\begin{align*}
    \mathrm{d}X^{1,\delta}_t
    =
    {b}(\mathbf{G}_\delta *\mu^{\delta}) \big(t,\Xbb^\delta_t \big) \mathrm{d}t
    +
    {\sigma}(\mathbf{G}_\delta *\mu^{\delta}) \big(t,\Xbb^\delta_t \big)W_t\;\;\mbox{and}\;\;\mathrm{d}X^{2,\delta}_t=\lambda(t,X^{2,\delta}_t) \mathrm{d}t
    +
    \beta(t,X^{2,\delta}_t)\mathrm{d}B_t
\end{align*}
where $\mu^\delta_t(\mathrm{d}\xbb)=\Lc(\Xbb^\delta_t)(\mathrm{d}\xbb)=p^\delta(t,\xbb)\mathrm{d}\xbb.$ We consider $(W^i,B^i)_{i \ge 1}$ a sequence of independent random variables  s.t for each $i$, $W^i$ and $B^i$ are two $\R$--valued Brownian motions s.t. $\mathrm{d} \langle W^i,B^i \rangle_t=\theta \mathrm{d}t$. In addition to the assumptions of \Cref{thm:main}, the initial density $p_0$ is s.t. $\int_{\R^2} |\xbb|^r p_0(\mathrm{d}\xbb) < \infty$ for $r > e \ge 1$. For each $\delta >0$, let $(\Xbb^{N,1},\cdots,\Xbb^{N,N})$ be the unique solution of: $(\Xbb^{N,1}_0,\cdots,\Xbb^{N,N}_0)$ is i.i.d. with $\Lc(\Xbb^{N,i}_0)(\mathrm{d}\xbb)=p_0(\xbb) \mathrm{d}\xbb$,
\begin{align*}
    \mathrm{d}\Xbb^{N,i,1}_t
    =
    b\big(\widehat{G}_\delta *\mu^{N,\delta} \big) \big(t,\Xbb^{N,i}_t \big) \mathrm{d}t
    +
    \sigma \big(\widehat{G}_\delta *\mu^{N,\delta} \big) \big(t,\Xbb^{N,i}_t \big) \mathrm{d}W^i_t\;\;\mbox{and}\;\;\mathrm{d}\Xbb^{N,i,2}_t
    =
    \lambda \big(t,\Xbb^{N,i,2}_t \big) \mathrm{d}t
    +
    \beta \big(t,\Xbb^{N,i,2}_t \big) \mathrm{d}B^i_t
\end{align*}
where $\mu^{N,\delta}_t:=\frac{1}{N}\sum_{i=1}^N \delta_{\Xbb^{N,i}_t}$. Let us fixed $T >0.$
\begin{proposition} \label{prop:numerical_scheme}
    \begin{enumerate}
        \item The sequence $(p^\delta)_{\delta >0}$ is relatively compact in $C([0,T] \x \R^2)$ for the uniform topology and each limit point $p$ is s.t. $p(t,\xbb)\mathrm{d}\xbb=\Lc(\Xbb_t)(\mathrm{d}\xbb)$ where $\Xbb$ is a solution of the system \eqref{eq:thm:states_1}.
        
        \item For each $\delta >0$, the sequence $((\mu^{N,\delta}_t)_{t \in [0,T]})_{N \ge 1}$ converges towards $(\mu^\delta_t)_{t \in [0,T]}$ in $\Wc_e$. In addition, for any $k \ge 1$ and any bounded measurable map $\phi: [0,T] \x \R^{2k} \to \R$
\begin{align*}
    \Lim_{N \to \infty}\E \bigg[\int_0^T \phi(t,\Xbb^{N,1}_t,\cdots,\Xbb^{N,k}_t)\;\mathrm{d}t \bigg]
    =
    \int_{[0,T] \x \R^{2k}} \phi(t,\xbb^1,\cdots,\xbb^k)\;p^\delta(t,\xbb^1)\cdots p^\delta(t,\xbb^k) \;\mathrm{d}\xbb^1\cdots\mathrm{d}\xbb^k\;\mathrm{d}t.
\end{align*}
    \end{enumerate}




\end{proposition}

\begin{remark}
    The system \eqref{eq:thm:states_1} as a McKean--Vlasov process i.e. with coefficients depending on the distribution can be challenging to simulate numerically. It is usually done through a propagation of chaos result i.e. approximation by particle system. 
    Unlike classical McKean--Vlasov process, for the same reasons explained for the existence result $($see {\rm \Cref{remark:explanation_first}}$)$, the system \eqref{eq:thm:states_1} contains a non--regularity that makes a propagation of chaos result difficult to obtain.
    {\rm \Cref{prop:numerical_scheme}} shows a propagation of chaos result which naturally leads to a numerical scheme to solve the system \eqref{eq:thm:states_1}. However, we leave the more in--depth analysis of the numerical simulation of the system \eqref{eq:thm:states_1} for future research because it requires further study and this is not the main objective of this paper.
\end{remark}

\subsection{Local stochastic volatility model} \label{sec:application:vol} 

As mentioned in the introduction, one of the famous applications of the SDE like \Cref{eq:system_intro} or \Cref{eq:thm:states_1} is in the calibration of local stochastic volatility models (see \citeauthor*{Lipton2002} \cite{Lipton2002}, \citeauthor*{Piterbarg2006} \cite{Piterbarg2006}, \citeauthor*{guyon2013nonlinear} \cite{guyon2013nonlinear}, \citeauthor*{YuZili15} \cite{YuZili15}, \citeauthor*{Saporito2019TheCO} \cite{Saporito2019TheCO}, \citeauthor*{Bayer2022} \cite{Bayer2022} ). Let us here describe this model more precisely. Let $(\sigma_t)_{t \in [0,T]}$ be an $\R$--valued $\F$--predictable process. The process $(\sigma_t)_{t \in [0,T]}$ plays the role of stochastic volatility process. We consider $(S_t)_{t \in [0,T]}$ an $\R$--valued $\F$--adapted which is the spot price following the dynamics (under risk neutral measure with no interest rate and no dividends for simplification)  
\begin{align*}
    \mathrm{d}S_t
    =
    S_t \Sigma(t,S_t) \sigma_t \mathrm{d}W_t.
\end{align*}
If we define the map $\overline{\sigma}$ and the process $\overline{S}$ by
\begin{align*}
    \overline{\sigma}(t,x)
    :=
    \Sigma(t,x) \sqrt{\E[\sigma_t^2|S_t=x]}\;\;\mbox{and}\;\;\mathrm{d}\overline{S}_t
    =
    \overline{S}_t \overline{\sigma}(t,\overline{S}_t) \mathrm{d}W_t\;\;\mbox{with}\;\;\overline{S}_0=S_0,
\end{align*}
by Markovian projection, we know that $\Lc(S_t) = \Lc(\overline{S}_t)$ for each $t \in [0,T]$. With this manipulation, the marginal distribution of $S$ matches the marginal distribution of $\overline{S}$. The dynamics of $\overline{S}$ has the particularity to follow the dynamics of a local volatility model. These two spot prices leads to, for instance, the same European option prices. Let $C(t,K):=C(t,K;S_0)$ be the observed call option price of maturity $t,$ strike $K$ with $S_0$ the initial price of the underlying asset. If we assume that we have access to the observed prices $\{C(t,K),\;K > 0, t >0\},$ by using Dupire's formula (see \citeauthor*{Dupire1994} \cite{Dupire1994}), we know that the local volatility model $\overline{S}$ calibrates the observed call option prices if
\begin{align*}
    \overline{\sigma}(t,K)^2
    =
    \frac{2 \partial_t C(t,K)}{K^2 \partial^2_K C(t,K)}=:\sigma_D(t,K)^2.
\end{align*}
Consequently, in order to make $S$ fits the observed call option prices the map $\Sigma(t,x)$ must verify
$$
    \Sigma(t,x)
    =
    \frac{\sigma_D(t,x)}{\sqrt{\E[\sigma_t^2|S_t=x]}}.
$$
This means $S$ has to satisfy
\begin{align} \label{eq:calibrated_sde}
    \mathrm{d}S_t
    =
    \frac{ S_t\sigma_D(t,S_t)\sigma_t}{\sqrt{\E[\sigma_t^2|S_t]}} \mathrm{d}W_t\;\;\;\mbox{or equivalently}\;\;\;\mathrm{d}X_t=-\frac{1}{2}\frac{\sigma^2_t \sigma_D(t,e^{X_t})^2}{\E[\sigma_t^2|X_t]}\mathrm{d}t
    +
    \frac{\sigma_t \sigma_D(t,e^{X_t})}{\sqrt{\E[\sigma_t^2|X_t]}} \mathrm{d}W_t\;\;\mbox{for}\;\;X_t=\log(S_t).
\end{align}

The existence of a local stochastic volatility model calibrating the European option prices is then equivalent to find $S$ or $X$ solution of \eqref{eq:calibrated_sde}. This problem is considered by the mathematical finance community as a very difficult problem given the implication of the conditional expectation of $\sigma^2_t$ given $S_t$ or $X_t$ which generates a non regularity in the Wasserstein distance (see the discussion in \Cref{remark:explanation_first}) . With an appropriate choice of the stochastic volatility process $(\sigma_t)_{t \in [0,T]},$ when $\sigma_D$ satisfies some conditions, \Cref{thm:main} allows to prove the existence of $X$ (or  $S$).

\medskip
We set $\sigma_t$ by 
$$
    \sigma^2_t:=c + p_X(t,X_t)v(t,X_t,Y_t)
$$
where $c >0,$ $p_X$ is the density of $\Lc(X_t),$ and $\mathrm{d}Y_t=\lambda(t,Y_t)\mathrm{d}t+\beta(t,Y_t) \mathrm{d}B_t$ with $B$ is a Brownian motion verifying $\mathrm{d} \langle W,B \rangle_t=\theta \mathrm{d}t.$ 
We have the next result.
\begin{proposition}
   Let $T>0$ and $\theta \in (-1,1)$. We assume that the maps $v$ and $\sigma_D$ are bounded above and below by positive constants, $\R_+ \x \R^2 \ni (t,x,y) \mapsto v(t,x,y)\sigma_D(t,e^x)^2 \in \R$ is Lipschitz in $(x,y)$ uniformly in $t$, and the maps $\lambda$ and $\beta$ are bounded with $\beta$  Lipschitz in $y$ uniformly in $t$ and $\beta^2$ bounded below by a positive constant. 
   Then, there exists $(X,Y)$ an $\R^2$--valued $\F$--adapted continuous process verifying: $\Lc(X_0,Y_0)(\mathrm{d}x,\mathrm{d}y)=p_0(x,y)\mathrm{d}x \mathrm{d}y$, for each $t \le T$,
    \begin{align} \label{eq:LVS1}
        \mathrm{d}X_t=-\frac{1}{2}\sigma^2_D(t,e^{X_t})\frac{c + p_X(t,X_t)\;v(t,X_t,Y_t) }{c+p_X(t,X_t)\;\E[v(t,X_t,Y_t)|X_t]}\mathrm{d}t
        +
        \sigma_D(t,e^{X_t})\sqrt{\frac{c + p_X(t,X_t)\;v(t,X_t,Y_t) }{c+p_X(t,X_t)\;\E[v(t,X_t,Y_t)|X_t]}} \mathrm{d}W_t
    \end{align}
    and
    \begin{align} \label{eq:LVS2}
        \mathrm{d}Y_t=\lambda(t,Y_t)\mathrm{d}t+\beta(t,Y_t) \mathrm{d}B_t
    \end{align}
    where $p_X(t,\cdot)$ is the density of $\Lc(X_t).$
\end{proposition}

\begin{remark}
    As we will see in proof $($just below$)$, this proposition is essentially an application of {\rm \Cref{thm:main}}. Due to the assumptions we must verify for the application of {\rm \Cref{thm:main}} $($see {\rm \Cref{assum:main1}} $)$, we are unable to take $c=0$. 
\end{remark}

\begin{proof}
    Let us introduce the maps 
    \begin{align*}
        \sigma^\circ(t,x_1,x_2,e_0,e_1,e_2)^2
        :=
        \sigma^2_D(t,e^{x_1}) \frac{c + e_1 v(t,x_1,x_2)}{c + e_2}\;\;\mbox{and}\;\;b^\circ(t,x_1,x_2,e_0,e_1,e_2)
        :=
        -\frac{1}{2}\sigma^\circ(t,x_1,x_2,e_0,e_1,e_2)^2.
    \end{align*}
    Since we assume that $v$ and $\sigma_D$ are bounded above and below by positive constants, and $\R_+ \x \R^2 \ni (t,x,y) \mapsto v(t,x,y)\sigma_D(t,e^x)^2 \in \R$ is Lipschitz in $(x,y)$ uniformly in $t$, we can check that the maps $\R_+ \x \R^2 \ni (t,x,y) \mapsto v(t,x,y) \in \R$ and $\R_+ \x \R \ni (t,x) \mapsto \sigma_D(t,e^x) \in \R$ are Lipschitz uniformly in $t$. Therefore, differentiable almost everywhere with bounded weak derivative.
    Here, we have
    \begin{align*}
        \Ec_v^T
        =
        \big\{ (t,\xbb,e_0,e_1,e_2) \in [0,T] \x \R^2 \x \R_+ \x \R_+ \x \R:\;\; e_1 m_v \le e_2 \le e_1 M_v
    \big\}
    \end{align*}
     where $0<m_v \le v \le M_v$ and $|\nabla v| \le M_v$.
    To prove the proposition, it is enough to show that the map $(b^\circ,\lambda,\sigma^\circ,\beta)$ satisfies assumptions of the first example of \Cref{paragr:example}. For any $e_1 \ge 0$,
    \begin{align*}
        m_D \frac{c+e_1m_v}{c+e_1M_v}
        \le \sigma^\circ(t,x_1,x_2,e_0,e_1, e_2)^2 
        \le M_D \frac{c+e_1M_v}{c+e_1 m_v},\;\;\mbox{for all}\;(t,\xbb,e_0,e_1,e_2) \in \Ec_h^T
    \end{align*}
    where $0 < m_D \le \sigma^2_D \le M_D$. We can verify that the map $\R_+ \ni z \mapsto \frac{c+zm_v}{c+zM_v} \in \R$ is decreasing and the map $\R_+ \ni z \mapsto \frac{c+zM_v}{c+zm_v} \in \R$ is increasing. Therefore, for $(t,x_1,x_2,e_0,e_1,e_2) \in \Ec_h^T$,
    \begin{align}  \label{eq:first_minors}
        m_D \frac{m_v}{M_v}
        \le \sigma^\circ(t,x_1,x_2,e_0,e_1,e_2)^2 
        \le M_D \frac{M_v}{m_v}.
    \end{align}
    Since $\beta^2$ is non--degenerate and $\theta \in (-1,1)$, we verify that $0<\inf_{\Ec^T_v} \inf_{\boldsymbol{z} \neq 0} \frac{\boldsymbol{z}^\top\;\ab\;\boldsymbol{z}}{|\boldsymbol{z}|^2}$ where
    $a^{1,1}:=\frac{1}{2} (\sigma^\circ)^2$, $a^{2,2}:=\frac{1}{2} \beta^2$, $a^{1,2}=a^{2,1}:=\frac{1}{2}\theta \beta \sigma^\circ$. Indeed, by \eqref{eq:first_minors}, $a^{1,1} > 0$, and by the non--degeneracy of $\beta^2$ with $\theta \in (-1,1)$, $a^{1,1}a^{2,2}-a^{1,2}a^{2,1} >0$. Therefore $\ab$ is non--degenerate (on $\Ec^T_v$) since all its principal minors are positive. 

    \medskip
    As $c>0$ and the map $\R^2 \ni (x_1,x_2) \mapsto v(t,x_1,x_2)\sigma^2_D(t,e^{x_1}) \in \R$ is Lipschitz uniformly in $t$, the map $\Ec_v^T \ni (t,x_1,x_2,e_0,e_1,e_2) \mapsto \sigma^\circ(t,(x_1,x_2,e_0,e_1,e_2)) \in \R$ is Lipschitz in $(x_1,x_2,e_0,e_1,e_2)$  uniformly in $t.$ Using the fact that the maps $\R^2 \ni (x_1,x_2) \mapsto v(t,x_1,x_2)\sigma^2_D(t,e^{x_1}) \in \R$, $\R_+ \x \R^2 \ni (t,x,y) \mapsto v(t,x,y) \in \R$ and $\R_+ \x \R \ni (t,x) \mapsto \sigma_D(t,e^x) \in \R$ are Lipschitz (so differentiable a.e. with bounded derivatives), after computations, we get
    \begin{align*}
        \partial_{x_1} \sigma_\circ^2
        =
        \partial_{x_1} \sigma^2_D \frac{c+e_1 v}{c+e_2},\;\partial_{x_2}\sigma_\circ^2=\frac{e_1 \sigma^2_D}{c+e_2}\partial_{x_2}v,\;\partial_{e_0} \sigma_\circ^2=0,\;\partial_{e_1} \sigma_\circ^2=\frac{v \sigma_D^2}{c+e_2}\;\mbox{and}\;\partial_{e_2} \sigma_\circ^2 = - \sigma_D^2\frac{c+e_1v}{(c+e_2)^2}.
    \end{align*}
    Therefore, with $c>0$ and $m_D \frac{m_v}{M_v}
        \le \inf_{(t,x_1,x_2,e_0,e_1, e_2) \in \Ec^T_v}\sigma^\circ(t,x_1,x_2,e_0,e_1, e_2)^2 $, we verify that
    \begin{align*}
       \sup_{(t,x_1,x_2,e_0,e_1, e_2) \in \Ec^T_v}|\nabla \sigma^\circ(t,x_1,x_2,e_0,e_1, e_2)| < \infty.
    \end{align*}

    We can conclude that the map $(b^\circ,\lambda,\sigma^\circ,\beta)$ satisfies the assumptions of the first example of \Cref{paragr:example}. By applying \Cref{thm:main}, there exists $(X,Y)$ satisfying:
    $\Lc(X_0,Y_0)=p_0(x,y)\mathrm{d}x \mathrm{d}y$,
    \begin{align*}
        \mathrm{d}X_t=-\frac{1}{2}\sigma^2_D(t,e^{X_t})\frac{c + p_{\mathbf{1}}(t,X_t)\;v(t,X_t,Y_t) }{c+(pv)_{\mathbf{1}}(t,Y_t)}\mathrm{d}t
        +
        \sigma_D(t,e^{X_t})\sqrt{\frac{c + p_{\mathbf{1}}(t,X_t)\;v(t,X_t,Y_t) }{c+(pv)_{\mathbf{1}}(t,Y_t)}} \mathrm{d}W_t
    \end{align*}
    and
    \begin{align*}
        \mathrm{d}Y_t=\lambda(t,Y_t)\mathrm{d}t+\beta(t,Y_t) \mathrm{d}B_t
    \end{align*}
    where $p$ is the density of $\Lc(X_t,Y_t)$. Notice that $(pv)_{\mathbf{1}}(t,X_t)=p_{\mathbf{1}}(t,X_t)\E[v(t,X_t,Y_t)|X_t]$ with $p_X(t,X_t)=p_{\mathbf{1}}(t,X_t)$.
    This is enough to deduce the proposition.
\end{proof}

\begin{remark}
    By putting \eqref{eq:LVS1} $+$ \eqref{eq:LVS2} into the framework of {\rm\Cref{thm:main}} as we have done in the previous proof, we can check that the couple $(X,Y)$ falls into the context of {\rm \Cref{prop:numerical_scheme}}. We can therefore provide an approximation by particle system of a solution $(X,Y)$ of \eqref{eq:LVS1} $+$ \eqref{eq:LVS2}. 
\end{remark}

\section{Proof of main results} \label{sec:proofs}

\subsection{Main idea leading the proofs} \label{sec:leadingproof}

For ease of reading, we provide in this part the idea leading the proof. The proof of \Cref{thm:main} is essentially an application of a fixed point Theorem namely Schauder fixed point Theorem. This fixed point Theorem is applied on a map defined over an appropriate set of probability density $p$. In order to obtain the appropriate set of probability density, the key step is to provide an explicit estimate of the gradients of $p$ and $p_{\mathbf{1}}$ where $p$ is the density of $\Lc(\Xbb_t)$ with $\Xbb$ an SDE process. These estimates must be given explicitly according to the regularity of the coefficients of the SDE $\Xbb$. For sake of simplification, we present the main idea in less general framework and only for the gradient of $p$. We use Einstein notation and refer to \Cref{sec_functions} for the definition of some functional spaces.  Let $T_0=1$ and ${\boldsymbol{a}}:=(a^{i,j})_{1 \le i,j \le 2}:[0,T_0] \x \R^2 \to \S^2$ be a Borel map s.t. $0 < m \mathrm{I}_2 \le a \le  M \mathrm{I}_2$. Let $\Xbb=(X^1,X^2)$ be a process satisfying
\begin{align*}
    \mathrm{d}\Xbb_t
    =
    \sqrt{2}~ \ab(t,\Xbb_t)^{1/2} \binom{\mathrm{d}W^1_t}{\mathrm{d}W^2_t}\;\mbox{where}\;W^1\mbox{ and }W^2\mbox{ are independent Brownian motions.}
\end{align*}
The measure $\Lc(\Xbb_t)(\mathrm{d}\xbb)\mathrm{d}t$ has a density $p(t,\xbb)$ w.r.t. the Lebesgue measure on $[0,T] \x \R^2$ (see for instance \Cref{prop:density-integrability}). Before giving the detailed proofs in the next sections, let us give some heuristic arguments on the estimate of gradient of $p$ i.e. $\nabla p$. First, by Itô's formula (see also \Cref{lemm:FP-estimates})
\begin{align*}
    \int_{[0,T_0] \x \R^2} \wr(t)\big[ \partial_t \phi + a^{i,j} \partial_{x_i} \partial_{x_j} \phi \big](t,\xbb) p(t,\xbb)\;\mathrm{d}\xbb\;\mathrm{d}t \le \|\wr' \phi\; p\|_{L^1([0,T_0] \x \R^2)},\;\mbox{for any }\wr \in C^\infty([0,T_0])\;\mbox{with }\wr(T_0)=0,
\end{align*} $\phi \in C^{1,2}([0,T_0] \x U_R)$ with $\phi(0,\cdot)=0$ and $U_R$ is a ball of diameter $R$. 

\medskip
For any $R >0$, let $\fb:=(f^i)_{1 \le i \le 2} \subset C^{\infty}_c((0,T_0) \x U_R)$. Let $\phi:[0,T_0] \x \R^2 \to \R$ be a map verifying: $\phi(0,\cdot)=0$, $\phi_{[0,T_0]\x \partial U_R}=0$ and 
$$
    \partial_t \phi + a^{i,j} \partial_{x_i} \partial_{x_j} \phi=\partial_{x_i}f^i\;\;\mbox{on}\;\;(0,T_0) \x U_R.
$$
As we will show in \Cref{prop:estimates_PPDE}, there exists $N$ a constant and $\ell$ a map depending only on $m$, $M$, $T_0$, the dimension $d=2$ and $\alpha \in (1,\infty) \setminus \{2\}$ (we write $\alpha':=\alpha/(\alpha-1)$ see \Cref{sec_functions} just below) s.t. for any $\pi \in (0,1)$,
\begin{align*}
        &\| \phi \|_{\H^{\alpha,1}(U_{R,T_0})} 
        \le 
        N \| \partial_{x_i} f^i \|_{\H^{\alpha,-1}(U_{R,T_0})} + \ell (\pi,\alpha,R) \Big\{ \sup_{t \in (0,T_0)} \|\ab(t,\cdot) \|_{C^{0,\pi}(U_R)} + \|\nabla \ab\|_{L^{\alpha \vee \alpha'}(U_{R,T_0})} \Big\} \|  \phi \|_{\H^{\alpha,1}(U_{R,T_0})}
    \end{align*}
    where $\R_+ \ni r \to \ell (\pi,\alpha,r) \in \R_+$ is continuous and increasing  with $\ell (\pi,\alpha,0)=0$.  Let us take $0<\overline{R} \le 1$ s.t.
    \begin{align*}
        1-\ell (\pi,\alpha,\overline{R}) \sup_{E\in \Sc}\Big\{ \sup_{t \in (0,T_0)} \|\ab(t,\cdot) \|_{C^{0,\pi}(E)} + \|\nabla \ab\|_{L^{\alpha \vee \alpha'}([0,T_0] \x E)} \Big\} >0\;\;\;\mbox{where}\;\;\Sc:=\{E \subset \R^2\mbox{ open set s.t. {\rm diam}}(E) \le 1\}.
    \end{align*}
    In \Cref{prop_estimate FP}, we then show that: for any $\fb:=(f^i)_{1 \le i \le 2} \subset C^{\infty}_c((0,T_0) \x U_R)$ and $R \le \overline{R}$,
    \begin{align*}
        &\bigg| \int_{U_{R,T_0}}  \partial_{x_i}f^i(t,\xbb)\wr(t) p(t,\xbb)\;\mathrm{d}\xbb\; \mathrm{d}t \bigg|
        \\
        &\le \frac{{N} \|\wr'\;p\|_{L^{\alpha'}(U_{R,T_0})} }{1-\ell (\pi,\alpha,\overline{R}) \sup_{E\in \Sc}\Big\{ \sup_{t \in (0,T_0)} \|\ab(t,\cdot) \|_{C^{0,\pi}(E)} + \|\nabla \ab\|_{L^{\alpha \vee \alpha'}([0,T_0] \x E)} \Big\}}\| \partial_{x_i}f^i\|_{\H^{\alpha,-1}( U_{R,T_0})}. 
    \end{align*}
    Since $\| \partial_{x_i}f^i\|_{\H^{\alpha,-1}( U_{R,T_0})} \le \| \fb\|_{L^{\alpha}( U_{R,T_0})}$, this estimate allows us to deduce by a duality argument in \Cref{prop:inegal_general},
    \begin{align*}
        \|\wr \nabla p\|_{L^{\alpha'}([0,T_0] \x \R^2)} 
        &\le 
        \frac{{N}2^{2/{\alpha'}}  }{1-\ell (\pi,\alpha,\overline{R}) \sup_{E\in \Sc}\Big\{ \sup_{t \in (0,T_0)} \| \ab(t,\cdot) \|_{C^{0,\pi}(E)} + \|\nabla \ab\|_{L^{\alpha \vee \alpha'}([0,T_0] \x E)} \Big\}} \|\wr'\;p\|_{L^{\alpha'}([0,T_0] \x \R^2)}
        \\
        &=:A(\pi,\alpha,\alpha')\|\wr'\;p\|_{L^{\alpha'}([0,T_0] \x \R^2)}.
    \end{align*}
    When $\alpha'=q=3+1/4$, by using Sobolev embedding Theorem, the result of Krylov recall in \Cref{prop:density-integrability} and some recursive properties, we find in \Cref{prop:recurcive_estimates} a constant $H:=H\big(A(\pi,q,q'),m,M,|\wr|_{\infty},\cdots,|\wr^{(2)}|_\infty)$ depending only on $A(\pi,q,q' \big)$, $m$, $M$, and the supremum of $\wr$ and its first $2$--derivatives s.t. the constant $H$ is a locally bounded function of the indicated quantities and
    \begin{align*}
        \|\wr\;p\|_{L^{q}([0,T_0] \x \R^2)} \le  H\big(A(\pi,q,q'),m,M,|\wr|_{\infty},\cdots,|\wr^{(2)}|_\infty).
    \end{align*}
    As $p$ satisfies a parabolic equation, if the initial density $p(0,\cdot)$ is regular enough, we can deduce H\"older estimate of $p$ using the $L^q$ estimates of $\nabla p$ (see \Cref{prop:estimates_holder} or \Cref{prop:holder_estimates_p}).
    Therefore, using the detail of the estimates we have obtained, we can choose $L$ and $\overline{R}$ (see \Cref{para_choice_constant}) s.t. 
    $$
        \sup_{E\in \Sc}\Big\{ \sup_{t \in (0,T_0)} \|\ab(t,\cdot) \|_{C^{0,\pi}(E)} + \|\nabla \ab\|_{L^{q}([0,T_0] \x E)} \Big\} \le L\;\mbox{and}\;1-\ell(\pi,q',\overline{R}) L >0
    $$
    leading to
    \begin{align*}
        \sup_{t \in (0,T_0)}\|\wr(t) p(t,\cdot)\|_{C^{0,\pi}(\R^2)} + \|\wr \nabla p\|_{L^{q}([0,T_0] \x \R^2)} \le L.
    \end{align*}
    Consequently, if $\ab$ depends on $p$, we see that we can realize a fixed point. By using similar techniques, we obtain in \Cref{prop:holder_estimates_p1} and \Cref{prop:estimates_marg}, similar estimates for $p_{\mathbf{1}}$ using the regularity of the second component of the SDE i.e. $X^2$. In the next sections, we will provide the details of the arguments mentioned in a more general framework.  In particular, we will see how to handle the map $\wr$ appearing in the estimates.

\subsection{Background materials} \label{sec_bcgrd_materials}

\medskip
    \paragraph*{Reminder on functional spaces} \label{sec_functions} We start by recalling some functional spaces and their properties. Let $d \in \N^*.$  The open ball of $\R^d$ with radius $r/2$ (thus of diameter $r$) centered at $a$ is denoted by $U(a, r)$ or $U_r(a)$. When the center is not important, we will simply write $U_r$. As already mentioned in the previous section, we use Einstein notation i.e. in expressions like $a^{ij}x_iy_i$ and $b_i x_i$ the standard summation rule with respect to repeated indices will be meant.

\medskip
    Given an open set $\Om \subset \R^d$ and a Banach space $(E,||_{\rm E})$, for a nonnegative measure $\mu$ and $q \in [1,\infty)$, the symbols $L^q(\mu;E)$ or $L^q(\Om,\mu;E)$ denote the space of equivalence classes of $\mu$--measurable functions $f$ such that the function $|f|^q_E$ is $\mu$--integrable. This space is equipped with the standard norm
    \begin{align*}
        \| f \|_q
        :=
        \| f \|_{L^q(\Om,\mu;E)}
        :=
        \Big( \int_{\Om} |f(\xbb)|^q_E\; \mu(\mathrm{d}\xbb) \Big)^{1/q}.
    \end{align*}
    We say $f \in L^q_{\ell oc}(\Om,\mu;E)$ if for any compact $Q \subset \Om,$ we have $|f|^q_E\1_{Q}$ is $\mu$--integrable. When $E=\R^\ell$ for $\ell \ge 1$ or when $E$ is obvious, we will simply write $L^q(\mu)$ or $L^q(\Om,\mu)$. Also, the notation $L^q(\Om)$ or $L^q_{\ell oc}(\Om)$ always refers to the situation where $\mu$ is the classical Lebesgue measure. 
    As usual, for $q \in [1,\infty]$ we set $q'$ the conjugate of $q$ by
    $$
        q'
        :=
        \frac{q}{q-1}.
    $$
    We denote by $W^{q,1}(\Om)$ or $H^{q,1}(\Om)$  the Sobolev class of all functions $f \in L^q(\Om)$ whose generalized partial derivatives $\partial_{x_i} f$ are in $L^q(\Om)$. A generalized (or Sobolev) derivative is defined by the equality (the integration by parts formula)
    \begin{align*}
        \int_{\Om} f(\xbb) \partial_{x_i} \varphi(\xbb) \mathrm{d}\xbb
        =
        -\int_{\Om} \partial_{x_i}f(\xbb) \varphi(\xbb) \mathrm{d}\xbb,\;\;\mbox{for all}\;\varphi \in C_c^\infty(\Om).
    \end{align*}
This space is equipped with the Sobolev norm
\begin{align*}
    \|f\|_{q,1}
    :=
    \|f\|_{H^{q,1}(\Om)}
    :=
    \Big(\|f\|^q_q
    +
    \sum_{i=1}^d\|\partial_{x_i}f\|^q_q \Big)^{1/q}.
\end{align*}

\medskip
    The class $W^{q,1}_0(\Om)$ is defined as the closure of $C^\infty_c(\Om)$ in $W^{q,1}(\Om).$
    Let $U_R$ be an open ball of diameter $R$. The dual of the space $W^{q',1}_0(U_R)$ is denoted $W^{q,-1}(U_R)$ or $H^{q,-1}(U_R).$  
    When $u:U_R \to \R$ is a map, we will say $u \in W^{q,-1}(U_R)$ whenever the linear map $W^{q',1}_0(U_R) \ni v \to \langle u ,v \rangle \in \R$ is well defined and continuous. We refer to \cite[Chapter 3]{adams2003sobolev} for an overview on the dual of Sobolev spaces.  
    
    \medskip
    Let $J \subset [0,T]$ be an interval (open or closed) and $U$ be an open set in $\R^d$. Let $\H^{q,1}(J \x U)$ denote the space of all measurable functions $u$ on the set $J \x U$ such that $u (t,\cdot) \in W^{q,1}(U)$ for almost all $t$ and the norm
    $$
        \| u \|_{\H^{q,1}(J \x U)} 
        :=
        \Big( \int_J \|u(t,\cdot)\|_{H^{q,1}(U)}^q \mathrm{d}t \Big)^{1/q}
    $$
    is finite. The space $\H^{q,1}_0(J \x U)$ is defined similarly, but with $H^{q,1}_0(U)$ in place of $H^{q,1}(U)$, and $\H^{q',-1}(J \x U)$ denotes its dual. 

\medskip
\paragraph*{Estimates for parabolic equation} Now, we provide local estimate for a parabolic equation. This estimate is important to give estimates of solution of Fokker--Planck equation as we will see in \Cref{sec_estimates_FP_general}. Let $T_0>0,$ $d=2,$ $U_R \subset \R^d$ be an open ball of diameter $R>0,$ and we denote $U_{R,T_0}:=(0,T_0) \x U_R.$ Also, let us take Borel maps $\bb:=(b^{i})_{1 \le i \le d}:[0,T_0] \x \R^d \to \R^d$, $\ab:=(a^{i,j})_{1 \le i,j \le d}:[0,T_0] \x \R^d \to \S^{d \x d}$, $g:[0,T_0] \x \R^d \to \R$ and a differentiable map in space $\fb:=(f^i)_{1 \le i \le d}:[0,T_0] \x \R^d \to \R^d$. The maps $g$ and $\fb$ have compact supports on $[0,T_0] \x \R^d$ included in $U_{R,T_0}$. We consider $\phi \in C^{1,2}(U_{R,T_0})$ satisfying: $\phi(0,\cdot)=0,$ $\phi_{[0,T_0] \x \partial U_R}=0,$ and 
\begin{align*}
    \partial_t \phi(t,\xbb) + \sum_{i=1}^d b^i(t, \xbb) \partial_{x_i} \phi(t,\xbb) + \sum_{i,j=1}^d a^{ij}(t,\xbb) \partial_{x_i} \partial_{x_j} \phi (t,\xbb) = \sum_{i=1}^d \partial_{x_i} f^i(t,\xbb) + g(t,\xbb)\;\;\mbox{for all}\;(t,\xbb) \in U_{R,T_0}.
\end{align*}
To simplify the notation, we will use the standard summation rule with respect to repeated indices and write
\begin{align*}
    \partial_t \phi(t,\xbb) + b^i(t,\xbb) \partial_{x_i} \phi(t,\xbb) + a^{ij}(t,\xbb) \partial_{x_i} \partial_{x_j} \phi (t,\xbb) = \partial_{x_i} f^i(t,\xbb)+g(t,\xbb)\;\;\mbox{for all}\;(t,\xbb) \in U_{R,T_0}.
\end{align*}
For any $r \ge 1,$ recall that $r'$ is the conjugate of $r$ i.e. $r'=r/(r-1).$  We assume that $\ab$ satisfies $0 < m\;\mathrm{I}_2 \le \ab \le M\;\mathrm{I}_2$. Although presented very differently, the next proposition is inspired by \cite[Lemma 6.2.5]{FK-PL-equations}. 
\begin{proposition} \label{prop:estimates_PPDE}
    Let $\alpha>1$ and $s>(2-s) \vee 1,$ there exists $N>0$ depending only on $T_0$, $\alpha$, $m,$ $M$ and the dimension $d$ s.t. for any $\pi \in (0,1)$
    \begin{align*}
        &\| \phi \|_{\H^{\alpha,1}(U_{R,T_0})} 
        \le 
        N \| \partial_{x_i} f^i + g \|_{\H^{\alpha,-1}(U_{R,T_0})} + \ell (\pi,\alpha,R) \Big\{ \esssup_{t \in (0,T_0)} \|\ab(t,\cdot) \|_{C^{0,\pi}(U_R)} + \Gamma(\alpha,\ab,\bb,U_R,T_0) \Big\} \|  \phi \|_{\H^{\alpha,1}(U_{R,T_0})}
    \end{align*}
    where $\ell: (0,1) \x (1,\infty) \x [0,\infty) \to [0,\infty)$ is a continuous function in the third variable depending only on $(T_0,\alpha,m,M,d,\pi)$ with $\ell(\cdot,\cdot,0)=0$ and
    \begin{align*}
        \Gamma(\alpha,\ab,\bb,U_R,T_0)
        :=
        \|(\nabla \ab, \bb) \|_{L^{\alpha}(U_{R,T_0})} \1_{\alpha > 2} + \|(\nabla \ab,\bb) \|_{L^{2s/(2-s)}(U_{R,T_0})} \1_{\alpha = 2} 
        +
        \|(\nabla \ab,\bb) \|_{L^{\alpha'}(U_{R,T_0})} \1_{\alpha < 2}.
    \end{align*}
\end{proposition}

\begin{proof}
    Let $\xbb_0 \in U_R.$ We set $a_0(t):=a(t,\xbb_0).$ We rewrite $\partial_t \phi + b^i \partial_{x_i} \phi + a^{ij} \partial_{x_i} \partial_{x_j} \phi =\partial_{x_i}f^i + g$ as
    \begin{align*}
        \partial_t \phi(t,\xbb) + a^{ij}_0(t) \partial_{x_i} \partial_{x_j} \phi(t,\xbb)= \partial_{x_i}f^i(t,\xbb)+g(t,\xbb)-b^i \partial_{x_i} \phi - \partial_{x_i} \big( (a^{ij}-a_0^{ij}) \partial_{x_j} \phi \big)(t,\xbb) + \partial_{x_i} a^{ij}(t,\xbb) \partial_{x_j} \phi(t,\xbb).
    \end{align*}
    
    Since $\phi_{[0,T_0] \x \partial U_R}=0,$ for $t \in (0,T_0)$, we can extend $\phi(t,\cdot)$ by $0$ over $\R^d \setminus U_{R}$ while keeping the same regularity. By \Cref{prop_app:estimates}, there exists $N=N(T_0,\alpha,d,m,M)>0$ satisfying
    \begin{align*}
        \| \phi \|_{\H^{\alpha,1}(U_{R,T_0})}
        =
        \| \phi \|_{\H^{\alpha,1}([0,T_0]\x \R^2)}
        &\le N  \| \partial_{x_i} f^i + g - b^i \partial_{x_i} \phi - \partial_{x_i} \big( (a^{ij}-a_0^{ij}) \partial_{x_j} \phi \big) + \partial_{x_i} a^{ij} \partial_{x_j} \phi \|_{\H^{\alpha,-1}([0,T_0] \x \R^2)}
        \\
        &=N  \| \partial_{x_i} f^i + g - b^i \partial_{x_i} \phi - \partial_{x_i} \big( (a^{ij}-a_0^{ij}) \partial_{x_j} \phi \big) + \partial_{x_i} a^{ij} \partial_{x_j} \phi \|_{\H^{\alpha,-1}(U_{R,T_0})}.
    \end{align*}
    First, one has
    \begin{align*}
         \|\partial_{x_i} \big( (a^{ij}-a_0^{ij}) \partial_{x_j} \phi \big) \|_{\H^{\alpha,-1}([0,T_0] \x \R^2)}
         &=
         \|\partial_{x_i} \big( (a^{ij}-a_0^{ij}) \partial_{x_j} \phi \big) \|_{\H^{\alpha,-1}(U_{R,T_0})}
         \\
         &\le 
         \| (a^{ij}-a_0^{ij}) \partial_{x_j} \phi \|_{L^\alpha(U_{R,T_0})}
         \\
         &\le 
         \esssup_{(t,\xbb) \in (0,T_0)\x U_R}|a^{ij}(t,\xbb)-a_0^{ij}(t)|\|  \nabla \phi \|_{L^\alpha(U_{R,T_0})}.
    \end{align*}
    For $\pi \in (0,1),$ $t \in [0,T_0]$ and $\xbb \in U_R$ s.t. $\xbb \neq \xbb_0$, one has
    \begin{align*}
        |a^{ij}(t,\xbb)-a_0^{ij}(t)|=
        |\xbb-\xbb_0|^{\pi} \frac{|a^{ij}(t,\xbb)-a_0^{ij}(t,\xbb_0)|}{|\xbb-\xbb_0|^{\pi}}
        &\le |\xbb-\xbb_0|^{\pi} \|\ab(t,\cdot) \|_{C^{0,\pi}(U_R)}
        \\
        &\le {\rm diam}(U_R)^{\pi} \|\ab(t,\cdot) \|_{C^{0,\pi}(U_R)},
    \end{align*}
    where we recall that ${\rm diam}(U_R)$ denotes the diameter of $U_R.$
    Therefore
    \begin{align*}
        \|\partial_{x_i} \big( (a^{ij}-a_0^{ij}) \partial_{x_j} \phi \big) \|_{\H^{\alpha,-1}(U_{R,T_0})} \le {\rm diam}(U_R)^{\pi} \sup_{t \in (0,T_0)} \|\ab(t,\cdot) \|_{C^{0,\pi}(U_R)} \|  \phi \|_{\H^{\alpha,1}(U_{R,T})}.
    \end{align*}
    Secondly, by \cite[Lemma 1.1.7.]{FK-PL-equations},

    \begin{itemize}
        \item For $\alpha > 2,$ there exists a universal constant $C>0$ depending only on the dimension s.t. 
    \begin{align*}
        \| \partial_{x_i} a^{ij}(t,\cdot) \partial_{x_j} \phi(t,\cdot) \|_{H^{\alpha,-1}(U_{R})}
        &\le C  \|\partial_{x_j}a^{ij}(t,\cdot) \|_{L^{2}(U_{R})} \|  \nabla \phi(t,\cdot) \|_{L^\alpha(U_{R})}.
    \end{align*}
    By using Hölder inequality successively, we find a constant $C$ depending on the dimension and $T_0$ s.t.
    \begin{align*}
        \| \partial_{x_i} a^{ij} \partial_{x_j} \phi \|_{\H^{\alpha,-1}(U_{R,T_0})}
        \le C {\rm diam}(U_R)^{(\alpha-2)/\alpha} \|\nabla \ab \|_{L^{\alpha}(U_{R,T_0})} \|  \phi \|_{\H^{\alpha,1}(U_{R,T_0})}.
    \end{align*}
    Using similar approach, we get 
    \begin{align*}
        \| b^i \partial_{x_i} \phi \|_{\H^{\alpha,-1}(U_{R,T_0})}
        \le C {\rm diam}(U_R)^{(\alpha-2)/\alpha} \|\bb \|_{L^{\alpha}(U_{R,T_0})} \|  \phi \|_{\H^{\alpha,1}(U_{R,T_0})}.
    \end{align*}
    
    \item For $\alpha < 2,$ there exists a universal constant $C>0$ only depending on the dimension s.t. 
    \begin{align*}
        \| \partial_{x_i} a^{ij} \partial_{x_j} \phi \|_{\H^{\alpha,-1}(U_{R,T_0})}
        &\le C {\rm diam}(U_R)^{1-d/\alpha'} \|\partial_{x_j}a^{ij} \|_{L^{\alpha'}(U_{R,T_0})} \|  \nabla \phi \|_{L^\alpha(U_{R,T_0})}
        \\
        &\le C {\rm diam}(U_R)^{1-d/\alpha'} \|\nabla \ab \|_{L^{\alpha'}(U_{R,T_0})} \|  \phi \|_{\H^{\alpha,1}(U_{R,T_0})}.
    \end{align*}
    Again, we can find that
    \begin{align*}
        \| b^i \partial_{x_i} \phi \|_{\H^{\alpha,-1}(U_{R,T_0})}
        &\le C {\rm diam}(U_R)^{1-d/\alpha'} \|\bb \|_{L^{\alpha'}(U_{R,T_0})} \|  \phi \|_{\H^{\alpha,1}(U_{R,T_0})}.
    \end{align*}
    
    \item For $\alpha = 2,$ and $s>(2-s) \vee 1,$  there exists a universal constant $C>0$ only depending on the dimension s.t. 
    \begin{align*}
        \| \partial_{x_i} a^{ij}(t,\cdot) \partial_{x_j} \phi(t,\cdot) \|_{H^{2,-1}(U_{R})}
        &\le C {\rm diam}(U_R)^{2+2/s} \|\partial_{x_j}a^{ij}(t,\cdot)  \nabla \phi(t,\cdot) \|_{L^s(U_{R})}.
    \end{align*}
    Here again, by using Hölder inequality successively, we find a constant $C$ depending on the dimension and $T_0$ s.t.
    \begin{align*}
        \| \partial_{x_i} a^{ij} \partial_{x_j} \phi \|_{\H^{2,-1}(U_{R,T_0})}
        \le C {\rm diam}(U_R)^{2+2/s} \|\nabla \ab \|_{L^{2s/(2-s)}(U_{R,T_0})} \|  \phi \|_{\H^{2,1}(U_{R,T_0})}
    \end{align*}
    and
    \begin{align*}
        \| b^i \partial_{x_i} \phi \|_{\H^{2,-1}(U_{R,T_0})}
        \le C {\rm diam}(U_R)^{2+2/s} \|\bb \|_{L^{2s/(2-s)}(U_{R,T_0})} \|  \phi \|_{\H^{2,1}(U_{R,T_0})}.
    \end{align*}
    
    \end{itemize}

    By combining the previous results, we get the proof of the proposition by setting
    \begin{align*}
        \ell(\pi,\alpha,R)
        :=
        N \Big\{ {\rm diam}(U_R)^{\pi} +C {\rm diam}(U_R)^{1-d/\alpha'} \1_{\alpha<2} +C {\rm diam}(U_R)^{2+2/s} \1_{\alpha=2} +C {\rm diam}(U_R)^{(\alpha-2)/\alpha} \1_{\alpha>2} \Big\}.
    \end{align*}.
\end{proof}
Notice that, similarly to \Cref{remark:assump}, in the previous result when $\ab$ does not belong to a Sobolev or H\"older space, the upper bound is infinite and the result is still true. The proof shows that $\ell$ is given by
\begin{align} \label{eq:def_ell}
        \ell(\pi,\alpha,R)
        :=
        N \Big\{ {\rm diam}(U_R)^{\pi} +C {\rm diam}(U_R)^{1-d/\alpha'} \1_{\alpha<2} +C {\rm diam}(U_R)^{2+2/s} \1_{\alpha=2} +C {\rm diam}(U_R)^{(\alpha-2)/\alpha} \1_{\alpha>2} \Big\}.
    \end{align}
Let $\widehat{C}(\pi,\alpha,\ab,\bb,U_R,T_0)$ be defined by
\begin{align} \label{eq:def_c}
    \widehat{C}(\pi,\alpha,\ab,\bb,U_R,T_0):=
    \esssup_{t \in (0,T_0)} \|\ab(t,\cdot) \|_{C^{0,\pi}(U_R)} + \Gamma(\alpha,\ab,\bb,U_R,T_0).
\end{align}
From \Cref{prop:estimates_PPDE}, we easily deduce the next corollary which gives an estimate of the solution $\phi.$
\begin{corollary} \label{cor:estimates_general}
If $1-\ell (\pi,\alpha,R)\widehat{C}(\alpha,\ab,\bb,U_R,T_0) >0,$ one has
\begin{align*}
    \| \phi \|_{\H^{\alpha,1}(U_{R,T_0})} 
        \le \frac{N}{1-\ell (\pi,\alpha,R)\widehat{C}(\pi,\alpha,\ab,\bb,U_R,T_0)} \| \partial_{x_i}f^i + g \|_{\H^{\alpha,-1}(U_{R,T_0})}.
\end{align*}
\end{corollary}

\medskip
In the next part, we study the regularity, in terms of the Hölder norm, of functions verifying a certain type of equation. We recall that $d=2$ and $T_0>0$. Let $\boldsymbol{\psi}:=(\psi_i)_{0 \le i \le d}: [0,T_0] \x \R^d \to \R^{d+1}$ and $u: [0,T_0] \x \R^d \to \R$ be Borel maps satisfying:
\begin{align*}
    \mathrm{d} \langle u(t,\cdot), \phi \rangle 
    =
    \langle \psi_0 (t,\cdot), \phi \rangle
    +
    \sum_{i=1}^d\langle \psi_i (t,\cdot), \partial_{x_i}\phi \rangle \mathrm{d}t\;\;\mbox{for any }\phi \in C^\infty_c(\R^d).
\end{align*}

Let $\widehat{\alpha} > 2$ and $\frac{1}{2} > \eta > \kappa > \frac{1}{\hat{\alpha}}$ with $(1-2\eta) \widehat{\alpha} > d$.
\begin{proposition} \label{prop:estimates_holder}
    There exists $N >0$  depending only on $(\kappa,\widehat{\alpha},d,\eta,T_0)$ s.t.
    \begin{align*}
        \| u \|_{C^{0,\kappa-1/\hat \alpha}\big([0,T_0],\;C^{0,1-2\eta-d/\hat \alpha}(\R^d) \big)} \le N \Big[ \| u\|_{\H^{\hat{\alpha},1}([0,T_0] \x \R^d)} + \| \boldsymbol{\psi}\|_{\H^{\hat{\alpha},-1}([0,T_0] \x \R^d)} + \| u(0,\cdot) \|_{H^{\hat \alpha,1}(\R^d)} \Big].
    \end{align*}
\end{proposition}

\begin{proof}
    Let $\widehat{\alpha} > 2$ and $\frac{1}{2} > \eta > \kappa > \frac{1}{\hat{\alpha}}$.
    By \cite[Theorem 7.2]{AnaApproach} (see also \cite[Theorem 6.2.2]{FK-PL-equations}), 
    there exists $N >0$ depending only on $(\kappa,\widehat{\alpha},d,\eta,T_0)$ s.t.
    \begin{align*}
         \| u \|_{C^{0,\kappa-1/\hat \alpha}\big([0,T_0],W^{\hat \alpha,1-2\eta}(\R^d)\big)} \le N  \Big[ \|\nabla^2 u\|_{\H^{\hat \alpha,-1}([0,T_0] \x \R^d)} + \| \boldsymbol{\psi}\|_{\H^{\hat \alpha,-1}([0,T_0] \x \R^d)} + \| u(0,\cdot) \|_{H^{\hat \alpha,1}(\R^d)} \Big],
    \end{align*}
    where $W^{\hat \alpha,1-2\eta}(\R^2)$ denotes a fractional Sobolev space (see \cite[Chapter 4]{demengel2012functional} for details). We have $\|\nabla^2 u\|_{\H^{\hat \alpha,-1}([0,T_0] \x \R^d)} \le \| u\|_{\H^{\hat{\alpha},1}([0,T_0] \x \R^d)}$.  Let us mention that the statement in \cite[Theorem 7.2]{AnaApproach} assumes $\| u\|_{\H^{\hat{\alpha},1}([0,T_0] \x \R^d)} + \| \boldsymbol{\psi}\|_{\H^{\hat{\alpha},-1}([0,T_0] \x \R^d)} + \| u(0,\cdot) \|_{H^{\hat \alpha,1}(\R^d)} < \infty$. However, in the case where this quantity is not finite, the inequality is trivially true. By \cite[Theorem 4.47]{demengel2012functional} 
    for $(1-2\eta) \widehat{\alpha} > d,$ there exists a constant $C >0$ independent of $u$ s.t.
    \begin{align*}
         \| u \|_{C^{0,\kappa-1/\hat \alpha}\big([0,T_0],C^{0,1-2\eta-d/\hat \alpha}(\R^d)\big)}
         \le
         C\| u \|_{C^{0,\kappa-1/\hat \alpha}\big([0,T_0],W^{\hat \alpha,1-2\eta}(\R^d)\big)}.
    \end{align*}
    We can conclude the result.
\end{proof}

\subsection{Estimates for Fokker--Planck equation} \label{sec_estimates_FP_general}

This section is dedicated to the analysis of the regularity of solution of Fokker--Planck in terms of the regularity of the coefficients. We will provide here some explicit estimates. 

\medskip
Let $\theta \in (-1,1)$ and $[\widetilde{b},\widetilde{\lambda},\widetilde{\sigma},\widetilde{\beta}]: \R_+ \x \R^2 \to \R^4$ be a Borel map. For $\xbb=(x_1,x_2),$ let us define $b^1(t,\xbb):=\widetilde{b}(t,\xbb)$, $b^2(t,\xbb):=\widetilde{\lambda}(t,\xbb)$,
\begin{align*}
    a^{1,1}(t,\xbb):=\frac{1}{2} \widetilde{\sigma}(t,\xbb)^2,\;a^{2,1}(t,\xbb)=a^{1,2}(t,\xbb):= \frac{1}{2}\widetilde{\sigma}(t,\xbb)\widetilde{\beta}(t,\xbb) \theta\;\;\mbox{and}\;a^{2,2}(t,\xbb):=\frac{1}{2} \widetilde{\beta}(t,\xbb)^2.
\end{align*} 
These maps satisfy
\begin{align}\label{assum:minimal}
    m\; \mathrm{I}_2 \le \ab\;\;\mbox{and}\;\; |(\widetilde{b},\widetilde{\lambda},\widetilde{\sigma},\widetilde{\beta})|\le M.
\end{align} 
Let $\Xbb:=(X^1,X^2)$ be a weak solution of
\begin{align*}
    \mathrm{d}X^1_t
    =
    \widetilde{b}(t,\Xbb_t) \mathrm{d}t
    +
    \widetilde{\sigma}(t,\Xbb_t) \mathrm{d}W_t\;\;\mbox{and}\;\;\mathrm{d}X^2_t=\widetilde{\lambda}(t,\Xbb_t) \mathrm{d}t
    +\widetilde \beta(t,\Xbb_t)\mathrm{d}B_t\;\mbox{where}\;\mathrm{d}\langle W, B \rangle_t= \theta \mathrm{d}t.
\end{align*}
Under these assumptions, the map $p(t,x_1,x_2)$ which denotes the density of $\mu_t:=\Lc(\Xbb_t)$ is well defined almost every $t$. 

\medskip
\subsubsection{Estimates with Borel measurable coefficients } 
Without assuming any additional assumptions over $(\ab,\bb)$ i.e. $\ab$ and $\bb$ just satisfy \eqref{assum:minimal}, we give an estimate of the density $p$. This proposition is essentially an application of \cite[Chapter 2 Section 3 Theorem 4]{KrylovControlledDiffusion}. Its proof mainly uses geometric arguments.

\begin{proposition} \label{prop:density-integrability}
    Let $T_0>0$ and $\gamma \in [1,(d+1)']$. There exists a constant $G>0$ depending only on $T_0,$ $\gamma,$ the dimension $d=2,$ $m$ and M i.e. $G=G(T_0,\gamma,2,m,M)$ s.t.
    \begin{align*}         \int_{[0,T_0] \x     \R^2}           |p(t,x_1,x_2)|^\gamma \mathrm{d}x_1\; \mathrm{d}x_2\; \mathrm{d}t \le G.
    \end{align*}
\end{proposition}

\begin{proof}
    By \cite[Chapter 2 Section 3 Theorem 4]{KrylovControlledDiffusion}, there exists a constant $G>0$ depending only on $T_0,$ $\gamma,$ the dimension $d=2,$ $m$ s.t. for any Borel $f:[0,T_0] \x \R^d \to \R$ 
    \begin{align*}
        \int_{[0,T_0] \x     \R^d} |f(t,\xbb)|\;p(t,\xbb) \;\mathrm{d}\xbb\; \mathrm{d}t \le G \| f \|_{L^{d+1}([0,T_0] \x \R^d)}.
    \end{align*}
    By using a duality argument, we deduce the result.
\end{proof}

\medskip
Recall that $d=2$. By using Itô's formula and taking the expectation, we find that $\mbox{for any}\;\;\varphi \in C^{\infty}_c(\R_+ \x \R^d)$,
\begin{align*}
    \mathrm{d} \langle \varphi, \mu_t \rangle
    &=
    \int_{\R^2} \big[ \partial_t \varphi(t,\xbb) + b^{i}(t,\xbb) \partial_{x_i} \varphi(t,\xbb) + a^{i,j}(t,\xbb) \partial_{x_i} \partial_{x_j} \varphi(t,\xbb) \big] \mu_t(\mathrm{d}\xbb) \mathrm{d}t
    =
    \int_{\R^2} \big[ \partial_t \varphi(t,\xbb) + \Lc \varphi(t,\xbb) \big] \mu_t(\mathrm{d}\xbb) \mathrm{d}t.
\end{align*}
For any $T_0>0,$ let us introduce 
\begin{align*}
    C^\infty_0(T_0):=\{\wr \in C^\infty([0,\infty))\;\mbox{s.t.}\;\mbox{supp}(\wr) =[0,T_0]\}.
\end{align*}
Notice that if $\wr \in C^\infty_0(T_0)$, $\wr^{(j)}(T_0)=0$ and $\wr^{(j)} \in C^\infty_0(T_0)$ for all $j \ge 0$ where $\wr^{(j)}$ denote the $j$th derivative of $\wr$.
We provide, in the next lemma, estimates involving the generator of the process $\Xbb$ i.e. $\Lc$. Let $T_0 >0$, $R>0$ and $U_R=U_R^1 \x U_R^2$.
\begin{lemma} \label{lemm:FP-estimates}
    Let $\wr \in C_0^\infty(T_0)$ and $\alpha \in (1,\infty)$. For any $\phi \in C^{\infty}_c([0,\infty) \x U_R)$ verifying $\phi(0,\cdot)=0$, one has
    \begin{align*}
        \bigg| \int_{[0,T_0] \x \R^d} \wr(t)\big[ \partial_t \phi(t,\xbb) + \Lc\phi(t,\xbb) \big] \mu_t(\mathrm{d}\xbb)\;\mathrm{d}t \bigg|
        \le \int_{U_{R,T_0}} |\wr'(t)\phi(t,\xbb) p(t,\xbb)| \mathrm{d}\xbb\;\mathrm{d}t=\|\wr'\;\phi\;p\|_{L^1(U_{R,T_0})}
    \end{align*}
    and
    \begin{align*}
        &
        \|\wr'\;\phi\;p\|_{L^1(U_{R,T_0})}
        \le 
        2 \| \phi\|_{\H^{\alpha,1}(U_{R,T_0})}\int_{U^2_R}\|\wr'\;p(\cdot,\cdot,x_2)\|_{L^{\alpha'}(U^1_{R,T_0})}
        \mathrm{d}x_2.
    \end{align*}
\end{lemma}

\begin{proof}
Since $\wr \in  C_0^\infty(T_0)$, we have in particular that $\wr(T_0)=0$.  By applying Itô's formula and taking the expectation, we find
\begin{align*}
    -\int_{[0,T_0] \x \R^d} \wr'(t) \phi(t,\xbb) \mu_t(\mathrm{d}\xbb) \mathrm{d}t
    =
    \int_{[0,T_0] \x \R^d} \wr(t)\big[ \partial_t \phi(t,\xbb)  + b^{i}(t,\xbb) \partial_{x_i} \varphi(t,\xbb) + a^{i,j}(t,\xbb) \partial_{x_i} \partial_{x_j} \phi(t,\xbb) \big] \mu_t(\mathrm{d}\xbb) \mathrm{d}t.
\end{align*}
Consequently, we have
\begin{align*}
    \Big| \int_{[0,T_0] \x \R^d} \wr(t) \big[ \partial_t \phi(t,\xbb)  + b^{i}(t,\xbb) \partial_{x_i} \varphi(t,\xbb) + a^{i,j}(t,\xbb) \partial_{x_i} \partial_{x_j} \phi(t,\xbb) \big] \mu_t(\mathrm{d}\xbb) \mathrm{d}t \Big| \le \int_{U_{R,T_0}} |\wr'(t)\phi(t,\xbb) p(t,\xbb)| \mathrm{d}\xbb\;\mathrm{d}t.
\end{align*}

\medskip
    Next, for any $g \in C^\infty_c(U^2_R),$ by Morrey's inequality, for any $\alpha >1,$ there exists a constant $C$ depending only on $\alpha$ s.t.
    \begin{align*}
        \sup_{x_2 \in U^2_R} |g(x_2)| \le C \| g \|_{H^{\alpha,1}(U^2_R)}.
    \end{align*}
    The constant $C$ may be taken $C=C(\alpha)=2^{\frac{\alpha-1}{\alpha}}$ (see for instance the proof of \cite[Section 5.6 Theorem 4]{evans2010partial}). Then, $C(r) \le 2$ for any $r \ge 1$.
    Let $\chi\in C^\infty([0,T_0])$ s.t. $\chi(0)=0$ and $f \in C^\infty_c(U^1_R).$ We set $\psi(t,x_1,x_2):=\chi(t) f(x_1)g(x_2).$ One has
    \begin{align*}
        \| \psi\|_{\H^{\alpha,1}(U_{R,T_0})}^\alpha
        :=
        \|\chi\|_{L^\alpha([0,T_0])}^\alpha \bigg[ \|f\|_{L^\alpha(U^1_{R})}^\alpha \|g\|_{L^\alpha(U^2_{R})}^\alpha
        +
        \|f'\|_{L^\alpha(U^1_{R})}^\alpha \|g\|_{L^\alpha(U^2_{R})}^\alpha
        +
        \|f\|_{L^\alpha(U^1_{R})}^\alpha \|g'\|_{L^\alpha(U^2_{R})}^\alpha
        \bigg].
    \end{align*}
    We observe that
    \begin{align*}
        \|f\|_{L^\alpha(U^1_{R})}\|\chi\|_{L^\alpha([0,T_0])}\sup_{y \in U^2_R} |g(y)| \le C \| g \|_{H^{\alpha,1}(U^2_R)} \|f\|_{L^\alpha(U^1_{R})}\|\chi\|_{L^\alpha([0,T_0])}
        \le C \| \psi\|_{\H^{\alpha,1}(U_{R,T_0})}.
    \end{align*}
    Then, we find that
    \begin{align*}
        \int_{U_{R,T_0}} |\psi(t,\xbb)\wr'(t)p(t,\xbb)| \mathrm{d}\xbb\; \mathrm{d}t
        &=
        \int_{U_{R,T_0}} |\chi(t) f(x_1)g(x_2)\wr'(t)p(t,x_1,x_2)| \mathrm{d}x_1\; \mathrm{d}x_2\; \mathrm{d}t
        \\
        &\le 
        \sup_{y' \in U^2_R} |g(y')| \int_{U^2_R}\int_{U^1_{R,T_0}} |\chi(t) f(x_1)| \wr'(t)p(t,x_1,x_2) \mathrm{d}x_1\; \mathrm{d}x_2\; \mathrm{d}t
        \\
        &\le 
        \|f\|_{L^\alpha(U^1_{R})}\|\chi\|_{L^\alpha([0,T_0])}\sup_{y' \in U^2_R} |g(y')| \int_{U^2_R}\Big(\int_{U^1_{R,T_0}} |\wr'(t)p(t,x_1,x_2)|^{\alpha'} \mathrm{d}x_1\; \mathrm{d}t \Big)^{1/\alpha'}
        \mathrm{d}x_2
        \\
        &\le 
         C \| \psi\|_{\H^{\alpha,1}(U_{R,T_0})} \int_{U^2_R}\Big(\int_{U^1_{R,T_0}} |\wr'(t)p(t,x_1,x_2)|^{\alpha'} \mathrm{d}x_1\; \mathrm{d}t \Big)^{1/\alpha'}
        \mathrm{d}x_2.
    \end{align*}
    By using a reasoning by approximation, we can deduce that the previous inequality is true for any map $\psi$ belonging to $\H^{\alpha,1}(U_{R,T_0})$. This is enough to deduce the result.
\end{proof}

\subsubsection{Estimates with additional assumptions over the coefficients} 

Now, by considering additional assumptions over the coefficients, we provide some inequalities. Recall that we still keep $d=2$. Let $T_0>0$, $\alpha \in (1,\infty) \setminus \{2\}$ and $\pi \in (0,1)$. Recall that $\widehat{C}(\pi,\alpha,\ab,\bb,U_R,T_0)$ is defined in \eqref{eq:def_c} and $\ell(\pi,\alpha,R)$ is defined in \eqref{eq:def_ell}. We introduce the collection of sets $\Sc$ and the constant $\widehat{C}_\infty(\pi,\alpha,\ab,\bb,T_0)$ by
\begin{align} \label{eq:def_C}
    \Sc
    :=
    \big\{ E \subset \R^2\;\mbox{ open set s.t.}\;{\rm diam}(E)\le 1 \big\}\;\;\mbox{and}\;\;\widehat{C}_\infty(\pi,\alpha,\ab,\bb,T_0)
    :=
    \sup_{E \in \Sc} \widehat{C}(\pi,\alpha,\ab,\bb,E,T_0).
\end{align}

For any $U_R \subset \R^d$ with $R \le 1,$ we have $\widehat{C}(\pi,\alpha,\ab,\bb,U_R,T_0) \le \widehat{C}_\infty(\pi,\alpha,\ab,\bb,T_0)$. Notice that $\widehat{C}(\pi,\alpha,\ab,\bb,U_R,T_0)=\widehat{C}(\pi,\alpha',\ab,\bb,U_R,T_0)$, recall that $\alpha'$ is the conjugate of $\alpha$ i.e. $\alpha'=\alpha/(\alpha-1)$. We define $\Rc$ by
\begin{align*}
    \Rc(\pi,\alpha,\ab,\bb,T_0) 
    :=
    \Big{\{}0<R\le 1\;\mbox{ s.t. }1-\big[\ell(\pi,\alpha,R)+\ell(\pi,\alpha',R)\big]\widehat{C}_\infty(\pi,\alpha,\ab,\bb,T_0)>0\Big{\}}.
\end{align*}
We assume that
\begin{align}\label{eq:general_cond}
    \Rc(\pi,\alpha,\ab,\bb,T_0)\mbox{ is non--empty}.
\end{align}

The next proposition gives estimates that will help us find estimations for the gradients of $p$ and $p_{\mathbf{1}}$. The first estimate is used for estimating the gradient of $p$ and the second for estimating the gradient of $p_{\mathbf{1}}$.
\begin{proposition}
 \label{prop_estimate FP}    Let $\wr \in C^\infty_0(T_0)$. There exists ${N} >0$ depending only on $T_0$, $\alpha$, $m,M,$ and the dimension $d=2$ s.t. for $\overline{R} \in \Rc(\pi,\alpha,\ab,\bb,T_0)$, any $R \le \overline{R}$ and, $g$,  $(f^i)_{1 \le i \le d}$, and $v$ belonging to $C^\infty_c(U_{R,T_0})$,
    \begin{align*}
        \bigg| \int_{U_{R,T_0}}  \big[\partial_{x_i}f^i+g\big]v(t,\xbb)\wr(t) \mu_t(\mathrm{d}\xbb) \mathrm{d}t \bigg|
        \le \frac{{N}}{1-\ell(\pi,\alpha,R)\widehat{C}_\infty(\pi,\alpha,\ab,\bb,T_0)} \| (\partial_{x_i}f^i+g)v\|_{\H^{\alpha,-1}( U_{R,T_0})} \|\wr'\;p\|_{L^{\alpha'}(U_{R,T_0})} 
    \end{align*}
    and
    \begin{align*}
        &\bigg| \int_{U_{R,T_0}}  \big[\partial_{x_i}f^i+g\big]v(t,\xbb)\wr(t) \mu_t(\mathrm{d}\xbb) \mathrm{d}t \bigg|
        \\
        &\le \frac{2{N}}{1-\ell(\pi,\alpha,R)\widehat{C}_\infty(\pi,\alpha,\ab,\bb,T_0)} \| (\partial_{x_i}f^i+g)v\|_{\H^{\alpha,-1}( U_{R,T_0})} \int_{U^2_R}\|\wr'\;p(\cdot,\cdot,x_2)\|_{L^{\alpha'}(U^1_{R,T_0})}
        \mathrm{d}x_2.
    \end{align*}
\end{proposition}

\begin{proof}
   The numbers $\pi$ and $\alpha$ are given. Since we have $\overline{R} \in \Rc(\pi,\alpha,\ab,\bb,T_0)$, this means that
   \begin{align*}
       \mbox{$\overline{R}$ > 0 and $1-\big[\ell(\pi,\alpha,\overline{R})+\ell(\pi,\alpha',\overline{R})\big]\widehat{C}_\infty(\pi,\alpha,\ab,\bb,T_0)>0$},
   \end{align*}
   therefore 
   $$
    \sup_{E \in \Sc}\widehat{C}(\pi,\alpha,\ab,\bb,E,T_0)=
    \sup_{E \in \Sc}\bigg[\esssup_{t \in (0,T_0)} \|\ab(t,\cdot) \|_{C^{0,\pi}(E)} + \Gamma(\alpha,\ab,\bb,E,T_0) \bigg]< \infty.
$$
Consequently, there exists a sequence of smooth functions  $(\ab^n,\bb^n)_{n \ge 1} \subset C^\infty_c([0,T_0] \x \R^2)$ s.t. $m \mathrm{I}_2 \le \ab^n \le M \mathrm{I}_2$, $\lim_{n \to \infty}(\ab^n,\bb^n)=(\ab,\bb)$ a.e., and for any $E \in \Sc$,
    \begin{align} \label{conv_coeff_proof}
        \Lim_{n \to \infty}\esssup_{t \in (0,T_0)} \|\ab^n(t,\cdot) \|_{C^{0,\pi}(E)} = \esssup_{t \in (0,T_0)}\|\ab(t,\cdot) \|_{C^{0,\pi}(E)}\;\;\mbox{and}\;\; \Lim_{n \to \infty}\Gamma(\alpha,\ab-\ab^n,\bb-\bb^n,E,T_0) 
        =0.
    \end{align}
    Notice that in \Cref{conv_coeff_proof}, it is not possible to approximate $\ab$ in $C([0,T_0];C^{0,\pi}(E))$ as $\ab$ is not necessary continue in $t$. 
   For each $n \ge 1$, by \cite[Theorem 10.4.1]{krylov1996lectures}, there exists $\phi^n \in C^{1,2}((0,T_0) \x U_R)$ with derivatives belonging to H\"older spaces satisfying: $\phi^n(0,\cdot)=0,$ $\phi^n_{[0,T_0] \x \partial U_R}=0,$ and for any $(t,\xbb) \in (0,T_0) \x U_R$,
\begin{align*}
    &\partial_t \phi^n(t,\xbb) + b^{n,i}(t,\xbb) \partial_{x_i} \phi^n(t,\xbb) + a^{n,i,j}(t,\xbb) \partial_{x_i} \partial_{x_j} \varphi^n(t,\xbb)
    =\partial_t \phi^n(t,\xbb) + \Lc^n\varphi^n(t,\xbb)
    =
    \big[\partial_{x_i}f^i+g\big]v(t,\xbb).
\end{align*}
Let us observe that as $(0,1) \ni r \to \ell(\pi,\alpha,r) + \ell(\pi,\alpha',r) \in \R_+$ is increasing, for any $R \le \overline{R}$,
\begin{align*}
    1-\big[\ell(\pi,\alpha,R)+\ell(\pi,\alpha',R)\big]\widehat{C}(\pi,\alpha,\ab,\bb,U_R,T_0) 
    \ge 
    1-\big[\ell(\pi,\alpha,\overline{R})+\ell(\pi,\alpha',\overline{R})\big]\widehat{C}_\infty(\pi,\alpha,\ab,\bb,T_0) >0.
\end{align*}
Thanks to the previous observation, for $n$ large enough, $1-\big[\ell(\pi,\alpha,R)+\ell(\pi,\alpha',R)\big]\widehat{C}(\pi,\alpha,\ab^n,\bb^n,U_R,T_0)>0$.
We can check that $\widehat{C}(\pi,\alpha,\ab^n,\bb^n,U_R,T_0)=\widehat{C}(\pi,\alpha',\ab^n,\bb^n,U_R,T_0)$, by \Cref{cor:estimates_general}, there exists $N$ depending only on $(d,T_0,M,m,\alpha)$ s.t. for any $z \in \{\alpha,\alpha'\}$
\begin{align} \label{eq:estimates_phi_app}
    \|\phi^n\|_{\H^{z,1}(U_{R,T_0})} 
    \le \frac{N}{1-\ell(\pi,z,R)\widehat{C}(\pi,\alpha,\ab^n,\bb^n,U_R,T_0)} \| (\partial_{x_i}f^i+g)v\|_{\H^{z,-1}( U_{R,T_0})}.
\end{align}
Let $p^n(t,x_1,x_2)$ be the density of $\mu^n_t:=\Lc(\Xbb^n_t):=\Lc(X^{n,1}_t,X^{n,2}_t)$ where: $\Lc(\Xbb^n_0)=\Lc(\Xbb_0)$,
\begin{align*}
    \mathrm{d}\Xbb^n_t
    =
    \bb^n(t,\Xbb^n_t) \mathrm{d}t +  \ab^n(t,\Xbb^n_t)^{1/2} \binom{\mathrm{d}W^1_t}{\mathrm{d}W^2_t},
\end{align*}
and $(W^1,W^2)$ an $\R^2$--valued Brownian motion.
By \Cref{lemm:FP-estimates}, for any $\wr \in C^\infty_0(T_0)$,
    \begin{align*}
        \bigg| \int_{[0,T_0] \x \R^d} \wr(t)\big[ \partial_t \phi^n(t,\xbb) + \Lc^n\phi^n(t,\xbb) \big] \mu^n_t(\mathrm{d}\xbb) \mathrm{d}t \bigg|
        \le \int_{U_{R,T_0}} |\wr'(t)\phi^n(t,\xbb) p^n(t,\xbb)| \mathrm{d}\xbb\;\mathrm{d}t=\|\wr'\;\phi^n\;p^n\|_{L^1(U_{R,T_0})}.
    \end{align*}
    Therefore,
   \begin{align} \label{eq:estimates:appr}
       &\Liminf_{n \to \infty}\int_{U_{R,T_0}}  \wr(t)\big[\partial_{x_i}f^i+g\big]v(t,\xbb) \mu^n_t(\mathrm{d}\xbb) \mathrm{d}t \nonumber
       \\
       &= 
       \liminf_{n \to \infty}\int_{0}^{T_0} \int_{\R^d} \wr(t)\big[ \partial_t \phi^n(t,\xbb) + \Lc^n\phi^n(t,\xbb) \big] \mu^n_t(\mathrm{d}\xbb) \mathrm{d}t
       \le \Liminf_{n \to \infty} \|\wr'\;\phi^n\;p^n\|_{L^1(U_{R,T_0})}.
   \end{align}
   We know that $\| (\partial_{x_i}f^i+g)v\|_{\H^{\alpha',-1}( U_{R,T_0})} + \| (\partial_{x_i}f^i+g)v\|_{\H^{\alpha,-1}( U_{R,T_0})}< \infty$, then by \eqref{eq:estimates_phi_app}, $\sup_{n \ge 1} \|\phi^n\|_{\H^{{\alpha'},1}(U_{R,T_0})} + \|\phi^n\|_{\H^{{\alpha},1}(U_{R,T_0})} < \infty$. We can deduce that $(\phi^n)_{n \ge 1}$ is relatively compact in $L^{\hat \alpha}(U_{R,T_0})$ where $\hat{\alpha}:=\min(\alpha,\alpha')$ . Up to a sub--sequence, we can assume that there is $\phi$ s.t. $\lim_{n \to \infty} \phi^n=\phi$ a.e. Since $\phi^n(t,\cdot)$ is s.t $\phi^n_{[0,T_0] \x \partial U_R}=0$ and $\widetilde{\alpha}:=\max(\alpha,\alpha')> d=2$, by Morrey's inequality,
   \begin{align*}
        \sup_{n \ge 1}\int_0^{T_0}\|\phi^n(t,\cdot)\|_{C^{0,1-2/\tilde a}(\R^d)}^{\tilde \alpha} \mathrm{d}t \le
        \sup_{n \ge 1} C(2,\tilde \alpha)^{\tilde \alpha} \|\phi^n\|^{\tilde \alpha}_{\H^{\tilde \alpha,1}([0,T_0]\x\R^d)}
        =
        \sup_{n \ge 1} C(2,\tilde \alpha)^{\tilde \alpha} \|\phi^n\|^{\tilde \alpha}_{\H^{\tilde \alpha,1}(U_{R,T_0})} < \infty.
    \end{align*}
    
    For any $\beta \in C^\infty([0,T_0])$, by applying Arzelà--Ascoli theorem together with the previous result, the sequence of functions $(\varphi^{\beta,n})_{n \ge 1} \subset C(\R^d)$ is relatively compact for the uniform topology where
    \begin{align*}
        \varphi^{\beta,n}(\xbb)
        :=
        \int_0^{T_0} \beta(t) \phi^{n}(t,\xbb)\mathrm{d}t.
    \end{align*}
    We know that $\lim_{n \to \infty} \phi^n=\phi$ a.e., then $\lim_{n \to \infty}\varphi^{\beta,n}(\cdot) = \int_{0}^{T_0} \beta(t)\phi(t,\cdot)\mathrm{d}t$ for the uniform topology. Consequently,
    \begin{align*}
        \lim_{n \to \infty}\|\wr'\;(\phi^n-\phi)\;p^n\|_{L^1(U_{R,T_0})}=0.
    \end{align*}
   Since $\ab$ and $\bb$ are bounded, $\ab$ is non--degenerate and $\ab(t,\xbb)$ is uniformly continuous in $\xbb$ uniformly in $t$, then the density $p$ is uniquely defined. We know that $\lim_{n \to \infty} (\ab^n,\bb^n)=(\ab,\bb)$ a.e., consequently $\lim_{n \to \infty} p^n=p$ in the weak sense. Therefore
   \begin{align*}
       \Liminf_{n \to \infty} \|\wr'\;\phi^n\;p^n\|_{L^1(U_{R,T_0})}
       =
       \Liminf_{n \to \infty} \|\wr'\;\phi\;p^n\|_{L^1(U_{R,T_0})}
       =
       \|\wr'\;\phi\;p\|_{L^1(U_{R,T_0})}.
   \end{align*}
   By \eqref{eq:estimates:appr} and \eqref{eq:estimates_phi_app},
   \begin{align*}
       &\int_{U_{R,T_0}}  \big[\partial_{x_i}f^i+g\big]v(t,\xbb)\wr(t) \mu_t(\mathrm{d} \xbb) \mathrm{d}t
       \le 
       \|\wr'\;\phi\;p\|_{L^1(U_{R,T_0})}
       \le 
       \|\phi\|_{L^{\alpha}(U_{R,T_0})}
       \|\wr'\;p\|_{L^{\alpha'}(U_{R,T_0})}
       \\
       &\le 
       \frac{N}{1-\ell(\pi,\alpha,R)\widehat{C}(\pi,\alpha,\ab,\bb,U_R,T_0)} \| (\partial_{x_i}f^i+g)v\|_{\H^{{\alpha},-1}( U_{R,T_0})}
       \|\wr'\;p\|_{L^{\alpha'}(U_{R,T_0})}.
   \end{align*}
   
   We obtain the second estimates by using the second inequality in \Cref{lemm:FP-estimates} and the fact that $\| \phi\|_{\H^{\alpha,1}(U_{R,T_0})} \le \Liminf_{n \to \infty} \| \phi^n\|_{\H^{\alpha,1}(U_{R,T_0})}$.

\end{proof}

We set $v \in L^\infty_{\ell oc}([0,T_0] \x \R^d)$ s.t. $\nabla v \in L^{\alpha \vee \alpha'}_{\ell oc} ([0,T_0] \x \R^d)$ and $\overline{R} \in \Rc(\pi,\alpha,\ab,\bb,T_0)$, and we define $K(\pi,\alpha,\ab,\bb,v,T_0,\overline{R}):=K(\pi,\alpha,\ab,\bb,v,T_0)$ by
\begin{align} \label{eq:const_K}
        &K(\pi,\alpha,\ab,\bb,v,T_0) \nonumber
        \\
        &:=
        \frac{2{N}\;\;\sup_{E \in \Sc}\bigg[ \esup_{(t,\xbb) \in [0,T_0] \x E}|v(t,\xbb)|+ \ell(\pi,\alpha,\overline{R}) \Big[\|\nabla v\|_{L^\alpha([0,T_0] \x E)} \mathbf{1}_{\alpha >2} + \|\nabla v\|_{L^{\alpha'}([0,T_0] \x E)} \mathbf{1}_{\alpha <2} \Big] \bigg]}{1-\big[\ell(\pi,\alpha,\overline{R}) + \ell(\pi,\alpha',\overline{R})\big] \widehat{C}_\infty(\pi,\alpha,\ab,\bb,T_0)}.
    \end{align}
    Let $R \le \overline{R}$, $\fb:=(f^i)_{1 \le i \le d} \subset C_c^\infty(U_{R,T_0})$ and $g \in C_c^\infty(U_{R,T_0})$. 

\begin{corollary}\label{corr:estimates_ineq}
    Under the consideration of {\rm \Cref{prop_estimate FP}}, one has
    \begin{align*}
        \bigg| \int_{U_{R,T_0}}  \big[\partial_{x_i}f^i+g\big]v(t,\xbb) \wr(t) \mu_t(\mathrm{d}\xbb) \mathrm{d}t \bigg|
        \le K(\pi,\alpha,\ab,\bb,v,T_0) \| (\fb,g)\|_{L^{\alpha}( U_{R,T_0})} \|\wr'\;p\|_{L^{\alpha'}(U_{R,T_0})} 
    \end{align*}
    and
    \begin{align*}
        \bigg| \int_{U_{R,T_0}}  \big[\partial_{x_i}f^i+g\big]v(t,\xbb) \wr(t) \mu_t(\mathrm{d}\xbb) \mathrm{d}t \bigg|
        \le K(\pi,\alpha,\ab,\bb,v,T_0) \| (\fb,g)\|_{L^{\alpha}( U_{R,T_0})} \int_{U^2_R}\|\wr'\;p(\cdot,\cdot,x_2)\|_{L^{\alpha'}(U^1_{R,T_0})}
        \mathrm{d}x_2.
    \end{align*}
    
\end{corollary}

\begin{proof}
    Similarly to the proof of {\rm \Cref{prop_estimate FP}}, by definition, since $(0,1) \ni r \to \ell(\pi,\alpha,r) \in \R_+$ is increasing, for any $R \le \overline{R}$, we have
    \begin{align*}
        \frac{1}{1-\ell(\pi,\alpha,{R})\widehat{C}_\infty(\pi,\alpha,\ab,\bb,T_0)} \le \frac{1}{1-\ell(\pi,\alpha,\overline{R})\widehat{C}_\infty(\pi,\alpha,\ab,\bb,T_0)}. 
    \end{align*}
    Then, by \Cref{prop_estimate FP}, we find 
    \begin{align*}
        &\bigg| \int_{U_{R,T_0}}  \big[\partial_{x_i}f^i+g\big]v(t,\xbb) \wr(t) \mu_t(\mathrm{d}\xbb) \mathrm{d}t \bigg|
        \le \frac{N}{1-\ell(\pi,\alpha,\overline{R})\widehat{C}_\infty(\pi,\alpha,\ab,\bb,T_0)} \| (\partial_{x_i}f^i+g)v\|_{\H^{\alpha,-1}( U_{R,T_0})} \|\wr'\;p\|_{L^{\alpha'}(U_{R,T_0})} 
    \end{align*}
    and
    \begin{align*}
        &\bigg| \int_{U_{R,T_0}}  \big[\partial_{x_i}f^i+g\big]v(t,\xbb) \wr(t) \mu_t(\mathrm{d}\xbb) \mathrm{d}t \bigg|
        \\
        &\le \frac{2{N}}{1-\ell(\pi,\alpha,\overline{R})\widehat{C}_\infty(\pi,\alpha,\ab,\bb,T_0)} \| (\partial_{x_i}f^i+g)v\|_{\H^{\alpha,-1}( U_{R,T_0})} \int_{U^2_R}\|\wr'\;p(\cdot,\cdot,x_2)\|_{L^{\alpha'}(U^1_{R,T_0})}
        \mathrm{d}x_2.
    \end{align*}
    Notice that 
    \begin{align*}
        \| v\;\partial_{x_i}f^i \|_{\H^{\alpha,-1}(U_{R,T_0})}
        &=
        \| \partial_{x_i}(f^i v)-f^i \partial_{x_i}v\|_{\H^{\alpha,-1}( U_{R,T_0})}
        \\
        &\le \bigg[ \sup_{(t,\xbb) \in [0,T_0] \x U_R}|v(t,\xbb)| + \ell(\pi,\alpha,R) \Big[\|\nabla v\|_{L^\alpha(U_{R,T_0})} \1_{\alpha >2} + \|\nabla v\|_{L^{\alpha'}(U_{R,T_0})} \1_{\alpha <2} \Big]  \bigg] \| f^i\|_{L^{\alpha}( U_{R,T_0})}.
    \end{align*}
    This is enough to conclude the proof.
\end{proof}

\medskip
With $d=2$, for any Borel map $(h,f):[0,T_0] \x \R^d \to \R^d$, whenever it is well--defined, we recall that the notation $(hf)_{\mathbf{1}}(t,x_1)$ means $(hf)_{\mathbf{1}}(t,x_1):=\int_\R h(t,x_1,x_2) f(t,x_1,x_2) \mathrm{d}x_2$. In particular, the notations $|\partial_{x_1}(v\;p)|_{\mathbf{1}}$ and  $|(v\;p)|_{\mathbf{1}}$ mean
$$
    |(v\;p)|_{\mathbf{1}}:=\int_\R |(v\;p)(t,x_1,x_2)| \mathrm{d}x_2
    \;\mbox{and}\;|\partial_{x_1}(v\;p)|_{\mathbf{1}}:=\int_\R |\partial_{x_1}(v\;p)(t,x_1,x_2)| \mathrm{d}x_2.
$$
We are now able to provide in the next proposition a relation between $(p,p_{\mathbf{1}})$ and its gradient.
\begin{proposition} \label{prop:inegal_general}
    Let $\pi \in (0,1)$, $\alpha \in (1,\infty)\setminus \{2\}$ and $\wr \in C_0^\infty(T_0)$.
    If $\Rc(\pi,\alpha,\ab,\bb,T_0)$ is non--empty,
    we have 
    \begin{align*}
         \|\wr\;  v\;p\|_{\H^{\alpha',1}([0,T_0] \x \R^2)}^{\alpha'} \le K(\pi,\alpha,\ab,\bb,v,T_0)^{\alpha'} 2^d \|\wr'\;p\|_{L^{\alpha'}([0,T_0] \x \R^2)}^{\alpha'}
    \end{align*}
    and
    \begin{align*}
    \|\wr\;  (v\;p)_{\mathbf{1}}\|_{\H^{\alpha',1}([0,T_0] \x \R)}&\le
    \Big[\|\wr|\partial_{x_1}(v\;p)|_{\mathbf{1}}\|_{L^{\alpha'}([0,T_0]\x\R)}^{\alpha'}+ \|\wr|(v\;p)|_{\mathbf{1}}\|_{L^{\alpha'}([0,T_0] \x \R)}^{\alpha'}  \Big]^{1/\alpha'}
    \\
        &\le K(\pi,\alpha,\ab,\bb,v,T_0) 2^d \int_{\R}\|\wr'\;p(\cdot,\cdot,x_2)\|_{L^{\alpha'}([0,T_0] \x \R)}
        \mathrm{d}x_2.
    \end{align*}
    
\end{proposition}

\begin{remark}
    Among other things, the previous result shows the passage of local estimates to global estimates. The existence of a small diameter $\overline{R} \in \Rc(\pi,\alpha,\ab,\bb,T_0)$ allows to propagate the estimates over all the space. It is possible in particular because $K(\pi,\alpha,\ab,\bb,v,T_0)$ is independent of the choice of the ball $U_{\overline{R}}$. It only depends on the diameter $\overline{R}$. The techniques used in the proof below is similar to those used in the proof of {\rm \cite[Theorem 3.2.2]{FK-PL-equations}}.
\end{remark}

\begin{proof}

Let $0< r < 1$ s.t. $2r \le \overline{R}$. Let $(\widetilde{E}_k)_{k \ge 1}$ be a sequence of disjoint sets of diameter $r$ s.t.  $\cup_{k \ge 1} \widetilde{E}_k=\R$. Also, we consider a sequence of open sets of diameter $2r$, $(\widetilde{O}_k)_{k \ge 1}$, s.t. $\widetilde{E}_k \subset \widetilde{O}_k$. We take $(\widetilde{\varphi}^k)_{k \ge 1}$ satisfying $\widetilde{\varphi}_k \in C_c(\widetilde{O}_k)$, $0 \le \widetilde{\varphi}_k \le 1$ and $\widetilde{\varphi}_k(e)=1$ for $e \in \widetilde{E}_k$. Let $E_{k,q}:=\widetilde{E}_k \x \widetilde{E}_q$, $O_{k,q}:=\widetilde{O}_k \x \widetilde{O}_q$ and $\varphi_{k,q}(x_1,x_2):=\widetilde{\varphi}_k(x_1) \widetilde{\varphi}_q(x_2)$. 
By \Cref{corr:estimates_ineq}, for $f^i:=\widetilde{f}^i\varphi_{k,q}$ and $g:=\widetilde{g}\varphi_{k,q}$. 
    \begin{align*}
        \bigg| \int_{[0,T_0] \x O_{k,q}}  \big[\partial_{x_i}(\widetilde{f}^i\;\varphi_{k,q}) + \widetilde{g}\;\varphi_{k,q} \big] v(t,\xbb) \wr(t) \mu_t(\mathrm{d}\xbb) \mathrm{d}t \bigg|
        \le K(\pi,\alpha,\ab,\bb,v,T_0) \| (\widetilde{\fb},\widetilde{g})\|_{L^{\alpha}([0,T_0]\x O_{k,q})} \|\wr'\;p\|_{L^{\alpha'}([0,T_0] \x O_{k,q})} 
    \end{align*}
    and
    \begin{align*}
        &\bigg| \int_{[0,T_0] \x O_{k,q}}  \big[\partial_{x_i}(\widetilde{f}^i\;\varphi_{k,q}) + \widetilde{g}\;\varphi_{k,q} \big] v(t,\xbb) \wr(t) \mu_t(\mathrm{d}\xbb) \mathrm{d}t \bigg|
        \\
        &\le K(\pi,\alpha,\ab,\bb,v,T_0) \| (\widetilde{\fb},\widetilde{g})\|_{L^{\alpha}([0,T_0]\x O_{k,q})} \int_{\tilde{O}_q}\|\wr'\;p(\cdot,\cdot,x_2)\|_{L^{\alpha'}([0,T_0] \x \tilde{O}_k)}
        \mathrm{d}x_2.
    \end{align*}
    
    If $\|\wr'\;p\|_{L^{\alpha'}([0,T_0] \x O_{k,q})} < \infty$ or $\int_{\widetilde{O}_q}\|\wr'\;p(\cdot,\cdot,x_2)\|_{L^{\alpha'}([0,T_0] \x \widetilde{O}_k)}
        \mathrm{d}x_2< \infty$, the previous inequality shows that $\wr\; \partial_{x_i}(vp)$ is well defined on $[0,T_0] \x \widetilde{O}_{k,q}$ (see for instance point $(ii)$ of \cite[Section 5.8 Theorem 3]{evans2010partial}). Then, by integration by part, for any $\widetilde{f}^i$ and $\widetilde{g}$
    \begin{align*}
        \bigg| \int_{[0,T_0] \x O_{k,q}}  \big[\widetilde{f}^i\partial_{x_i}(v p) +\widetilde{g}\;(vp) \big]\;\varphi_{k,q}(t,\xbb) \wr(t)\; \mathrm{d}\xbb\; \mathrm{d}t \bigg|
        \le K(\pi,\alpha,\ab,\bb,v,T_0) \| (\widetilde{\fb},\widetilde{g})\|_{L^{\alpha}([0,T_0] \x O_{k,q})} \|\wr'\;p\|_{L^{\alpha'}([0,T_0] \x O_{k,q})}
    \end{align*}
    and
    \begin{align*}
        &\bigg| \int_{[0,T_0] \x O_{k,q}}  \big[\widetilde{f}^i\partial_{x_i}(v p) +\widetilde{g}\;(vp) \big]\;\varphi_{k,q}(t,\xbb) \wr(t)\; \mathrm{d}\xbb\; \mathrm{d}t \bigg|
        \\
        &\le K(\pi,\alpha,\ab,\bb,v,T_0) \| (\widetilde{\fb},\widetilde{g})\|_{L^{\alpha}([0,T_0] \x O_{k,q})} \int_{\widetilde{O}_q}\|\wr'\;p(\cdot,\cdot,x_2)\|_{L^{\alpha'}([0,T_0] \x \widetilde{O}_k)}
        \mathrm{d}x_2.
    \end{align*}
    Notice that, even when the quantities  $\|\wr'\;p\|_{L^{\alpha'}([0,T_0] \x O_{k,q})}$ or $\int_{\widetilde{O}_q}\|\wr'\;p(\cdot,\cdot,x_2)\|_{L^{\alpha'}([0,T_0] \x \widetilde{O}_k)}
        \mathrm{d}x_2$ are not finite, the previous inequalities are still true. Indeed, in this case all upper bounds are infinite. Besides, the inequalities are true for $\widetilde{f}^i$ belonging to $L^{\alpha}([0,T_0] \x O_{k,q})$.

\medskip
We start by dealing with the first inequality. In the first inequality, by taking the supremum over $\|( \widetilde{\fb},\widetilde{g})\|_{L^{\alpha}([0,T_0] \x O_{k,q})} \le 1$, we find
    \begin{align*}
        \|\varphi_{k,q} \wr\; \nabla (v\;p)\|_{L^{\alpha'}([0,T_0] \x O_{k,q})}^{\alpha'}
        +
        \|\varphi_{k,q} \wr\; (v\;p)\|_{L^{\alpha'}([0,T_0] \x O_{k,q})}^{\alpha'}
        \le K(\pi,\alpha,\ab,\bb,v,T_0)^{\alpha'} \|\wr'\;p\|_{L^{\alpha'}([0,T_0] \x O_{k,q})}^{\alpha'}.
    \end{align*}
    Therefore, as the constant $K(\pi,\alpha,\ab,\bb,v,T_0)$ is independent of $(k,q)$, $\widetilde{\varphi}_k(e)=1$ for $e \in \widetilde{E}_k$ and $0 \le \widetilde{\varphi}_k \le 1$, by taking the summation, we get 
    \begin{align*}
        \|\wr\;  v\;p\|_{\H^{\alpha',1}([0,T_0] \x \R^d)}^{\alpha'}=&\|\wr\; \nabla (v\;p)\|_{L^{\alpha'}([0,T_0] \x \R^d)}^{\alpha'}
        +
        \|\wr\; v\;p\|_{L^{\alpha'}([0,T_0] \x \R^d)}^{\alpha'}
        \\
        &\le
        \sum_{k,q \ge 1}\|\varphi_{k,q} \wr\; \nabla (v\;p)\|_{L^{\alpha'}([0,T_0] \x O_{k,q})}^{\alpha'}
        +
        \|\varphi_{k,q} \wr\; v\;p\|_{L^{\alpha'}([0,T_0] \x O_{k,q})}^{\alpha'}
        \\
        &\le K(\pi,\alpha,\ab,\bb,v,T_0)^{\alpha'} \sum_{k,q \ge 1} \|\wr'\;p\|_{L^{\alpha'}([0,T_0] \x O_{k,q})}^{\alpha'}
        \le K(\pi,\alpha,\ab,\bb,v,T_0)^{\alpha'} 2^d \|\wr'\;p\|_{L^{\alpha'}([0,T_0] \x \R^d)}^{\alpha'}.
    \end{align*}
    
    Now, we consider the second inequality. We take $\widetilde{f}^2=0$,
    $$
        \widetilde{f}^1(t,x_1,x_2):=|\beta^1(t,x_1)| [\1_{\wr\;\partial_{x_1}(v\;p)(t,x_1,x_2) \ge 0} - \1_{\wr\;\partial_{x_1}(v\;p)(t,x_1,x_2) \le 0}]\1_{(t,x_1,x_2) \in[0,T_0] \x O_{k,q}}
    $$
    and
    $$
        \widetilde{g}(t,x_1,x_2):=|\beta^0(t,x_1)| [\1_{\wr\;(v\;p)(t,x_1,x_2) \ge 0} - \1_{\wr\;(v\;p)(t,x_1,x_2) \le 0}]\1_{(t,x_1,x_2) \in[0,T_0] \x O_{k,q}}.
    $$
    We find
    \begin{align*}
        &
        \bigg| \int_{[0,T_0] \x \R^d}   |\wr(t)| \big[\;|\beta^1(t,x_1)||\partial_{x_1}(v p)(t,x_1,x_2)| + |\beta^0(t,x_1)||v p (t,x_1,x_2)| \big]\;\mathrm{d}x_1\mathrm{d}x_2\; \mathrm{d}t \bigg|
        \\
        &\le \sum_{k,q \ge 1}\bigg| \int_{[0,T_0] \x O_{k,q}}  \varphi_{k,q}(x_1,x_2) |\wr(t)| \big[\;|\beta^1(t,x_1)||\partial_{x_1}(v p)(t,x_1,x_2)| + |\beta^0(t,x_1)||v p (t,x_1,x_2)| \big]\;\mathrm{d}x_1\mathrm{d}x_2\; \mathrm{d}t \bigg|
        \\
        &\le K(\pi,\alpha,\ab,\bb,v,g,T_0) \sum_{k \ge 1} \| (\beta^1,\beta^0)\|_{L^{\alpha}([0,T_0] \x \widetilde{O}_{k})} \sum_{q \ge 1}\int_{\widetilde{O}_q}\|\wr'\;p(\cdot,\cdot,x_2)\|_{L^{\alpha'}([0,T_0] \x \widetilde{O}_k)}
        \mathrm{d}x_2
        \\
        &\le K(\pi,\alpha,a,\bb,v,g,T_0) \sum_{k \ge 1} \| (\beta^1,\beta^0)\|_{L^{\alpha}([0,T_0] \x \widetilde{O}_{k})} 2\;\int_{\R}\|\wr'\;p(\cdot,\cdot,x_2)\|_{L^{\alpha'}([0,T_0] \x \widetilde{O}_k)}
        \mathrm{d}x_2
        \\
        &\le K(\pi,\alpha,\ab,\bb,v,T_0) 2^d \| (\beta^1,\beta^0)\|_{L^{\alpha}([0,T_0] \x \R)} \int_{\R}\|\wr'\;p(\cdot,\cdot,x_2)\|_{L^{\alpha'}([0,T_0] \x \R)}
        \mathrm{d}x_2.
    \end{align*}
    We take the supremum over $\| (\beta^1,\beta^0)\|_{L^{\alpha}([0,T_0] \x \R)} \le 1$ and find the result. It is worth mentioning that we implicitly suppose that $\partial_{x_1}(v\;p)$ is well defined. Indeed, as mentioning earlier, when it is not the case, all the upper bonds are infinite.

\end{proof}
In the proposition just below, we provide a recursive estimate making a relation between the $L^z$ norm of $p$ and the $L^{z+1}$ norm of p. Its usefulness will be clear in \Cref{sec:invariantset}.
\begin{proposition} \label{prop:recursive}
    For any $z>2$ and $\wr \in C_0^\infty(T_0)$,  if $\Rc(\pi,z,\ab,\bb,T_0)$ is non--empty,
    we have
    \begin{align*}
        \|\wr\;p\|_{L^{z+1}([0,T_0] \x \R^d)}^{z+1}
        \le 
        \|\wr\|_{\infty} C(2,z)^z K(\pi,z',\ab,\bb,\1,T_0)^z 2^d \|\wr'\;p\|_{L^z([0,T_0] \x \R^d)}^z
    \end{align*}
    and
    \begin{align*}
        \int_0^{T_0} \|\wr(t)\;p(t,\cdot)\|^z_{C^{0,1-2/z}(\R^d)} \mathrm{d}t
        \le
        C(2,z)^z K(\pi,z',\ab,\bb,\1,T_0)^z 2^d \|\wr'\;p\|_{L^z([0,T_0] \x \R^d)}^z
    \end{align*}
    where $C(2,z)$ is a constant appearing in the Sobolev embedding Theorem depending on $d=2$ and $z$.
\end{proposition}

\begin{proof}
    Since $z > d=2,$ we use \Cref{prop:second_sobolev} and find for any $z < r \le  z+1$
        \begin{align*}
           \int_0^{T_0}\|\wr(t)\;p(t,\cdot)\|_{C^{0,1-2/z}(\R^d)}^z \mathrm{d}t \le C(2,z)^z \|\wr\;p\|^z_{\H^{z,1}([0,T_0] \x \R^d)}
        \end{align*}
        and
        \begin{align*}
            \|\wr\;p\|^r_{L^{r}([0,T_0] \x \R^d)} \le G(\wr)^{1/s'} C(d,z)^{z/s'}\|\wr\;p\|^{z/s}_{L^z([0,T_0] \x \R^d)} \|\wr\;p\|_{\H^{z,1}([0,T_0] \x \R^d)}^{z/s'}
        \end{align*}
        where $s:=\frac{1}{z+1-r}$. We can conclude by taking $r=z+1$ and using \Cref{prop:inegal_general}.
\end{proof}
The next result gives a more explicit estimate for the gradient of $p_{\mathbf{1}}$. Also, in the same spirit as the previous proposition, it provides a relation between the $L^{s+1}$ norm of $(v\;p)_{\mathbf{1}}$ and the $\H^{s,1}$ norm of $(v\;p)_{\mathbf{1}}$.
\begin{proposition} \label{prop:estimates_marg}
    Let $r >1$ and $1< s < r +1$ s.t. $s \neq 2$. For any $\wr \in C^\infty_0(T_0)$, if $\Rc(\pi,s,\ab,\bb,T_0)$ is non--empty,
    we have the estimate
    \begin{align*}
        & \Big[\|\wr|\partial_{x_1}(v\;p)|_{\mathbf{1}}\|_{L^{s}([0,T_0] \x \R)}^s+ \|\wr\;|(v\;p)|_{\mathbf{1}}\|_{L^{s}([0,T_0] \x \R)}^s  \Big]^{1/s}
        \\
        &\le
        K(\pi,s,\ab,\bb,v,T_0) 2^d\bigg[\int_0^{T_0} \|\wr'(t)p(t,\cdot)\|_\infty^r \mathrm{d}t \bigg]^{\frac{s-1}{sr}} \int_\R \bigg[\int_0^{T_0}\Big|\int_\R |\wr'(t)p(t,x_1,x_2)| \mathrm{d}x_1\Big|^{u'} \mathrm{d}t \bigg]^{\frac{1}{su'}} \mathrm{d}x_2
    \end{align*}
    where $u'=\frac{r}{r+1-s}$. In addition, 
    \begin{align*}
        \|\wr\;(v\;p)_{\mathbf{1}}\|^{s+1}_{L^{s+1}([0,T_0] \x\R)}
        \le\|\wr\|_{\infty} C(1,s)^{s} \|\wr\;(v\;p)_{\mathbf{1}}\|^{s}_{\H^{s,1}([0,T_0] \x \R)}
    \end{align*}
     where $C(1,s)$ is a constant appearing in the Sobolev embedding Theorem depending on $1$ and $s$.
\end{proposition}

\begin{proof}
    We observe that, for $s< r+1$ and $u >1$,
    \begin{align*}
        &\int_\R \bigg[ \int_0^{T_0} \int_\R |\wr'(t)p(t,x_1,x_2)|^s \mathrm{d}x_1\;\mathrm{d}t \bigg]^{1/s} \mathrm{d}x_2
        =
        \int_\R \bigg[ \int_0^{T_0} \int_\R |\wr'(t)p(t,x_1,x_2)|^{s-1} |\wr'(t)p(t,x_1,x_2)| \mathrm{d}x_1\;\mathrm{d}t \bigg]^{1/s} \mathrm{d}x_2
        \\
        &\le
        \int_\R \bigg[ \int_0^{T_0} \int_\R \sup_{\xbb' \in \R^d}|\wr'(t)p(t,\xbb')|^{s-1} |\wr'(t)p(t,x_1,x_2)| \mathrm{d}x_1\;\mathrm{d}t \bigg]^{1/s} \mathrm{d}x_2
        \\
        &\le \bigg[\int_0^{T_0} \sup_{\xbb' \in \R^2} |\wr'(t)p(t,\xbb')|^{(s-1)u} \mathrm{d}t \bigg]^{\frac{1}{su}} \int_\R \bigg[\int_0^{T_0}\Big|\int_\R |\wr'(t)p(t,x_1,x_2)| \mathrm{d}x_1\Big|^{u'} \mathrm{d}t \bigg]^{\frac{1}{su'}} \mathrm{d}x_2.
    \end{align*}
    We take $(s-1)u=r$. Then, we have $u'=\frac{r}{r+1-s}$ and 
    \begin{align*}
        &\int_\R \bigg[ \int_0^{T_0} \int_\R |\wr'(t)p(t,x_1,x_2)|^s \mathrm{d}x_1\;\mathrm{d}t \bigg]^{1/s} \mathrm{d}x_2
        \\
        &\le 
        \bigg[\int_0^{T_0} \sup_{\xbb' \in \R^2} |\wr'(t)p(t,\xbb')|^{r} \mathrm{d}t \bigg]^{\frac{s-1}{sr}} \int_\R \bigg[\int_0^{T_0}\Big|\int_\R |\wr'(t)p(t,x_1,x_2)| \mathrm{d}x_1\Big|^{u'} \mathrm{d}t \bigg]^{\frac{1}{su'}} \mathrm{d}x_2.
    \end{align*}
    Consequently, as $s \neq 2$ and $\Rc(\pi,s,\ab,\bb,T_0)$ is non--empty,
    \begin{align*}
        &\bigg[ \int_{[0,T_0] \x \R} \wr(t) \bigg|\int_\R |\partial_{x_1}(v p)(t,x_1,x_2)| \mathrm{d}x_2 \bigg|^{s} \mathrm{d}x_1\;\mathrm{d}t + \int_{[0,T_0] \x \R} \wr(t) \bigg|\int_\R |(v p)(t,x_1,x_2)| \mathrm{d}x_2 \bigg|^{s} \mathrm{d}x_1\;\mathrm{d}t \bigg]^{1/s}
        \\
        &\le K(\pi,s,\ab,\bb,v,T_0) 2^d \int_{\R}\bigg[\int_{[0,T_0] \x \R} |\wr'(t)p(t,x_1,x_2)|^{s} \mathrm{d}x_1\; \mathrm{d}t \bigg]^{1/s}
        \mathrm{d}x_2
        \\
        &\le K(\pi,s,\ab,\bb,v,T_0) 2^d\bigg[\int_0^{T_0} \sup_{\xbb' \in \R^2} |\wr'(t)p(t,\xbb')|^{r} \mathrm{d}t \bigg]^{\frac{s-1}{sr}} \int_\R \bigg[\int_0^{T_0}\Big|\int_\R |\wr'(t)p(t,x_1,x_2)| \mathrm{d}x_1\Big|^{u'} \mathrm{d}t \bigg]^{\frac{1}{su'}} \mathrm{d}x_2.
    \end{align*}
    The second inequality follows from \Cref{prop:second_sobolev}.
\end{proof}

Now, we are going to establish some estimates in H\"older norm.
Let $r > 2$ and $\frac{1}{2} > \eta > \kappa > \frac{1}{r}$ with $(1-2\eta) r > 2$.
\begin{proposition} \label{prop:holder_estimates_p}
    We assume that $\Rc(\pi,r,\ab,\bb,T_0)$ is non--empty then
    \begin{align*}
        &\|\; \wr\;p\; \|_{C^{0,\kappa-1/r}\big([0,T_0];C^{0,1-2\eta-2/r}(\R^d)\big)}
        \\
        &\le D \bigg[  \Big[\; \big[K(\pi,r',\ab,\bb,\1,T_0)+K(\pi,r', \ab,\bb,\ab,T_0) \big]\; 2^{d/r}
         + 1 \Big]
         \| \wr'\; p\|_{L^{r}([0,T_0] \x \R^2)}
         \\
         &~~~~~~~~~~~~~~~~~~~~~~~~~~~~~~~~~~~~~~~~+\|\bb\|_{\infty}\| \wr\; p\|_{L^{r}([0,T_0] \x \R^2)}+ \| \wr(0)\;p(0,\cdot) \|_{H^{r,1}(\R^2)} \bigg]
    \end{align*}
    where $D$ is the constant used in {\rm \Cref{prop:estimates_holder} } depending only on $(\eta,\kappa,r,d=2,T_0).$ 
\end{proposition}

\begin{proof}
    In a weak sense, we can write, $\partial_t (\wr\;p)(t,\xbb)= \wr'(t)\;p(t,\xbb)- \partial_{x_i}(b^i \wr\;p)(t,\xbb) + \partial_{x_i} \partial_{x_j} (a^{i,j} \wr\;p) (t,\xbb).$ Then by applying \Cref{prop:estimates_holder}, we find 
    \begin{align*}
        &\|\; \wr\;p\; \|_{C^{0,\kappa-1/r}\big([0,T_0],C^{0,1-2\eta-2/r}(\R^d)\big)}
        \\
        &\le D \Big[ \|\wr\;\nabla p\|_{L^{r}([0,T_0] \x \R^d)} + \| \wr'\;p- \partial_{x_i}(b^i \wr\;p) + \partial_{x_i} \partial_{x_j} (a^{i,j} \wr\;p)\|_{\H^{r,-1}([0,T_0] \x \R^d)} + \| \wr(0)\;p(0,\cdot) \|_{H^{r,1}(\R^d)} \Big].
    \end{align*}
    By \Cref{prop:inegal_general},
    \begin{align*}
         \| \partial_{x_i} \partial_{x_j} (a^{i,j} \wr\;p)\|_{\H^{r,-1}([0,T_0] \x \R^d)}
         &\le 
         \| \nabla (\ab\;\wr\; p)\|_{L^{r}([0,T_0] \x \R^d)}
         \\
         &\le K(\pi,r',\ab,\bb,\ab,T_0) 2^{d/r} \| \wr'\; p\|_{L^{r}([0,T_0] \x \R^d)}.
    \end{align*}
    Also,
    \begin{align*}
         &\|\wr\;\nabla p\|_{L^{r}([0,T_0] \x \R^d)}
         +
         \| \wr'\;p\|_{\H^{r,-1}([0,T_0] \x \R^d)}
         +
         \| \partial_{x_i}(b^i\;\wr\; p)\|_{\H^{r,-1}([0,T_0] \x \R^d)}
         \\
         &\le 
         \|\wr\;\nabla p\|_{L^{r}([0,T_0] \x \R^d)}
         +
         \| \wr'\;p\|_{L^{r}([0,T_0] \x \R^d)}
         +
         \| b^i\;\wr\; p\|_{L^{r}([0,T_0] \x \R^d)}
         \\
         &\le
          K(\pi,r',\ab,\bb,\1,T_0) 2^{d/r} \| \wr'\; p\|_{L^{r}([0,T_0] \x \R^d)} +  \| \wr'\;p\|_{L^{r}([0,T_0] \x \R^d)} + \|\bb\|_{\infty} \| \;\wr\; p\|_{L^{r}([0,T_0] \x \R^d)}.
    \end{align*}
\end{proof}

\medskip
Let us observe that, for any $\phi \in C^\infty_c(\R)$, by It\^o's formula,
\begin{align*}
    \mathrm{d} \int_{\R^2} \phi(x_1) \wr(t) p(t,x_1,x_2) \mathrm{d}x_2 \mathrm{d}x_1
    &=
    \bigg[\int_{\R^2} \phi(x_1) \wr'(t) p(t,x_1,x_2) \mathrm{d}x_2 \mathrm{d}x_1 +\int_{\R^2}\phi'(x_1)\wr(t)\;b^1(t,x_1,x_2)p(t,x_1,x_2) \mathrm{d}x_2\mathrm{d}x_1
    \\
    &\;\;\;\;\;\;+\int_{\R^2}\phi''(x_1)\wr(t)a^{1,1}(t,x_1,x_2)p(t,x_1,x_2) \mathrm{d}x_2\mathrm{d}x_1\bigg]\mathrm{d}t.
\end{align*}
For any bounded measurable map  $f:[0,T_0]\x \R^2 \to \R$, we recall that $(pf)_{\mathbf{1}}(t,x_1):=\int_{\R} f(t,x_1,x_2)p(t,x_1,x_2)\mathrm{d}x_2$. We rewrite
\begin{align*}
    \mathrm{d} \int_{\R} \phi(x_1) (p\wr)_{\mathbf{1}}(t,x_1) \mathrm{d}x_1
    =
    \bigg[ \int_{\R} \phi(x_1) (p\wr')_{\mathbf{1}}(t,x_1) \mathrm{d}x_1 + \int_{\R}\phi'(x_1)(p\wr \;b^1)_{\mathbf{1}}(t,x_1) \mathrm{d}x_1
    +\int_{\R}\phi''(x_1)(p\wr \;a^{1,1})_{\mathbf{1}}(t,x_1)\mathrm{d}x_1 \bigg] \mathrm{d}t.
\end{align*}

Let $\overline{r} > 2$ and $\frac{1}{2} > \overline{\eta} > \overline{\kappa} > \frac{1}{\overline{r}}$ with $(1-2\overline{\eta}) \overline{r}  > 1$. 
\begin{proposition} \label{prop:holder_estimates_p1}
    We have the estimate
     \begin{align*}
        \|\wr p_\mathbf{1}\|_{C^{0,\overline{\kappa}-1/\overline{r} }\big([0,T_0],C^{0,1-2\overline{\eta}-1/\overline{r} }(\R)\big)}
        \le \overline{D} \bigg[ \|\bb\|_\infty \Big{\|} \big[ \wr p_{\mathbf{1}}, \wr'p_{\mathbf{1}}, \wr|\partial_{x_1}p|_{\mathbf{1}},\wr |\partial_{x_1}(p\;a^{1,1})|_{\mathbf{1}} \big] \Big{\|}_{L^{\overline{r} }([0,T_0] \x \R)}
        + \|\wr p_{\mathbf{1}}(0,\cdot)\|_{H^{\overline{r} ,1}(\R)}  \bigg]
    \end{align*}
    where $\overline{D}$ is the constant used in {\rm \Cref{prop:estimates_holder} } depending only on $(\overline{\eta},\overline{\kappa},\overline{r},d=1,T_0).$ 
\end{proposition}
\begin{remark}
    Notice that, by {\rm \Cref{prop:estimates_marg}}, if $\Rc(\pi,\overline{r},\ab,\bb,T_0)$ is non--empty,
    \begin{align*}
        &
         \| \wr |\partial_{x_1}(p\;a^{1,1})|_{\mathbf{1}} \|_{L^{\overline{r} }([0,T_0] \x \R)}
         \\
         &\le K(\pi,\overline{r},\ab,\bb,a^{1,1},T_0) 2^d\bigg[\int_0^{T_0} \|\wr'(t)p(t,\cdot)\|_\infty^{\overline{r}} \mathrm{d}t \bigg]^{\frac{\overline{r} -1}{\overline{r}^2}} \int_\R \bigg[\int_0^{T_0}\Big|\int_\R |\wr'(t)p(t,x_1,x_2)| \mathrm{d}x_1\Big|^{\overline{r} } \mathrm{d}t \bigg]^{\frac{1}{\overline{r}^2}} \mathrm{d}x_2.
    \end{align*}
\end{remark}
\begin{proof}
     By \Cref{prop:estimates_holder},
    \begin{align*}
        &\|\wr p_\mathbf{1}\|_{C^{0,\overline{\kappa}-1/\overline{r} }\big([0,T_0],C^{0,1-2\overline{\eta}-1/\overline{r} }(\R)\big)}
        \\
        &\le \overline{D} \bigg[ \|\wr p_\mathbf{1}'\|_{L^{\overline{r} }([0,T_0] \x \R)}
        + \|\wr'p_{\mathbf{1}} + \wr (p \;b^1)_{\mathbf{1}}' + \wr (p\;a^{1,1})_{\mathbf{1}}'' \|_{\H^{\overline{r} ,-1}([0,T_0] \x \R)} 
        + \|\wr p_{\mathbf{1}}(0,\cdot)\|_{H^{\overline{r} ,1}(\R)} \bigg].
    \end{align*}
    Since $\| \wr (p\;a^{1,1})_{\mathbf{1}}'' \|_{\H^{\overline{r},-1}([0,T_0] \x \R)}
        \le 
        \| \wr (p\;a^{1,1})_{\mathbf{1}}' \|_{L^{\overline{r}}([0,T_0] \x \R)}\le
         \| \wr |\partial_{x_1}(p\;a^{1,1})|_{\mathbf{1}} \|_{L^{\overline{r} }([0,T_0] \x \R)}$, we obtain
    \begin{align*}
        &\|\wr p_{\mathbf{1}}\|_{C^{0,\overline{\kappa}-1/\overline{r} }\big([0,T_0],C^{0,1-2\overline{\eta}-1/\overline{r} }(\R)\big)}
        \\
        &\le \overline{D} \bigg[ \|\bb\|_\infty \Big{\|} \big[ \wr p_{\mathbf{1}}, \wr'p_{\mathbf{1}}, \wr|\partial_{x_1}p|_{\mathbf{1}},\wr |\partial_{x_1}(p\;a^{1,1})|_{\mathbf{1}} \big] \Big{\|}_{L^{\overline{r} }([0,T_0] \x \R)}
        + \|\wr p_{\mathbf{1}}(0,\cdot)\|_{H^{\overline{r},1}(\R)}  \bigg].
    \end{align*}
\end{proof}
Let us define $\gamma:=(1+1/2)\frac{2+1}{2}=2+1/4$. Contrary to the previous last results where we showed estimates putting in relation $p$ and its gradient, in the next proposition we give estimate of the $L^\gamma$ norm of $p$ involving only the coefficients of the SDE.
\begin{proposition} \label{prop:estimates-gamma}
    Let $\pi \in (0,1)$ and $\wr \in C_0^\infty(T_0)$. If $\Rc(\pi,3,\ab,\bb,T_0)$ is non--empty, we have
    \begin{align*}
        \|\wr\;p\|_{L^\gamma([0,T_0] \x \R^2)}^\gamma \le C(2,1+1/2)\|(\wr,\wr')\|_{\infty}^{1+1/2} K(\pi,3,\ab,\bb,\1,T_0)^{1+1/2}2^d G(T_0,1+1/2,2,m,M) 
    \end{align*}
    where $G(T_0,1+1/2,2,m,M)$ is the constant used in {\rm \Cref{prop:density-integrability}} and $C(2,1+1/2)$ is the constant appearing in Sobolev embedding theorem.
\end{proposition}

\begin{proof}
    Since $1+1/2 < d=2,$ by \Cref{prop:first_sobolev},
    \begin{align*}
            \|\wr\;p\|^\gamma_{L^{\gamma}([0,T_0] \x \R^2)} \le C(2,1+1/2)^{1+1/2} G(\wr)^{1/2+1/4} \|\wr\;\nabla p\|_{L^{1+1/2}([0,T_0] \x \R^2)}^{1+1/2}
        \end{align*}
        where $G(\wr):=\|\wr\|_{\infty}$, $\gamma=(1+1/2) \frac{2+1}{2}=2+1/4$ and $C(2,1+1/2)$ is a positive constant depending only on $1+1/2$ and $2$.
    As $\Rc(\pi,3,\ab,\bb,T_0)$ is non--empty, by applying \cref{prop:inegal_general} and then {\rm \Cref{prop:density-integrability}}, we obtain
    \begin{align*}
        \|\wr\;p\|_{L^{\gamma}([0,T_0] \x \R^2)}^\gamma 
        &\le 
        \|\wr\|^{1+1/2}_\infty \big(C(2,1+1/2)\;K(\pi,3,\ab,\bb,\1,T_0) \big)^{1+1/2} 2^d \|\wr'\;p\|_{L^{1+1/2}([0,T_0] \x \R^2)}^{1+1/2}
        \\
        &\le
        \|(\wr',\wr)\|^{1+1/2}_\infty \big(C(2,1+1/2)\;K(\pi,3,\ab,\bb,\1,T_0) \big)^{1+1/2} 2^d G(T_0,1+1/2,2,m,M).
    \end{align*}
\end{proof}

\paragraph*{Summary of the estimates} 
In order to get a general idea of all the estimates we have proven so far, we summarize them here before using them in the next section.

\begin{itemize}
    \item In \Cref{prop:estimates-gamma}, for $\gamma:=2+1/4$,
    \begin{align*}
        \|\wr\;p\|_{L^\gamma([0,T_0] \x \R^d)}^\gamma \le C(2,1+1/2)\|(\wr,\wr')\|_{\infty}^{(1+1/2)/2} K(\pi,3,\ab,\bb,\1,T_0)^{1+1/2}2^d G(T_0,1+1/2,2,m,M) ;
    \end{align*}

    \item In \Cref{prop:recursive}, for any $z > 2$,
    \begin{align*}
        \|\wr\;p\|_{L^{z+1}([0,T_0] \x \R^d)}^{z+1}
        \le 
        \|\wr\|_{\infty} C(2,z)^z K(\pi,z',\ab,\bb,\1,T_0)^z 2^d \|\wr'\;p\|_{L^z([0,T_0] \x \R^d)}^z
    \end{align*}
    and
    \begin{align*}
        \int_0^{T_0} \|\wr(t)\;p(t,\cdot)\|^z_{C^{0,1-2/z}(\R^d)} \mathrm{d}t
        \le
        C(2,z)^z K(\pi,z',\ab,\bb,\1,T_0)^z 2^d \|\wr'\;p\|_{L^z([0,T_0] \x \R^d)}^z;
    \end{align*}
    
    \item In \Cref{prop:inegal_general}, for $\alpha >1$, 
    \begin{align*}
        \|\wr\;  v\;p\|_{\H^{\alpha',1}([0,T_0] \x \R^d)}^{\alpha'} \le K(\pi,\alpha,a,b,v,T_0)^{\alpha'} 2^d \|\wr'\;p\|_{L^{\alpha'}([0,T_0] \x \R^d)}^{\alpha'};
    \end{align*}
    
    \item In \Cref{prop:estimates_marg}, for any $r >1$ and $1< s < r +1$ s.t. $r \neq 2$,
    \begin{align*}
        &\Big[\|\wr|\partial_{x_1}(v\;p)|_{\mathbf{1}}\|_{L^{s}([0,T_0] \x \R)}^s+ \|\wr\;|v\;p|_{\mathbf{1}}\|_{L^{s}([0,T_0] \x \R)}^s  \Big]^{1/s}
        \\
        &\le
        K(\pi,s,\ab,\bb,v,T_0) 2^d\bigg[\int_0^{T_0} \|\wr'(t)p(t,\cdot)\|_\infty^r \mathrm{d}t \bigg]^{\frac{s-1}{sr}} \int_\R \bigg[\int_0^{T_0}\Big|\int_\R |\wr'(t)p(t,x_1,x_2)| \mathrm{d}x_1\Big|^{u'} \mathrm{d}t \bigg]^{\frac{1}{su'}} \mathrm{d}x_2
    \end{align*}
    where $u'=\frac{r}{r+1-s}$. In addition, 
    \begin{align*}
        \|\wr\;(v\;p)_{\mathbf{1}}\|^{s+1}_{L^{s+1}([0,T_0] \x\R)}
        \le\|\wr\|_{\infty} C(1,s)^{s} \|\wr\;(v\;p)_{\mathbf{1}}\|^{s}_{\H^{s,1}([0,T_0] \x \R)};
    \end{align*}
    
    \item In \Cref{prop:holder_estimates_p}, for any $r > 2$ and $\frac{1}{2} > \eta > \kappa > \frac{1}{r}$ s.t. $(1-2\eta) r > 2$,
    \begin{align*}
        &\|\; \wr\;p\; \|_{C^{0,\kappa-1/r}\big([0,T_0];C^{0,1-2\eta-2/r}(\R^d)\big)}
        \\
        &\le D \bigg[  \Big[ \;\big[K(\pi,r',\ab,\bb,\mathbf{1},T_0)+K(\pi,r',\ab,\bb,\ab,T_0)\big]\; 2^{d/r}
         + 1 \Big]
         \| \wr'\; p\|_{L^{r}([0,T_0] \x \R^2)}
         \\
         &~~~~~~~~~~~~~~~~~~~~~~~~~~~~~~~~~~~~~~~~+\|\bb\|_{\infty}\| \wr\; p\|_{L^{r}([0,T_0] \x \R^2)}+ \| \wr(0)\;p(0,\cdot) \|_{H^{r,1}(\R^2)} \bigg];
    \end{align*}
    
    \item In \Cref{prop:holder_estimates_p1}, for any $\overline{r} > 2$ and $\frac{1}{2} > \overline{\eta} > \overline{\kappa} > \frac{1}{\overline{r}}$ with $(1-2\overline{\eta}) \overline{r} > 1$,
    \begin{align*}
        &\|\wr p_\mathbf{1}\|_{C^{0,\overline{\kappa}-1/\overline{r}}\big([0,T_0],C^{0,1-2\overline{\eta}-1/\overline{r}}(\R)\big)}
        \\
        &\le  \overline{D} \bigg[ \|\bb\|_\infty \Big{\|} \big[ \wr p_{\mathbf{1}}, \wr'p_{\mathbf{1}}, \wr|\partial_{x_1}p|_{\mathbf{1}},\wr |\partial_{x_1}(p\;a^{1,1})|_{\mathbf{1}} \big] \Big{\|}_{L^{\overline{r}}([0,T_0] \x \R)}
        + \|\wr p_{\mathbf{1}}(0,\cdot)\|_{H^{\overline{r},1}(\R)}  \bigg].
    \end{align*}
    
\end{itemize}

\medskip
\subsection{Construction of the appropriate set for the density $p$} \label{sec:invariantset}
This section is devoted to find the appropriate upper bounds for the Lebesgue norm of $p$ and its gradient in terms of the coefficients of the SDE. These upper bounds will be used to defined the right space for $p$.

\medskip
\medskip
Let $j \ge 1$ and $\pi \in (0,1)$. Recall that $\widehat{C}_\infty$ is defined in \Cref{eq:def_C} and $\gamma=2+1/4$. For any $2 < \tau \le \gamma + j$, we can check that
\begin{align*}
    \widehat{C}_\infty(\pi,\tau',\ab,\bb,T_0)
    =\widehat{C}_\infty(\pi,\tau,\ab,\bb,T_0) \le |T_0|^{\frac{1}{\tau}-\frac{1}{\gamma+j}}\widehat{C}_\infty(\pi,\gamma+j,\ab,\bb,T_0).
\end{align*}

We assume that there is $0<\overline{R} \le 1$ s.t.
\begin{align*}
    1-\bigg[\big[\ell (\pi,3,\overline{R})+\ell (\pi,3',\overline{R}) \big] |T_0|^{\frac{1}{3}-\frac{1}{\gamma+j}} + \sum_{i=0}^j\big[\ell (\pi,\gamma+i,\overline{R})+\ell (\pi,(\gamma+i)',\overline{R}) \big] |T_0|^{\frac{1}{\gamma+i}-\frac{1}{\gamma+j}} \bigg] \widehat{C}_\infty(\pi,\gamma+j,\ab,\bb,T_0)\;\; >\;\;0.
\end{align*}

Consequently, we easily verify that, as $3 \le \gamma + j$, 
\begin{align*}
    \overline{R} \in \Rc(\pi,3,a,b,T_0)\;\;\mbox{and}\;\;\overline{R} \in \Rc(\pi,\gamma+i,\ab,\bb,T_0)\;\mbox{for each }1\le i \le j.
\end{align*}

We set $A_j(\pi)$ by
\begin{align*}
    &A_j(\pi)
    \\
    &:=
    \frac{2N}{1-\bigg[\big[\ell (\pi,3,\overline{R})+\ell (\pi,3',\overline{R}) \big] |T_0|^{\frac{1}{3}-\frac{1}{\gamma+j}} + \sum_{i=0}^j\big[\ell (\pi,\gamma+i,\overline{R})+\ell (\pi,(\gamma+i)',\overline{R}) \big] |T_0|^{\frac{1}{\gamma+i}-\frac{1}{\gamma+j}} \bigg] \widehat{C}_\infty(\pi,\gamma+j,\ab,\bb,T_0)}.
\end{align*}

It is straightforward that $K(\pi,3,\ab,\bb,\1,T_0,\overline{R})\le A_j(\pi)$ and for any $1\le i \le j$,
\begin{align*}
    K(\pi,(\gamma+i)',\ab,\bb,\1,T_0,\overline{R})
    =K(\pi,\gamma+i,\ab,\bb,\1,T_0,\overline{R}) \le A_j(\pi)
\end{align*}
where we recall that $K$ is defined in \eqref{eq:const_K}.

\begin{proposition} \label{prop:recurcive_estimates}
    Let $j \ge 1$ and $\pi \in (0,1)$. For any $i \le j$, there exists a positive constant 
    $$
        \widehat{H}_{ij}:=\widehat{H}_i(A_j(\pi),T_0,m,M,|(\wr,\cdots,\wr^{(i+1)})|_\infty)
    $$
    depending only on $A_j(\pi)$, $i$, $T_0$, $m$, $M$, and the supremum of $\wr$ and its first $(i+1)$--derivatives s.t. the constant $\widehat{H}_{ij}$ is a locally bounded function of the indicated quantities and
    \begin{align} \label{eq:recursive}
        \|\wr \; p\|_{L^{\gamma + i}([0,T_0] \x \R^2)}^{\gamma+i} \le \widehat{H}_{ij}.
    \end{align}
    
\end{proposition}

\begin{proof}
    We proceed in a recursive way. For $i=0$. By the estimate obtained in \Cref{prop:estimates-gamma} and recall in $\mathbf{Summary\;of\;the\;estimates}$, we have
    \begin{align*}
        \|\wr\;p\|_{L^\gamma([0,T_0] \x \R^d)}^\gamma \le C(2,1+1/2)\|(\wr,\wr')\|_{\infty}^{(1+1/2)/2} K(\pi,3,\ab,\bb,\1,T_0)^{1+1/2}2^d G(T_0,1+1/2,2,m,M).
    \end{align*}
    Since $3 \le \gamma + j$, $K(\pi,3,\ab,\bb,\1,T_0) \le A_j(\pi)$, if we set
    \begin{align*}
        &\widehat{H}_0(A_j(\pi),T_0,m,M,|(\wr,\cdots,\wr^{(1)})|_\infty)
        :=
        C(2,1+1/2)\|(\wr,\wr')\|_{\infty}^{(1+1/2)/2} (A_{j}(\pi))^{1+1/2}2^d G(T_0,1+1/2,2,m,M),
    \end{align*}
    the inequality \eqref{eq:recursive} is true for $i=0$. Let us assume that it is true for $i\in \{0,\cdots,j-1\}$. We check for $i=j$. By \Cref{prop:recursive}, we know that
    \begin{align*}
        &\|\wr\;p\|_{L^{\gamma+j}([0,T_0] \x \R^d)}^{\gamma+j}
        \le 
        \|\wr\|_{\infty} C(2,\gamma+j-1)^{\gamma+j-1} K(\pi,\gamma+j-1,\ab,\bb,\1,T_0)^{\gamma+j-1} 2^d \|\wr'\;p\|_{L^{\gamma+j-1}([0,T_0] \x \R^d)}^{\gamma+j-1}.
    \end{align*}
    Again we have $K(\pi,\gamma+j-1,\ab,\bb,\1,T_0) \le A_j(\pi)$. Also, by recursive assumption
    \begin{align*}
        \|\wr'\;p\|_{L^{\gamma+j-1}([0,T_0] \x \R^d)}^{\gamma+j-1}
        \le 
        \widehat{H}_{j-1}(A_j(\pi),T_0,m,M,|(\wr',\cdots,\wr'^{(j)})|_\infty).
    \end{align*}
    Let $z:=\gamma+j-1,$ it is enough to define
    \begin{align*}
        &\widehat{H}_{j}(A_j(\pi),T_0,m,M,|(\wr,\cdots,\wr^{(j+1)})|_\infty)
        :=
        \|\wr\|_{\infty} C(2,z)^{z} (A_j(\pi))^{z} 2^d \widehat{H}_{j-1}(A_j(\pi),T_0,m,M,|(\wr',\cdots,\wr'^{(j)})|_\infty)
    \end{align*}
    to conclude the proof.
\end{proof}
\paragraph*{Reminder of the main estimates} Combining the previous results of $\mathbf{Summary\;of\;the\;estimates}$ and \Cref{prop:recurcive_estimates}, we get
\begin{itemize}
    \item For any $j \ge 1$ (see also \Cref{prop:inegal_general}),
    \begin{align*}
        \|\wr\; p\|_{\H^{\gamma+j,1}([0,T_0] \x \R^d)}^{\gamma+j} \le (A_j(\pi))^{\gamma+j} 2^d \widehat{H}_j\big(A_j(\pi),T_0,m,M,|(\wr',\cdots,(\wr')^{(j+1)})|_\infty\big);
    \end{align*}

    \item  For all $j \ge 1$ (see also \Cref{prop:recursive}),
    \begin{align*}
        \int_0^{T_0} \|\wr(t)\;p(t,\cdot)\|^{\gamma+j}_{C^{0,1-\frac{2}{\gamma+j}}(\R^d)} \mathrm{d}t
        \le
        C(2,\gamma+j) (A_j(\pi))^{\gamma+j} 2^d \widehat{H}_j(A_j(\pi),T_0,m,M,|(\wr',\cdots,(\wr')^{(j+1)})|_\infty);
    \end{align*}

    \item For any $j \ge 1$ (see also \Cref{prop:estimates_marg}),
    \begin{align*}
        & \Big[\|\wr|\partial_{x_1}(v\;p)|_{\mathbf{1}}\|_{L^{\gamma+j}([0,T_0] \x \R)}^{\gamma+j}+ \|\wr(v\;p)_{\mathbf{1}}\|_{L^{\gamma+j}([0,T_0] \x \R)}^{\gamma+j}  \Big]^{\frac{1}{\gamma+j}}
        \\
        &\le
        K(\pi,\gamma+j,\ab,\bb,v,T_0) 2^d\bigg[\int_0^{T_0} \|\wr(t)\;p(t,\cdot)\|^{\gamma+j}_{C^{0,1-\frac{2}{\gamma+j}}(\R^d)} \mathrm{d}t \bigg]^{\frac{\gamma+j-1}{(\gamma+j)^2}} \int_\R \bigg[\int_0^{T_0}\Big| |\wr'(t)p_{X_2}(t,x_2)|\Big|^{\gamma+j} \mathrm{d}t \bigg]^{\frac{1}{(\gamma+j)^2}} \mathrm{d}x_2
    \end{align*}
    where $p_{X_2}(t,x_2):=\int_{\R}p(t,x_1,x_2)\;\mathrm{d}x_1$;

    \item For any $j \ge 1$, $2 <r = \gamma+j$ and $\frac{1}{2} > \eta > \kappa > \frac{1}{r}$ s.t. $(1-2\eta) r > 2$ (see also \Cref{prop:holder_estimates_p}),
    \begin{align*}
        &\|\;\wr\;p\; \|_{C^{0,\kappa-1/r}\big([0,T_0],C^{0,1-2\eta-2/r}(\R^d)\big)}
        \\
        &\le D \bigg[  \Big[\;\; \Big[1+\|\ab\|_\infty+\sup_{E \in \Sc}\|\nabla \ab\|_{L^{\gamma+j}([0,T_0] \x E)} \Big] A_j(\pi)\;2^{d/r}
         + 1 \;\;\Big]
         \| \wr'\; p\|_{L^{r}([0,T_0] \x \R^2)}
         \\
         &~~~~~~~~~~~~~~~~~~~~~~~~~~~~~~~~~~~~~~~~+\|b\|_{\infty}\| \wr\; p\|_{L^{r}([0,T_0] \x \R^d)}+ \| \wr(0)p(0,\cdot) \|_{H^{r,1}(\R^d)} \bigg];
    \end{align*}
    
    \item 
    For any $j \ge 1$, $\overline{r}=\gamma+j > 2$ and $\frac{1}{2} > \overline{\eta} > \overline{\kappa} > \frac{1}{\overline{r}}$ with $(1-2\overline{\eta}) \overline{r} > 1$ (see also \Cref{prop:holder_estimates_p1}),
    \begin{align*}
        &\|\;\wr\;p_\mathbf{1}\;\|_{C^{0,\overline{\kappa}-1/\overline{r}}\big([0,T_0],C^{0,1-2\overline{\eta}-1/\overline{r}}(\R)\big)}
        \\
        &\le \overline{D} \bigg[ \|\bb\|_\infty \Big{\|} \big[ \wr p_{\mathbf{1}}, \wr'p_{\mathbf{1}}, \wr|\partial_{x_1}p|_{\mathbf{1}},\wr |\partial_{x_1}(p\;a^{1,1})|_{\mathbf{1}} \big] \Big{\|}_{L^{\overline{r}}([0,T_0] \x \R)}
        + \|\wr p_{\mathbf{1}}(0,\cdot)\|_{H^{\overline{r},1}(\R)}  \bigg].
    \end{align*}
    
\end{itemize}

\paragraph*{Choice of the appropriate constant for the upper bound} \label{para_choice_constant} 
We recall that $\gamma=2+1/4$. Let $j \ge 1$, $2 <r = \gamma+j$ with $\frac{1}{2} > \eta > \kappa > \frac{1}{r}$ s.t. $(1-2\eta) r > 2$ and $\overline{r}=\gamma+j > 2$ with $\frac{1}{2} > \overline{\eta} > \overline{\kappa} > \frac{1}{\overline{r}}$ s.t. $(1-2\overline{\eta}) \overline{r} > 1$. We set $\pi:=1-2\overline{\eta}-\frac{1}{\gamma+j}$.  It is straightforward to check that
\begin{align*}
    &\widehat{C}_\infty(\pi,\gamma+j,\ab,\bb,\R^d,T_0)
    \\
    &\le \sup_{E \in \Sc(\R^d)}\bigg[\esup_{t \in [0,T_0]}\|\ab(t,\cdot)\|_{C^{0,\pi}(E)} + \sup_{(t,\xbb) \in [0,T_0] \x E}|\bb(t,\xbb)| + \|\nabla \ab\|_{L^{\gamma+j}([0,T_0] \x E)}\bigg]
    \\
    &\le \sup_{E \in \Sc(\R^d)}\esup_{t \in [0,T_0]}\|\ab(t,\cdot)\|_{C^{0,\pi}(E)} + \|\bb\|_\infty + \sup_{E \in \Sc(\R^d)}\|\nabla \ab\|_{L^{\gamma+j}([0,T_0] \x E)}.
\end{align*}
Let $L^\infty_j$ and $L_j$ be two constants s.t.
\begin{align*}
    \sup_{E \in \Sc(\R^d)}\sup_{t \in [0,T_0]}\|\ab(t,\cdot)\|_{C^{0,\pi}(E)} \le L^\infty_j
\end{align*}
and
\begin{align*}
    &\|\bb\|_\infty + \sup_{E \in \Sc(\R^d)}\|\nabla \ab\|_{L^{\gamma+j}([0,T_0] \x E)}
    \le \|\bb\|_\infty + 1 + \sup_{E \in \Sc(\R^d)}\|\nabla \ab\|_{L^{\gamma+j}([0,T_0] \x E)}^{\gamma+j}
    \le L_j.
\end{align*}
Let $\varepsilon>0$, we choose $0<R_{\varepsilon}(L^\infty_j+L_j):=\overline{R}\le 1$ s.t.
\begin{align*}
    1-\bigg[\big[\ell (\pi,3,\overline{R})+\ell (\pi,3',\overline{R}) \big] |T_0|^{\frac{1}{3}-\frac{1}{\gamma+j}} + \sum_{i=0}^j\big[\ell (\pi,\gamma+i,\overline{R})+\ell (\pi,(\gamma+i)',\overline{R}) \big] |T_0|^{\frac{1}{\gamma+i}-\frac{1}{\gamma+j}} \bigg] \big(L^\infty_j+L_j\big)
    =
    \varepsilon >0.
\end{align*}
Notice that, it is possible because $[0,1] \ni r \to \ell(\pi,\nu,r) \in \R_+$ is continuous with $\ell(\pi,\nu,0)=0$ for any $\nu >1$.
This leads to
\begin{align*}
        A_j(\pi) \le \frac{2N}{\varepsilon}.
    \end{align*}
By using the previous inequality and $\mathbf{Reminder\;of\;the\;main\;estimates}$, we have
\begin{align*}
    &\|\;\wr\; p\;\|_{\H^{\gamma+j,1}([0,T_0] \x \R^d)}^{\gamma+j}
    \le \Big(\frac{2N}{\varepsilon}\Big)^{\gamma+j} 2^d \widehat{H}_j\bigg(\frac{2N}{\varepsilon},T,m,M,|(\wr',\cdots,(\wr')^{(j+1)})|_\infty\bigg)
    =: \widehat{L^0}_j(|\wr^{(1)},\cdots,\wr^{(j+2)}|_\infty),
\end{align*}

\begin{align*}
        &\int_0^{T_0} \|\wr(t)\;p(t,\cdot)\|^{\gamma+j}_{C^{0,1-\frac{2}{\gamma+j}}(\R^d)} \mathrm{d}t
        \le
        C(2,\gamma+j) \widehat{L^0}_j(|\wr^{(1)},\cdots,\wr^{(j+2)}|_\infty) =: \widehat{L^1}_j(|\wr,\cdots,\wr^{(j+2)}|_\infty)
    \end{align*}
    and
    \begin{align*}
        & \Big[\|\wr|\partial_{x_1}p|_{\mathbf{1}}\|_{L^{\gamma+j}([0,T_0] \x \R)}^{\gamma+j}+ \|\;\wr p_{\mathbf{1}}\|_{L^{\gamma+j}([0,T_0] \x \R)}^{\gamma+j}  \Big]^{\frac{1}{\gamma+j}}
        \\
        &\le
        \frac{2N}{\varepsilon} 2^d\big[\widehat{L^1}_j(|\wr,\cdots,\wr^{(j+2)}|_\infty) \big]^{\frac{\gamma+j-1}{(\gamma+j)^2}} \int_\R \bigg[\int_0^{T_0}\Big| |\wr'(t)p_{X_2}(t,x_2)| \Big|^{\gamma+j} \mathrm{d}t \bigg]^{\frac{1}{(\gamma+j)^2}} \mathrm{d}x_2 =: \widehat{L^2}_j(|\wr,\cdots,\wr^{(j+2)}|_\infty)^{\frac{1}{\gamma+j}}.
    \end{align*}
    
    \medskip
    We choose $L_j$ s.t.
    \begin{align} \label{eq:choice1}
        1+\widehat{L^0}_j(|\wr^{(1)},\cdots,\wr^{(j+2)}|_\infty)+\widehat{L^1}_j(|\wr,\cdots,\wr^{(j+2)}|_\infty)+\widehat{L^2}_j(|\wr,\cdots,\wr^{(j+2)}|_\infty) \le \frac{L_j}{M}.
    \end{align}
    Notice that $L_j$ depends on $\varepsilon$, $j$, $N$, $T_0$, $m$, $M$, $\int_\R \Big[\int_0^{T_0}\Big| \wr'(t)p_{X_2}(t,x_2) \Big|^{\gamma+j} \mathrm{d}t \Big]^{\frac{1}{(\gamma+j)^2}} \mathrm{d}x_2$ and $(\wr,\cdots,\wr^{(j+2)})$. 
    
    \vspace{5mm}
    \medskip
    In addition,
    \begin{align*}
        &\|\; \wr\;p\; \|_{C^{0,\kappa-1/(\gamma+j)}\big([0,T_0],C^{0,1-2\eta-2/(\gamma+j)}(\R^d)\big)}
        \\
        &\le D \Big[  \big[(1+M+L_j) \frac{2N}{\varepsilon}\;2^{d/(\gamma+1)}
         + 1 + \|\bb\|_{\infty} \big]
         \widehat{L^0}_j(|\wr^{(1)},\cdots,\wr^{(j+2)}|_\infty) + \| \wr(0)p(0,\cdot) \|_{H^{\gamma+j,1}(\R^d)} \Big] =: \widehat{L^3}_j(|\wr,\cdots,\wr^{(j+2)}|_\infty)
    \end{align*}
    and
    \begin{align*}
        &\|\wr p_{\mathbf{1}}\|_{C^{0,\overline{\kappa}-1/\overline{r}}\big([0,T_0],C^{0,1-2\overline{\eta}-1/\overline{r}}(\R)\big)}
        \\
        &\le \overline{D} \bigg[ M\;\Big[2\widehat{L^2}_j(\wr,\cdots,\wr^{(j+2)}|_\infty)+\frac{2N}{\varepsilon}\Big[M + L_j \Big]2^d\Big[\widehat{L^1}_j(|\wr',\cdots,\wr^{(j+3)}|_\infty) \Big]^{\frac{\overline{r}-1}{\overline{r}^2}}+\int_\R \bigg[\int_0^{T_0}\Big| \wr'(t)p_{X_2}(t,x_2)\Big|^{\overline{r}} \mathrm{d}t \bigg]^{\frac{1}{\overline{r}^2}} \mathrm{d}x_2 \Big]  \\
        &~~~~~~~~~~~~+ \|\wr p_{\mathbf{1}}(0,\cdot)\|_{H^{\overline{r},1}(\R)}  \bigg]
        :=\widehat{L^4}_j(|\wr,\cdots,\wr^{(j+3)}|_\infty).
    \end{align*}
    
    \medskip
    We choose $L^\infty_j$ s.t.
    \begin{align}\label{eq:choice2}
        1+\widehat{L^3}_j(|\wr,\cdots,\wr^{(j+2)}|_\infty)+\widehat{L^4}_j(|\wr,\cdots,\wr^{(j+3)}|_\infty)  \le \frac{L^\infty_j}{M}.
    \end{align}
    the constant $L^\infty_j$ depends on $\varepsilon$, $j$, $N$, $D$, $\overline{D}$, $T_0$, $m$, $M$, $\int_\R \Big[\int_0^{T_0}\Big| \wr'(t)p_{X_2}(t,x_2)\Big|^{\gamma+j} \mathrm{d}t \Big]^{\frac{1}{(\gamma+j)^2}} \mathrm{d}x_2$,  $\| \wr(0)\;p(0,\cdot) \|_{H^{\gamma+j,1}(\R^d)}$, $\|\wr p_{\mathbf{1}}(0,\cdot)\|_{H^{\gamma+j,1}(\R)} $ and $(\wr,\cdots,\wr^{(j+3)})$.  
    We can easily check the next proposition
\begin{proposition} \label{prop:choice_constant}
    For any $j \ge 1$, with the previous choice of $L_j$ and $L^\infty_j$, one has if 
    $$
        \sup_{E \in \Sc(\R^d)}\sup_{t \in [0,T_0]}\|\ab(t,\cdot)\|_{C^{0,1-2\overline{\eta}-1/\overline{r}}(E)} \le L^\infty_j
    $$
    and
    $$
         \|\bb\|_\infty + 1 + \sup_{E \in \Sc(\R^d)}\|\nabla \ab\|_{L^{\gamma+j}([0,T_0] \x E)}^{\gamma+j}
        \le L_j
    $$
    then 
    \begin{align*}
        1+\|\wr\; p\|_{\H^{\gamma+j,1}([0,T_0] \x \R^d)}^{\gamma+j} &+ \int_0^{T_0} \|\wr(t)\;p(t,\cdot)\|^{\gamma+j}_{C^{0,1-\frac{2}{\gamma+j}}(\R^d)} \mathrm{d}t
        +  \Big[\|\wr|\partial_{x_1}p|_{\mathbf{1}}\|_{L^{\gamma+j}([0,T_0] \x \R)}^{\gamma+j}+ \|\;\wr p_{\mathbf{1}}\|_{L^{\gamma+j}([0,T_0] \x \R)}^{\gamma+j}  \Big]
        \;\;\le \;\;\frac{L_j}{M}
    \end{align*}
    and
    \begin{align*}
        1+\|\; \wr\;p\; \|_{C^{0,\kappa-1/(\gamma+j)}\big([0,T_0],C^{0,1-2\eta-2/(\gamma+j)}(\R^d)\big)}
        +
        \|\wr p_{\mathbf{1}}\|_{C^{0,\overline{\kappa}-1/\overline{r}}([0,T_0],C^{0,1-2\overline{\eta}-1/\overline{r}}(\R))}
        \le \frac{L^\infty_j}{M}
    \end{align*}
    where $r=\gamma+j$, $\overline{r}=\gamma+j$, $\frac{1}{2} > \eta > \kappa > \frac{1}{r}$ s.t. $(1-2\eta) r > 2$ and $\frac{1}{2} > \overline{\eta} > \overline{\kappa} > \frac{1}{\overline{r}}$ with $(1-2\overline{\eta}) \overline{r} > 1$. 
\end{proposition}

\subsection{Existence of solution}
We set $j \ge 1$, $r=\overline{r}=\gamma+j$, $\eta=\overline{\eta}$, $\kappa=\overline{\kappa}$ satisfying $\frac{1}{2} > \eta > \kappa > \frac{1}{r}$ and $(1-2\eta) r > 2$. In particular $\frac{1}{2} > \overline{\eta} > \overline{\kappa} > \frac{1}{\overline{r}}$ and $(1-2\overline{\eta}) \overline{r} > 1$. 

\medskip
We recall again that $d=2$. Let $T_0>0$ and $\theta \in [-1,1]$. We give ourselves the following maps
\begin{align*}
    [\overline{b},\overline{\sigma}]:L^1([0,T_0] \x \R^d) \x [0,T_0] \x \R^d \to \R \x \R\;\;\mbox{and}\;\;[\overline{\lambda},\overline{\beta}]:[0,T_0] \x \R \to \R^d.
\end{align*}
\begin{assumption} \label{general_coef_diffusion}
The positive numbers $m$ and $M$ are s.t. for any map $f\in L^1([0,T_0] \x \R^d)$ s.t. $f \ge 0$,  for each $t \in [0,T_0]$ and $\xbb=(x_1,x_2) \in \R^2,$ if we introduce the vector map $\overline{\bb}$ and the symmetric matrix $\overline{a}:=(\overline{\ab}^{i,j})_{1 \le i ,j \le d}$ by $\overline{\bb}^1(t,\xbb):=\overline{b}(f)(t,\xbb),$ $\overline{\bb}^2(t,\xbb):=\overline{\lambda}(t,x_2)$,
\begin{align*}
    \overline{a}^{1,1}(t,\xbb)
    :=
    \frac{1}{2}\overline{\sigma}( f) (t,\xbb)^2,\;\;\overline{a}^{2,1}(t,\xbb)
    :=
    \overline{a}^{1,2}(t,\xbb)
    :=
    \frac{1}{2}\overline{\sigma}( f) (t,\xbb)\overline{\beta}(t,x_2) \theta\;\mbox{and}\;\overline{a}^{2,2}(t,\xbb)
    :=
    \frac{1}{2}\overline{\beta}(t,x_2)^2
\end{align*}
one has that $|\overline{\bb}| \le {\rm M}$, $ m\;\mathrm{I}_d \le \overline{\ab} \le M\;\mathrm{I}_d,\;$ for  $E \subset \R^d$ open set with ${\rm diam}(E) \le 1$
\begin{align*}
    \|\nabla \overline{\ab}\|_{L^{\gamma+j}([0,T_0] \x E)}^{\gamma+j} \le M \bigg[1+ \| f \|_{\H^{\gamma+j,1}([0,T_0]\R^d)}^{\gamma+j} + \|f_{\mathbf{1}}\|_{L^{\gamma+j}([0,T_0] \x \R)}^{\gamma+j}
    +\||\partial_{x_1}f|_{\mathbf{1}}\|_{L^{\gamma+j}([0,T_0] \x \R)}^{\gamma+j}
    \bigg]
\end{align*}
and
\begin{align*}
    \sup_{t \in [0,T_0]}\|\overline{\ab}(t,\cdot)\|_{C^{0,1-2{\eta}-2/{r}}(E)} \le  M \bigg[1+ \sup_{t \in [0,T_0]}\| f(t,\cdot) \|_{C^{0,1-2{\eta}-2/r}(\R^d)} + \sup_{t \in [0,T_0]}\|f_{\mathbf{1}}(t,\cdot)\|_{C^{0,1-2{\eta}-1/{r}}(\R)}
    \bigg].
\end{align*}
The map $[0,T_0] \x \R \ni (t,x_2) \to \overline{\beta}(t,x_2) \in \R$ is Lipschitz in $x_2$ uniformly in $t$.
\end{assumption}
\medskip
Let $\wr \in C^\infty_0(T_0)$ s.t. $\wr \ge 0$, and $\varrho: [0,T_0) \x \R^d \to \R_+$ be a map s.t. for each $t \in [0,T_0),$ $\varrho(t,\cdot)$ is a density of probability i.e. $\int_{\R^d} \varrho(t,\xbb) \mathrm{d}\xbb=1$  and $\wr p$ can be extended to $[0,T_0] \x \R^d$. 
We introduce the density $p(t,\cdot)$ of $\mu_t:=\Lc(\Xbb_t)$ where $\Xbb:=(X^{1},X^{2})$ is solution of: $p(0,\cdot)=\varrho(0,\cdot)=\varrho_0(\cdot)$, 
\begin{align*}
    \mathrm{d}X^{1}_t
    =
    \overline{b} \big(\wr\varrho \big) (t,\Xbb_{t} ) \mathrm{d}t
    +
    \overline{\sigma} \big(\wr\varrho \big) (t,\Xbb_{t} ) \mathrm{d}W_t\;\;\mbox{and}\;\;\mathrm{d}X^{2}_t=\overline{\lambda}(t,X^2_{t})\mathrm{d}t+
    \overline{\beta}(t,X^2_{t})\mathrm{d}B_t
\end{align*}
$\;\mbox{where}\;\mathrm{d}\langle W, B \rangle_t= \theta \mathrm{d}t.$ 
Let $j \ge 1$, we take the initial density $\varrho_0$ s.t.
\begin{align*}
     \| \varrho_0 \|_{H^{\gamma+j,1}(\R^d)} + \|(\varrho_0)_{\mathbf{1}}\|_{H^{\gamma+j,1}(\R)} < \infty\;\;\mbox{and there is }\alpha>0\;\mbox{s.t.}\int_{\R} e^{\alpha\;|z_2|^2}\bigg|\int_{\R}\varrho_0(z_1,z_2)\mathrm{d}z_1\bigg|^{\gamma+j}\;\mathrm{d}z_2< \infty.
\end{align*}
Since \Cref{general_coef_diffusion} holds, $\overline{\lambda}$ is bounded and $\overline{\beta}$ is Lipschitz. Therefore (see Appendix \Cref{prop:appendix:integrability})
\begin{align*}
   \int_\R \bigg[\int_0^{\Tb} |p_{X_2}(t,x_2)|^{\gamma+j} \mathrm{d}t \bigg]^{\frac{1}{(\gamma+j)^2}} \mathrm{d}x_2 < \infty,\;\;\mbox{for each}\;\Tb>0,
\end{align*}
where the upper bound depends only on $\overline{\lambda}$, $\overline{\beta}$ and $\varrho_0$.
Using the properties verified by $(\overline{b},\overline{\sigma},\overline{\lambda},\overline{\beta})$ in \Cref{general_coef_diffusion}, by an obvious application of \Cref{prop:choice_constant}, we obtain the upper bounds of $p$ given the upper bounds of $\varrho$.
\begin{proposition} \label{prop:preservation}
    Let $j \ge 1$ and $\wr \in C_0^\infty(T_0)$.  For $L_j$ and $L^\infty_j$ chosen as in \eqref{eq:choice1} and \eqref{eq:choice2}, we find that: $\mbox{if}\;\;\varrho\;\;\mbox{is s.t.}\;\;\;$
\begin{align*}
    1+  \|\wr\; \varrho \|_{C^{0,\kappa-1/(\gamma+j)}\big([0,T_0],C^{0,1-2\eta-2/(\gamma+j)}(\R^d)\big)} + \|\wr\;\varrho_{\mathbf{1}}\|_{C^{0,\overline{\kappa}-1/\overline{r}}\big([0,T_0],C^{0,1-2\overline{\eta}-1/\overline{r}}(\R)\big)} \le \frac{L^\infty_j}{M}
\end{align*}
and 
\begin{align*}
     1+\;\|b\|_\infty +\Big[\|\wr\;|\partial_{x_1}\varrho|_{\mathbf{1}} \|_{L^{\gamma+j}([0,T_0] \x \R)}^{\gamma+j} + \|\wr\;\varrho_{\mathbf{1}} \|_{L^{\gamma+j}([0,T_0] \x \R)}^{\gamma+j} \Big] + \|\wr\;\varrho \|_{\H^{\gamma+j,1}([0,T_0] \x \R^d)}^{\gamma+j}\le \frac{L_j}{M}
\end{align*}
    then
    \begin{align*}
        1+\|\wr\; p\|_{\H^{\gamma+j,1}([0,T_0] \x \R^d)}^{\gamma+j} &+ \int_0^{T_0} \|\wr(t)\;p(t,\cdot)\|^{\gamma+j}_{C^{0,1-\frac{2}{\gamma+j}}(\R^d)} \mathrm{d}t
        +  \Big[\|\wr|\partial_{x_1}p|_{\mathbf{1}}\|_{L^{\gamma+j}([0,T_0] \x \R)}^{\gamma+j}+ \|\;\wr p_{\mathbf{1}}\|_{L^{\gamma+j}([0,T_0] \x \R)}^{\gamma+j}  \Big]
        \le \frac{L_j}{M}
    \end{align*}
    and
    \begin{align*}
        1+\| \wr\;p \|_{C^{0,\kappa-1/(\gamma+j)}\big([0,T_0],C^{0,1-2\eta-2/(\gamma+j)}(\R^d)\big)}
        +
        \|\wr p_{\mathbf{1}}\|_{C^{0,\overline{\kappa}-1/\overline{r}}\big([0,T_0],C^{0,1-2\overline{\eta}-1/\overline{r}}(\R)\big)}
        \le \frac{L^\infty_j}{M}.
    \end{align*}
\end{proposition}

\medskip
Let $j \ge 1$ and $\varepsilon >0$ be fixed. We take  $\wr \in C_0^\infty(T_0)$ and, $L_j$ and $L^\infty_j$ chosen as previously. We say that $f \in C([0,T_0) \x \R^d;\R_+)$ belongs to $\Theta(\wr)$ if: $\int_{\R^d} f(t,\xbb) \mathrm{d}\xbb=1\;\mbox{for each}\;t\in [0,T_0),$ $f(0,\cdot)=\varrho_0(\cdot),$
\begin{align*}
    \Big[\|\wr\;|\partial_{x_1}f|_{\mathbf{1}} \|_{L^{\gamma+j}([0,T_0] \x \R)}^{\gamma+j} + \|\wr\;f_{\mathbf{1}} \|_{L^{\gamma+j}([0,T_0] \x \R)}^{\gamma+j} \Big] + \|\wr\;f \|_{\H^{\gamma+j,1}([0,T_0] \x \R^d)}^{\gamma+j} \le \frac{L_j}{M}-1
\end{align*}
and
\begin{align*}
    \| \wr\;f \|_{C^{0,\kappa-1/(\gamma+j)}\big([0,T_0],C^{0,1-2\eta-2/(\gamma+j)}(\R^d)\big)}
        +
        \|\wr f_{\mathbf{1}}\|_{C^{0,\overline{\kappa}-1/\overline{r}}\big([0,T_0],C^{0,1-2\overline{\eta}-1/\overline{r}}(\R)\big)} \le \frac{L^\infty_j}{M}-1.\;
\end{align*}

We equip $C([0,T_0) \x \R^d;\R_+)$ with the locally uniform topology i.e. $(f^n)_{n \ge 1} \subset C([0,T_0) \x \R^d;\R_+)$ converges to $f \in C([0,T_0) \x \R^d;\R_+)$ if for any $s < T_0$ and any $K>0,$
$$
    \Lim_{n \to \infty} \sup_{(t,x)\in [0,s] \x [-K,K]^d} |f^n(t,\xbb)-f(t,\xbb)|=0.
$$

\begin{proposition}
    For any $\wr \in C_0^\infty(T_0)$ the set $\Theta(\wr)$ is convex. If for each $s<T_0$, $\inf_{t \in [0,s]} \wr(t)>0$ then $\Theta(\wr)$ is compact for the locally uniform topology.
\end{proposition}
\begin{proof}
    The convexity of $\Theta(\wr)$ for any $\wr \in C^\infty_0(T_0)$ is obvious. Let $\wr \in C^\infty_0(T_0)$ be s.t. for any  $s<T_0$, $\inf_{t \in [0,s]} \wr(t)>0$. Let $(f^n)_{n \ge 1} \subset \Theta(\wr)$. As $ \| \;\wr\;f^n\; \|_{C^{0,\kappa-1/(\gamma+j)}\big([0,T_0],C^{0,1-2\eta-2/(\gamma+j)}(\R^d)\big)} \le \frac{L^\infty_j}{M}-1$, the sequence $(\wr\;f^n)_{n \ge 1}$ is relatively compact in $C([0,T_0] \x \R^d)$ for the uniform topology. Let $f^{\wr}$ be the limit of a sub--sequence. We use the same notation for the sequence and its sub--sequence. We set $f(t,\xbb):=\frac{f^{\wr}(t,\xbb)}{\wr(t)}$ for any $(t,\xbb) \in [0,T_0) \x \R^d$. It is easy to see that for $s <T_0$,
    $$
        f(t,\xbb)=\lim_{n \to \infty} f^n(t,\xbb)=\lim_{n \to \infty} \frac{\wr(t)f^n(t,\xbb)}{\wr(t)}\;\mbox{for any }(t,\xbb) \in [0,s] \x \R^d.
    $$
    Notice that, for all $s<T_0$ and any $(t,t')\in [0,s]\x [0,s]$
    \begin{align*}
        |f(t,\xbb)-f(t',\xbb)| 
        \le 
        \frac{1}{\delta^2} \big[ \|\wr\|_\infty |\wr(t)\;f(t,\xbb)-\wr(t')\;f(t',\xbb)|+ \|\wr\;f\|_\infty |\wr(t)-\wr(t')|\big],
    \end{align*}
    where $\delta:=\inf_{t \in [0,s]} \wr(t)$.  Then $\sup_{n \ge 1} \| f^n \|_{C^{0,\kappa-1/(\gamma+j)}\big([0,s],C^{0,1-2\eta-2/(\gamma+j)}(\R^d)\big)}< \infty$.
    This is true for any $s>T_0$, consequently, the (sub--)sequence $(f^n)_{n \ge 1}$ converges towards $f$ for the locally uniform topology on $C([0,T_0) \x \R^d)$. Using  $\|\wr f^n_{\mathbf{1}}\|_{C^{0,\overline{\kappa}-1/\overline{r}}([0,T_0],C^{0,1-2\overline{\eta}-1/\overline{r}}(\R))} \le \frac{L^\infty_j}{M}-1$, by similar way, up to a sub--sequence, we also conclude that the sequence $(f^n_{\mathbf{1}})_{n \ge 1}$ converges to $f_{\mathbf{1}}$ for the locally uniform topology on $C([0,T_0) \x \R)$. We easily verify that
    \begin{align*}
         &\| \;\wr\;f\; \|_{C^{0,\kappa-1/(\gamma+j)}\big([0,T_0],C^{0,1-2\eta-2/(\gamma+j)}(\R^d)\big)} + \|\wr f_{\mathbf{1}}\|_{C^{0,\overline{\kappa}-1/\overline{r}}([0,T_0],C^{0,1-2\overline{\eta}-1/\overline{r}}(\R))}
         \\
         &\le \Liminf_{n \to \infty}
         \| \;\wr\;f^n\; \|_{C^{0,\kappa-1/(\gamma+j)}\big([0,T_0],C^{0,1-2\eta-2/(\gamma+j)}(\R^2)\big)} + \|\wr f^n_{\mathbf{1}}\|_{C^{0,\overline{\kappa}-1/\overline{r}}([0,T_0],C^{0,1-2\overline{\eta}-1/\overline{r}}(\R))}
         \le \frac{L^\infty_j}{M}-1.
    \end{align*}
    By using $\|\wr\;|\partial_{x_1}f^n|_{\mathbf{1}} \|_{L^{\gamma+j}([0,T_0] \x \R)}^{\gamma+j} + \|\wr\;\nabla f^n \|_{L^{\gamma+j}([0,T_0] \x \R^d)}^{\gamma+j} \le \frac{L_j}{M}-1$, we deduce that $\wr\;\nabla f$ and $\wr\;|\partial_{x_1}f|_{\mathbf{1}}$ are well defined. By using the compactness results from the Sobolev embedding Theorem and Fatou lemma, we check the different inequalities for $f$. All of these results allow us to conclude that $f \in \Theta(\wr)$. Consequently, $\Theta(\wr)$ is compact.
\end{proof}

\medskip
Let $\wr \in C^\infty_0(T_0)$ be s.t. for each $s<T_0$, $\inf_{t \in [0,s]} \wr(t)>0$. We introduce the application 
\begin{align*}
    \Psi_{\wr}: \Theta(\wr) \ni \varrho \mapsto \Psi_{\wr}(\varrho) \in \Theta(\wr)
\end{align*}
where $\Psi_{\wr}(\varrho)(t,\xbb)=p(t,\xbb)$ for $(t,\xbb) \in [0,T_0) \x \R^d,$ and $p$ is the density of $\mu_t:=\Lc(\Xbb_t)$ where $\Xbb:=(X^{1},X^{2})$ is solution of: $p(0,\cdot)=\varrho_0(\cdot)$, 
\begin{align*}
    \mathrm{d}X^{1}_t
    =
    \overline{b} (\wr\;\varrho ) (t,\Xbb_{t} ) \mathrm{d}t
    +
    \overline{\sigma} (\wr\;\varrho ) (t,\Xbb_{t} ) \mathrm{d}W_t\;\;\mbox{and}\;\;\mathrm{d}X^{2}_t=\overline{\lambda}(t,X^2_{t})\mathrm{d}t+
    \overline{\beta}(t,X^2_{t})\mathrm{d}B_t
\end{align*}
$\;\mbox{with}\;\mathrm{d}\langle W, B \rangle_t= \theta \mathrm{d}t$. 
Besides \Cref{general_coef_diffusion}, we assume that $(\overline{b},\overline{\sigma})$ satisfies:

\begin{assumption} \label{assumption:continuity_coeff}
     Whenever $\Lim_{n \to \infty}\|f^n-f\|_{L^1([0,s] \x \R^d)}=0$ for any $s \in [0,T_0)$ and any $\varphi \in C_c([0,T_0) \x \R^d)$,
\begin{align*} 
    \Lim_{n \to \infty}\int_{[0,T_0] \x \R^d} \overline{b}(f^n)(t,\xbb) \varphi(t,\xbb) \;\mathrm{d}\xbb\;\mathrm{d}t
    =
    \int_{[0,T_0] \x \R^d} \overline{b}(f)(t,\xbb) \varphi(t,\xbb) \;\mathrm{d}\xbb\;\mathrm{d}t
\end{align*}
and
\begin{align*}
    \Lim_{n \to \infty}\int_{[0,T_0] \x \R^d} a^{i,j}(f^n)(t,\xbb) \varphi(t,\xbb) \;\mathrm{d}\xbb\;\mathrm{d}t
    =
    \int_{[0,T_0] \x \R^d} a^{i,j}(f)(t,\xbb) \varphi(t,\xbb) \;\mathrm{d}\xbb\;\mathrm{d}t
\end{align*}
where $a^{1,1}(\rho)(t,\xbb)
    :=
    \frac{1}{2}\overline{\sigma}(\rho) (t,\xbb)^2,\;\;a^{2,1}(\rho)(t,\xbb)
    :=
    a^{1,2}(\rho)(t,\xbb)
    :=
    \frac{1}{2}\overline{\sigma}(\rho) (t,\xbb)\overline{\beta}(t,x_2) \theta\;\mbox{and}\;a^{2,2}(\rho)(t,\xbb)
    :=
    \frac{1}{2}\overline{\beta} (t,x_2)^2$.
\end{assumption}
We say that the map $\Psi_{\wr}: \Theta(\wr) \to \Theta(\wr)$ is well defined if for $\varrho \in \Theta(\wr)$ there is a unique density $\Psi_{\wr}(\varrho)$ and we have $\Psi_{\wr}(\varrho) \in \Theta(\wr).$
\begin{proposition} \label{prop:continuity_app}
    The map $\Psi_{\wr}: \Theta(\wr) \to \Theta(\wr)$ is well defined and continuous.
\end{proposition}
\begin{proof}
    By \cite[6.4.4 Corollary]{stroock2007multidimensional}, we know that for $\varrho \in \Theta(\wr),$ the density $\Psi_{\wr}(\varrho)$ is uniquely defined. By \Cref{prop:preservation}, we see that $\Psi_{\wr}(\varrho) \in \Theta(\wr).$ Let us verified that $\Psi_{\wr}$ is continuous. Let $(\varrho^n)_{n \ge 1} \subset \Theta(\wr)$ be a sequence s.t. $\lim_{n \to \infty} \varrho^n=\varrho^\infty$ for the locally uniform topology with $\varrho^\infty \in \Theta(\wr).$ As $\Theta(\wr)$ is a compact set, the sequence $\big(\Psi_{\wr}(\varrho^n)\big)_{n \ge 1} \subset \Theta(\wr)$ is relatively compact. Notice that, as $\lim_{n \to \infty} \varrho^n=\varrho^\infty$ for the locally uniform topology, using \Cref{assumption:continuity_coeff},
    and passing in the limit in the Fokker--Planck equation, we can verify that the limit of any sub--sequence of $(\Psi_{\wr}(\varrho^n)_{n \ge 1}$ is $\Psi_{\wr}(\varrho^\infty).$ Indeed, let $\widetilde{p}^\infty$ be the limit of a sub--sequence. For simplification, the sub--sequence and the sequence $(\widetilde{p}^n:=\Psi_{\wr}(\varrho^n))_{n \ge 1}$ share the same notation. For each $s<T_0$ and $K >0$, we have
    \begin{align*}
        \|\varrho^n-\varrho^\infty\|_{L^1(\R^d_{s})} 
        &\le s(2K)^d \sup_{(t,x_1,x_2) \in [0,s] \x (-K,K)^d} |(\varrho^n-\varrho^\infty)(t,x_1,x_2)|  + \|\varrho^n-\varrho^\infty\|_{L^1([0,s] \x (\R^d \setminus (-K,K)^d))}
        \\
        &\le s(2K)^d\sup_{(t,x_1,x_2) \in [0,s] \x (-K,K)^d} |(\varrho^n-\varrho^\infty)(t,x_1,x_2)| + \int_{\R^d_{s}} \1_{|\xbb| \ge K}(\varrho^n + \varrho^\infty) (t,\xbb)\mathrm{d}\xbb)\;\mathrm{d}t.
    \end{align*}
    It is easy to check that for each $t \in [0,s],$ the sequence of distribution $(\varrho^n(t,\xbb) \mathrm{d}\xbb)_{n \ge 1}$ converges weakly to $\varrho^\infty(t,\xbb) \mathrm{d}\xbb$. As $\int_{\R^d_{s}} \1_{|\xbb| = K} \varrho^\infty(t,\xbb) \mathrm{d}\xbb\;\mathrm{d}t=0$ for each $K>0$, by Portemanteau Theorem, we have that 
    \begin{align*}
        \lim_{n \to \infty} \int_{\R^d_{s}} \1_{|\xbb| \ge K} \varrho^n(t,\xbb) \mathrm{d}\xbb\;\mathrm{d}t
        =
        \int_{\R^d_{s}} \1_{|\xbb| \ge K} \varrho^\infty(t,\xbb) \mathrm{d}\xbb\;\mathrm{d}t.
    \end{align*}
    Therefore, by taking first $n \to \infty$ and then $K \to \infty$, we find that $ \lim_{n \to \infty}\|\varrho^n-\varrho^\infty\|_{L^1(\R^d_{s})}=0 $. Consequently, 
    we have $ \lim_{n \to \infty}\|\wr\varrho^n-\wr\varrho^\infty\|_{L^1(\R^d_{s})}=0 $ and $ \lim_{n \to \infty}\|\widetilde{p}^n-\widetilde{p}^\infty\|_{C^\infty([0,s]\x \R^d)}=0 $ for any $s \in[0,T_0)$. By \Cref{assumption:continuity_coeff}, we obtain that for any $\varphi \in C_c([0,T_0) \x \R^d)$ and $(\varphi^{i,j})_{1 \le i,j \le d} \subset C_c([0,T_0) \x \R^d)$
\begin{align} \label{proof:eq:conv:coeff}
    \Lim_{n \to \infty}\int_{\R^d_{T_0}} (\overline{b},\overline{\lambda})(\wr\;\varrho^n)(t,\xbb) \varphi(t,\xbb) \widetilde{p}^n(t,\xbb) \mathrm{d}\xbb\;\mathrm{d}t
    =
    \int_{\R^d_{T_0}} (\overline{b},\overline{\lambda})(\wr\;\varrho^\infty)(t,\xbb) \varphi(t,\xbb) \widetilde{p}^\infty(t,\xbb) \mathrm{d}\xbb\;\mathrm{d}t
\end{align}
and
\begin{align} \label{proof:eq:conv:coeff2}
    \Lim_{n \to \infty}\int_{\R^d_{T_0}} a^{i,j}(\wr\;\varrho^n)(t,\xbb) \varphi^{i,j}(t,\xbb) \widetilde{p}^n(t,\xbb) \mathrm{d}\xbb\;\mathrm{d}t
    =
    \int_{\R^d_{T_0}} a^{i,j}(\wr\;\varrho^\infty)(t,\xbb) \varphi^{i,j}(t,\xbb) \widetilde{p}^\infty(t,\xbb) \mathrm{d}\xbb\;\mathrm{d}t.
\end{align}
For each $1 \le n \le \infty,$ if we introduce the vector map $\bb^n$ by ${\bb}^{n,1}(t,\xbb)=\overline{b}(\wr\;\varrho^n)(t,\xbb),$ ${\bb}^{n,2}(t,\xbb)=\overline{\lambda}(t,x_2)$, using \eqref{proof:eq:conv:coeff} and \eqref{proof:eq:conv:coeff2}, we find that, for any $r \in [0,T_0)$ and $\varphi \in C^\infty_c(\R^d_{r})$
\begin{align*}
    &\int_{\R^d} \varphi(r,\xbb) \widetilde{p}^\infty(r,\xbb)\;\mathrm{d}\xbb
    =
    \lim_{n \to \infty}\int_{\R^d} \varphi(r,\xbb) \widetilde{p}^n(r,\xbb)\;\mathrm{d}\xbb
    \\
    &= \int_{\R^d} \varphi(0,\xbb) \varrho_0(\xbb)\;\mathrm{d}\xbb
    +
    \lim_{n \to \infty}\int_{\R^d_r} \big[ \partial_t \varphi(t,\xbb) + \bb^{n,i}(t,\xbb) \partial_{x_i} \varphi(t,\xbb) + a^{i,j}(\wr\;\varrho^n)(t,\xbb) \partial_{x_i} \partial_{x_j} \varphi(t,\xbb) \big] \widetilde{p}^n(t,\xbb)\;\mathrm{d}\xbb\; \mathrm{d}t
    \\
    &=
    \int_{\R^d} \varphi(0,\xbb) \varrho_0(\xbb)\;\mathrm{d}\xbb
    +
    \int_{\R^d_r} \big[ \partial_t \varphi(t,\xbb) + \bb^{\infty,i}(t,\xbb) \partial_{x_i} \varphi(t,\xbb) + a^{i,j}(\wr\;\varrho^\infty)(t,\xbb) \partial_{x_i} \partial_{x_j} \varphi(t,\xbb) \big] \widetilde{p}^\infty(t,\xbb)\;\mathrm{d}\xbb\; \mathrm{d}t.
\end{align*}
Let ${\Xbb}:=({X}^1,{X}^2)$ be the $\R^d$--valued $\F$--adapted continuous process satisfying: $\Lc(\Xbb_0)(\mathrm{d}\xbb)=\varrho_0(\xbb)\mathrm{d}\xbb$ and for each $t \in [0,T_0]$, 
\begin{align*}
    \mathrm{d}{X}^{1}_t
    =
    \overline{b} (\wr\;\varrho^\infty ) (t,{\Xbb}_{t} ) \mathrm{d}t
    +
    \overline{\sigma} (\wr\;\varrho^\infty) (t,{\Xbb}_{t} ) \mathrm{d}W_t\;\;\mbox{and}\;\;\mathrm{d}{X}^{2}_t=\overline{\lambda}(t,{X}^2_{t})\mathrm{d}t+
    \overline{\beta}(t,{X}^2_{t})\mathrm{d}B_t.
\end{align*}
We denote by $p$ the density of $\Lc(\Xbb_t)$ i.e. $\Lc(\Xbb_t)(\mathrm{d}\xbb)=p(t,\xbb)\mathrm{d}\xbb$. By uniqueness and equivalence between Fokker--Planck equation and Mckean--Vlasov process, we deduce that for each $t \in [0,T_0)$, $p(t,\xbb)=\widetilde{p}^\infty(t,\xbb)$.
We deduce that $\widetilde{p}^\infty=\Phi_{\wr}(\varrho^\infty)$. This is true for any sub--sequence of $(\Psi_{\wr}(\varrho^n))_{n \ge 1}$. We deduce that the entire sequence $(\Psi_\wr(\varrho^n))_{n \ge 1}$ converges to $\Phi(\varrho^\infty)$. We can conclude the proof.

\end{proof}
Since $\Psi_{\wr}$ is continuous and $\Theta(\wr)$ is a compact convex set, by a direct application of the Schauder fixed--point theorem, we have the following results
\begin{theorem} \label{thm:fixed_point_existence}
    The map $\Psi_{\wr}: \Theta(\wr) \to \Theta(\wr)$ has at least one fixed point i.e. there exists $p \in \Theta(\wr)$ s.t. $p=\Psi_{\wr}(p)$. 
\end{theorem}

\begin{corollary} \label{cor:existence}
    Let $j \ge 1$. For any $T < T_0,$ there exists an $\R^2$--valued $\F$--adapted continuous process $\Xbb:=(X^1,X^2)$ satisfying: $p(0,\cdot)=\varrho_0(\cdot)$
\begin{align*}
    \mathrm{d}X^{1}_t
    =
    \overline{b} (p) \big(t,\Xbb_t \big) \mathrm{d}t
    +
    \overline{\sigma} (p) \big(t,\Xbb_t \big) \mathrm{d}W_t\;\;\mbox{and}\;\;\mathrm{d}X^{2}_t=\overline{\lambda}(t,X^2_t)\mathrm{d}t
    +
    \overline{\beta}(t,X^2_t)\mathrm{d}B_t\;\;\mbox{for any }t \le T
\end{align*}
where $\Lc(\Xbb_t)(\mathrm{d}\xbb)=p(t,\xbb)\mathrm{d}\xbb$ and $p$ satisfies
\begin{align*}
        1+\|\; p\|_{\H^{\gamma+j,1}([0,T] \x \R^d)}^{\gamma+j} +  \Big[\||\partial_{x_1}p|_{\mathbf{1}}\|_{L^{\gamma+j}([0,T] \x \R)}^{\gamma+j}+ \|\; p_{\mathbf{1}}\|_{L^{\gamma+j}([0,T] \x \R)}^{\gamma+j}  \Big]
        \le \frac{L_j}{M}
    \end{align*}
    and
    \begin{align*}
        1+\| \;p \|_{C^{0,\kappa-1/(\gamma+j)}\big([0,T],C^{0,1-2\eta-2/(\gamma+j)}(\R^d)\big)}
        +
        \| p_{\mathbf{1}}\|_{C^{0,\overline{\kappa}-1/\overline{r}}\big([0,T],C^{0,1-2\overline{\eta}-1/\overline{r}}(\R)\big)}
        \le \frac{L^\infty_j}{M}
    \end{align*}
with $L_j$ and $L^\infty_j$ depend only on $j$, $m,$ $ M,$ $\varrho_0(\cdot):=\varrho(0,\cdot)$, $T$ and $T_0$.
\end{corollary}

\begin{proof}[Proof of \Cref{cor:existence}]
    As $T < T_0$, it is enough to choose $\wr \in C^\infty_0(T_0)$ such that $\wr(t)=1$ for each $t \in [0,T]$.
\end{proof}

\subsubsection{Proof of \cref{thm:main}}

Let $T>0$, $d=2$. Given the properties of the initial density $p_0$ given in \cref{thm:main}, in order to proof \cref{thm:main},  we just need to check that the coefficients $(b,\lambda,\sigma,\beta)$ given in \Cref{assum:main1} satisfy \Cref{general_coef_diffusion} and \Cref{assumption:continuity_coeff} for $j=1$. Let $\nu >0$. We consider $T_0=T + \nu$. By considering the restriction of $(b,\lambda,\sigma,\beta)$ on $L^1([0,T_0];L^1(\R^d)) \x [0,T_0] \x \R^d$, given the property of $(b,\lambda,\sigma,\beta)$, \Cref{general_coef_diffusion} is clearly satisfied. Let us check \Cref{assumption:continuity_coeff}. We take the sequence $(f^n)_{1 \le n \le \infty}$ s.t. $\lim_{n \to \infty} \|f^n-f^\infty\|_{L^1([0,s] \x \R^d)}$ for any $s < T_0$. For $\varphi \in C_c([0,T_0) \x \R^d)$, we can find $\overline{t}<T_0$ s.t. ${\rm supp}(\varphi) \subset [0,\overline{t}] \x \R^d$. Since $\overline{t}<T_0$, we have $\lim_{n \to \infty} \|f^n-f^\infty\|_{L^1([0,\overline{t}] \x \R^d)}=0$, by the continuity assumption satisfies by $(b,\lambda,\sigma,\beta)$ in \Cref{assum:main1}, we can therefore check \Cref{assumption:continuity_coeff}.

\qed

\subsubsection{Checking the assumptions for the examples given in \Cref{paragr:example}} \label{sub_sec_checkExample}
In this part, we prove that the examples given in \Cref{paragr:example} satisfy \Cref{assum:main1}. We only check the first example, the proof for the second example is similar.   Recall that the coefficients $b^\circ$, $\sigma^\circ$, $\lambda$ and $\beta$ are given in \Cref{paragr:example}. We have the definition
\begin{align*}
    {b}(f)(t,x_1,x_2)
    =
    b^\circ \big(t,x_1,x_2,f(t,\xbb),f_{\mathbf{1}}(t,x_1),(fh)_{\mathbf{1}}(t,x_1) \big)\;\mbox{and}\;{\sigma}(f)(t,x_1,x_2)
    =
    \sigma^\circ \big(t,x_1,x_2,f(t,\xbb),f_{\mathbf{1}}(t,x_1),(fh)_{\mathbf{1}}(t,x_1) \big).
\end{align*}
As in \Cref{assum:main1}, we introduce the vector map $\overline{\bb}$ and the symmetric matrix $\overline{\ab}:=(\overline{a}^{i,j})_{1 \le i ,j \le d}$ by $\overline{\bb}^1(f)(t,\xbb)={b}(f)(t,\xbb),$ $\overline{\bb}^2(f)(t,\xbb)={\lambda}(t,x_2)$,
\begin{align*}
    \overline{a}^{1,1}(f)(t,\xbb)
    :=
    \frac{1}{2}{\sigma}( f) (t,\xbb)^2,\;\;\overline{a}^{2,1}(f)(t,\xbb)
    :=
    \overline{a}^{1,2}(f)(t,\xbb)
    :=
    \frac{1}{2}{\sigma}( f) (t,\xbb){\beta}(t,x_2) \theta\;\mbox{and}\;\overline{a}^{2,2}(f)(t,\xbb)
    :=
    \frac{1}{2}{\beta}(t,x_2)^2
\end{align*}

\medskip
{\small$\bullet$ $\mathbf{Verification\;of\;\underline{Growth\;assumption}}$} Let $T>0$. Notice that, for any $(t,x_1,x_2) \in [0,T] \x \R^2$ and $f \ge 0$, we have
\begin{align*}
    \big(t,x_1,x_2,f(t,\xbb),f_{\mathbf{1}}(t,x_1),(fh)_{\mathbf{1}}(t,x_1) \big) \in \Ec_h^T\;\;\;\mbox{and}\;\;\;\big(t,x_1,x_2,f(t,\xbb),f_{\mathbf{1}}(t,x_1),(fv)_{\mathbf{1}}(t,x_1) \big) \in \Ec_v^T,
\end{align*}
recall that $\Ec_h^T$ and $\Ec_v^T$ are defined in \Cref{eq:set_h} and \Cref{eq:set_v}. Knowing the assumptions satisfy by $(b^\circ,\sigma^\circ)$ in the example, we set 
\begin{align*}
   0< m:=\inf_{\Ec^T_v} \inf_{z \neq 0} \frac{z^\top\;\ab\;z}{|z|^2},\;\;\;
    M
    :=
    \sup_{\Ec_h^T}|\bb| + |\lambda| + \sup_{\Ec_v^T} |\overline{\ab}|+1+\sup_{\Ec^T_v} |\nabla \overline{\ab}|^q + \sum_{i=1}^k |c^i_M|^q.
\end{align*}

\medskip
We have for all $(t,\xbb) \in [0,T] \x \R^d$, $|\overline{\bb}(f)(t,\xbb)| \le M$, $m \mathrm{I}_d \le \overline{\ab}(f)(t,\xbb) \le M\mathrm{I}_d$. For any $E \subset \R^d$ open set s.t. ${\rm diam}(E) \le 1$
\begin{align*}
    \|\nabla \overline{\ab}(f)\|_{L^q([0,T] \x E)}^{q} \le M \bigg[1+ \| f\|_{\H^{q,1}([0,T] \x \R^d)}^q + \||\partial_{x_1} f|_{\mathbf{1}}\|_{L^\alpha([0,T] \x \R)}^q + \| f_{\mathbf{1}}\|_{L^\alpha([0,T] \x \R)}^q
    \bigg]
\end{align*}
and
\begin{align*}
    \sup_{t \in [0,T]}\|\overline{\ab}(f)(t,\cdot)\|_{C^{0,1-2\overline{\eta}-1/v}(E)} \le  M \bigg[1+ \sup_{t \in [0,T]}\| f(t,\cdot) \|_{C^{0,1-2{\eta}-2/(\gamma+j)}(\R^d)} + \sup_{t \in [0,T]}\|f_{\mathbf{1}}(t,\cdot)\|_{C^{0,1-2\overline{\eta}-1/v}(\R)}
    \bigg].
\end{align*}

\medskip
{\small $\bullet$ $\mathbf{Verification\;of\;\underline{Continuity\;assumption}}$} Let $T>0$ and $(f^n)_{1 \le n \le \infty}$ be a sequence s.t. $\lim_{n \to \infty} \|f^n-f^\infty\|_{L^1([0,T] \x \R^d)}=0$. We can verify that
\begin{align} \label{eq:verification_conv_noapp}
    \lim_{n \to \infty}\int_{[0,T] \x \R}|f^n_{\mathbf{1}}(t,x_1)-f^\infty_{\mathbf{1}}(t,x_1) |\; \mathrm{d}x_1\; \mathrm{d}t
    \le 
    \lim_{n \to \infty} \|f^n-f^\infty\|_{L^1([0,T] \x \R^d)}
    =
    0.
\end{align}
Let $\varphi \in C_c([0,T] \x \R^d)$ and $\varphi^{i,j} \in C_c([0,T] \x \R^d)$. The sequences $(z^n)_{n \ge 1} \subset \R^d$ and $(c^n)_{n \ge 1}$ are bounded where
\begin{align*}
    z^n
    :=
    \int_{[0,T] \x \R^d} \overline{\bb}(f^n)(t,\xbb) \varphi(t,\xbb) \mathrm{d}\xbb\;\mathrm{d}t\;\mbox{and}\;c^n:=\int_{[0,T] \x \R^d} \overline{a}^{i,j}(f^n)(t,\xbb) \varphi^{i,j}(t,\xbb) \;\mathrm{d}\xbb\;\mathrm{d}t.
\end{align*}
We only deal with $(z^n)_{n \ge 1}$, the analysis of the sequence $(c^n)_{n \ge 1}$ is similar. The sequence $(z^n)_{n \ge 1}$ is compact. Let $z^\infty$ be the limit of a convergent sub--sequence $(z^{n_k})_{k \ge 1}$. By using \Cref{eq:verification_conv_noapp}, we can find a sub--sequence $(n'_k:=U(n_k))_{k \ge 1}$ where $U$ is a non--decreasing map, s.t. $\lim_{k \to \infty} \big(f^{n'_k}(t,\xbb), (f^{n'_k}h)_{\mathbf{1}} (t,x_1) \big)=\big(f^{\infty}(t,\xbb), (f^{\infty}h)_{\mathbf{1}} (t,x_1) \big)$ a.e. $(t,\xbb)$. Consequently, since $({b}^\circ,\lambda,{\sigma}^\circ,\lambda)(t,x_1,x_2,\cdot)$ is continuous uniformly in $(t,x_1,x_2)$, by dominated convergence theorem, 
\begin{align*}
    z^\infty
    =
    \lim_{k \to \infty} z^{n_k}
    =
    \lim_{k \to \infty} z^{n'_k}
    =
    \lim_{k \to \infty} \int_{[0,T] \x \R^d} {b}(f^{n'_k})(t,\xbb) \varphi(t,\xbb) \mathrm{d}\xbb\;\mathrm{d}t
    =
    \int_{[0,T] \x \R^d} {b}(f^{\infty})(t,\xbb) \varphi(t,\xbb) \mathrm{d}\xbb\;\mathrm{d}t.
\end{align*}
We proved that any convergent--sub--sequence $(z^{n_k})_{k \ge 1}$ converges towards $\int_{[0,T] \x \R^d} \overline{\bb}(f^{\infty})(t,\xbb) \varphi(t,\xbb) \mathrm{d}\xbb\;\mathrm{d}t$. We can deduce that the entire sequence $(z^n)_{n \ge 1}$ converges to $\int_{[0,T] \x \R^d} \overline{\bb}(f^{\infty})(t,\xbb) \varphi(t,\xbb) \mathrm{d}\xbb\;\mathrm{d}t$. Therefore, \Cref{assum:main1} is satisfied for $({b},\lambda,{\sigma},\beta)$. 

\subsection{Approximation by particle system} \label{sec_approx}

Let $\delta>0$ and $G_\delta:\R \to \R_+$ be a kernel satisfying: $\int_{\R} G_\delta(x) \mathrm{d}x=1$, and for any $\alpha \ge1$,
\begin{align*}
    \Lim_{\delta \to 0}G_\delta * \varphi=\varphi,\;\mbox{a.e.},\;\;\|G_\delta * \varphi\|_{L^\alpha(\R)} \le C\|\varphi\|_{L^\alpha(\R)}\;\mbox{and}\;\|\nabla(G_\delta * \varphi)\|_{L^\alpha(\R)} \le C\|\nabla \varphi\|_{L^\alpha(\R)},\;\mbox{for all }\varphi \in C^\infty_c(\R)
\end{align*}
where $C$ is \underline{independent} of $\delta$ and $*$ denotes the convolution product. 
We set $\widehat{G}_\delta(x)=G_\delta(x_1)G_\delta(x_2).$ 
We always consider $(b,\lambda,\sigma,\beta)$ as in  \Cref{assum:main1}. We define
\begin{align*}
    (\overline{b}_\delta,\overline{\sigma}_\delta)(f)(t,x_1,x_2)
    :=
    (b,\sigma) (f_\delta) \big(t,x_1,x_2 \big)
\end{align*}
where for any $f \in L^1_{\ell oc}(\R_+;L^1(\R^d))$, $f_\delta$ is defined by $f_\delta(t,\xbb)
    :=
    \widehat{G}_\delta * (f(t,\cdot))(\xbb)$. For proving the approximation by particle system, we start by giving some results involving the regularization of the dependency w.r.t. the density.

\begin{proposition} \label{prop:regularized_existence}
    Let $j = 1$, $\delta >0$ and $T>0$.  There exists an $\R^d$--valued $\F$--adapted continuous process $\Xbb^\delta:=\Xbb:=(X^1,X^2)$ satisfying $p(0,\cdot)=\varrho_0(\cdot)$, for $t \in [0,T]$,
\begin{align*}
    \mathrm{d}X^{1}_t
    =
    \overline{b}_\delta(p^\delta) \big(t,\Xbb_t \big) \mathrm{d}t
    +
    \overline{\sigma}_\delta(p^\delta) \big(t,\Xbb_t \big)W_t\;\;\mbox{and}\;\;\mathrm{d}X^{2}_t=\lambda(t,X^2_t) )\mathrm{d}t
    +
    \beta(t,X^2_t)\mathrm{d}B_t
\end{align*}
where $p^\delta(t,x)\mathrm{d}x=\Lc(\Xbb^\delta_t)(\mathrm{d}x).$ In addition
\begin{align*}
        1+\|\; p^\delta\|_{\H^{\gamma+j,1}([0,T] \x \R^d)}^{\gamma+j} +  \Big[\||\partial_{x_1}p^\delta|_{\mathbf{1}}\|_{L^{\gamma+j}([0,T] \x \R)}^{\gamma+j}+ \|\; p^\delta_{\mathbf{1}}\|_{L^{\gamma+j}([0,T] \x \R)}^{\gamma+j}  \Big]
        \le \frac{L_j}{M}
    \end{align*}
    and
    \begin{align*}
        1+\| \;p^\delta \|_{C^{0,\kappa-1/(\gamma+j)}\big([0,T],C^{0,1-2\eta-2/(\gamma+j)}(\R^d)\big)}
        +
        \| p^\delta_{\mathbf{1}}\|_{C^{0,\overline{\kappa}-1/\overline{r}}\big([0,T],C^{0,1-2\overline{\eta}-1/\overline{r}}(\R)\big)}
        \le \frac{L^\infty_j}{M}.
    \end{align*}
    with $L_j$ and $L^\infty_j$ are {\rm\underline{independent}} of $\delta.$
\end{proposition}

\begin{proof}
    We want now to apply \cref{thm:main}. 
    For any $f$ where we recall that $f_\delta(t,\xbb)
    :=
    \widehat{G}_\delta * (f(t,\cdot))(\xbb)$. Using the property of $\widehat{G}_\delta$, we notice that
    \begin{align*}
        \|\nabla f_\delta (t,\cdot)\|_{L^\alpha(\R^d)} \le C \|\nabla f (t,\cdot)\|_{L^\alpha(\R^d)},\;\;\| |\partial_{x_1}f_\delta |_{\mathbf{1}} (t,\cdot)\|_{L^\alpha(\R)} \le C \| |\partial_{x_1}f|_{\mathbf{1}} (t,\cdot)\|_{L^\alpha(\R)}
    \end{align*}
    and
    \begin{align*}
        \| f_\delta(t,\cdot) \|_{C^{0,1-2{\eta}-2/(\gamma+j)}(\R^d)} + \|(f_\delta)_{\mathbf{1}}(t,\cdot)\|_{C^{0,1-2\overline{\eta}-1/\overline{r}}(\R)}
        \le C \Big[ \| f(t,\cdot) \|_{C^{0,1-2{\eta}-2/(\gamma+j)}(\R^d)} + \|f_{\mathbf{1}}(t,\cdot)\|_{C^{0,1-2\overline{\eta}-1/\overline{r}}(\R)} \Big].
    \end{align*}
    Therefore the map $(\overline{b}_\delta,{\lambda},\overline{\sigma}_\delta,{\beta})$ satisfies \small{$\mathbf{\underline{Growth\;assumption}}$} of \Cref{assum:main1} with $(m,M)$ independent of $\delta$. Notice that $m$ and $M$ can depend of $C$. For applying \Cref{thm:main}, we need to verify \small{$\mathbf{\underline{Continuity\;assumption}}$} of \Cref{assum:main1}. To do so, it is enough to notice that if $(f^n)_{1 \le n \le \infty}$ is a sequence s.t. $\lim_{n \to \infty} \|f^n-f^\infty\|_{L^1([0,T] \x \R^d)}=0$, one has
\begin{align} \label{eq:verification_conv}
    \Lim_{n \to \infty} \;\int_0^{T}\|\widehat{G}_\delta * (f^n(t,\cdot))-\widehat{G}_\delta * (f^\infty(t,\cdot)) \|_{L^1(\R^d)}\; \mathrm{d}t
    =
    0.
\end{align}
We can conclude that $\mathbf{\underline{Continuity\;assumption}}$ of \Cref{assum:main1} is verified. We can apply \Cref{thm:main}.

\end{proof}

As $L_1$ and $L^\infty_1$ are independent of $\delta$, the sequence $(p^\delta)_{\delta >0} \subset C([0,T] \x \R^d;\R_+)$ is relatively compact in $C([0,T] \x \R^d;\R_+).$ By using the property of $\widehat{G}_\delta$ and similar techniques as in the proof of \Cref{prop:continuity_app}, we can show that:
\begin{proposition} \label{prop:conver_regularize}
    The sequence $(p^\delta)_{\delta >0} \subset C([0,T] \x \R^d;\R_+)$ is relatively compact and the limit $p$ of any convergent sub--sequence satisfied $p(t,\xbb)\;\mathrm{d}\xbb=\Lc(\Xbb_t)(\mathrm{d}\xbb)$ where
    $\Xbb:=(X^1,X^2)$ verifies $p(0,\cdot)=\varrho_0(\cdot)$, 
\begin{align*}
    \mathrm{d}X^{1}_t
    =
    b(p) \big(t,\Xbb_t \big) \mathrm{d}t
    +
    \sigma(p) \big(t,\Xbb_t \big) \mathrm{d}W_t\;\;\mbox{and}\;\;\mathrm{d}X^{2}_t=\lambda(t,X^2_t) )\mathrm{d}t
    +
    \beta(t,X^2_t))\mathrm{d}B_t.
\end{align*}

\end{proposition}
\begin{proof}
    Let $p$ be the limit of a convergent sub--sequence of $(p^\delta)_{\delta >0}$. We use the same notation for the sequence and the sub--sequence. By using the same technique as in the proof of \Cref{prop:continuity_app}, we can show that $\lim_{\delta \to 0}\|p^\delta-p\|_{L^1(\R^d_{T})}=0$. Notice that
    \begin{align*}
        \sup_{t \in [0,T]}\sup_{x \in [-K,K]^d} |\widehat{G}_\delta * (p^{\delta}(t,\cdot))(\xbb)-\widehat{G}_\delta * (p(t,\cdot))(\xbb)| \le \sup_{(t,\xbb) \in [0,T] \x [-K,K]^d} |p^{\delta}(t,\xbb)-p(t,\xbb)| 
    \end{align*}
    and by techniques from the proof of \Cref{prop:continuity_app}
    \begin{align*}
        \lim_{K \to \infty}\limsup_{\delta \to 0}\int_{\R^d \setminus [-K,K]^d} |\widehat{G}_\delta * (p^{\delta}(t,\cdot))(\xbb)|+ |\widehat{G}_\delta * (p(t,\cdot))(\xbb)|\;\mathrm{d}\xbb=0 .
    \end{align*}
    Therefore, $\lim_{\delta \to 0} \|\widehat{G}_\delta * (p^{\delta}(,\cdot))-p\|_{L^1([0,T] \x \R^d)}=0$. By using similar techniques as in proof of \Cref{prop:regularized_existence} and proof of \Cref{prop:continuity_app}, we can let $\delta \to 0$ in the Fokker--Planck equation (in weak sense) satisfied by $p^\delta$, we find that $p$ verified the desired equation.
\end{proof}
We now provide the approximation by particle system.  Let us mention that
\begin{align*}
    \widehat{G}_\delta*(p^\delta(t,\cdot))(\xbb)
    =
    \int_{\R^d} \widehat{G}_\delta(\xbb-\xbb') p^\delta(t,\xbb') \mathrm{d}\xbb'
    =
    \int_{\R^d} \widehat{G}_\delta(\xbb-\xbb') \mu^\delta_t(\mathrm{d}\xbb')
\end{align*}
where $\mu^\delta_t=\Lc(\Xbb^\delta_t).$ For any $\nu:=(\nu_t)_{t \in [0,T]} \subset \Pc(\R^d),$ we recall that $\widehat{G}_\delta * \nu(t,\xbb):=\int_{\R^d} \widehat{G}_\delta(x_1-x_1',x_2-x_2')\nu_t(\mathrm{d}x_1',\mathrm{d}x_2')$. For each $\delta >0,$ we know that $G_\delta$ is smooth then for any $\nu:=(\nu_t)_{t \in [0,T]}$ and $\nu':=(\nu'_t)_{t \in [0,T]} \in C([0,T];\Pc_e(\R^d))$
\begin{align*}
    |(\widehat{G}_\delta * \nu)_{\mathbf{1}}(t,x_1)
    -
    (\widehat{G}_\delta * \nu')_{\mathbf{1}}(t,x'_1)| + |\widehat{G}_\delta * \nu(t,\xbb)
    -
    \widehat{G}_\delta * \nu'(t,\xbb')|
    \le K(\delta) \Big[ |\xbb-\xbb'| + \sup_{t \in [0,T]}\Wc_1(\nu_t,\nu'_t) \Big],
\end{align*}
and since $\|\widehat{G}_\delta * \nu(t,\cdot)-\widehat{G}_\delta * \nu'(t,\cdot)\|_{L^1(\R^d)}=\sup_{|\phi| \le 1}\int_{\R^d} \phi(\xbb)(\widehat{G}_\delta * \nu(t,\xbb)-\widehat{G}_\delta * \nu'(t,\xbb)) \;\mathrm{d}\xbb$, we check that
\begin{align*}
    \|\widehat{G}_\delta * \nu-\widehat{G}_\delta * \nu'\|_{L^1([0,T] \x \R^d)}
    \le K(\delta) \sup_{t \in [0,T]}\Wc_1(\nu_t,\nu'_t)
\end{align*}
where $K(\delta)$ is a constant depending on $\delta$ and $T$. Consequently, under \Cref{assm_particles}, for each $T>0$ and $\delta >0$, $C([0,T];\Pc_e(\R^d)) \x [0,T] \x \R^d \ni (\nu,t,\xbb) \to [b,\sigma](\widehat{G}_\delta *\nu)(t,\xbb) \in \R^d$ is Lipschitz in $(\nu,\xbb)$ uniformly in $t$. $(W^i,B^i)_{i \ge 1}$ is a sequence of independent random variables  s.t for each $i$, $W^i$ and $B^i$ are two $\R$--valued Brownian motions s.t. $\mathrm{d} \langle W^i,B^i \rangle_t=\theta \mathrm{d}t$. The initial density $\varrho_0 \in H^{\gamma+j,1}(\R^d)$ is s.t. $\int_{\R^d} |\xbb|^r \varrho_0(\mathrm{d}\xbb) < \infty$ for $r > e \ge 1$. 
\begin{proposition} \label{prop:regularize:particle}
    For each $\delta >0,$ if we let $(\Xbb^{N,1},\cdots,\Xbb^{N,N})$ be the solution of: $(\Xbb^{N,1}_0,\cdots,\Xbb^{N,N}_0)$ is i.i.d.,  $\Lc(\Xbb^{N,i}_0)(\mathrm{d}\xbb)=\varrho_0(\xbb) \mathrm{d}\xbb$,
    \begin{align*}
    \mathrm{d}\Xbb^{N,i,1}_t
    =
    b\big(\widehat{G}_\delta * \mu^N  \big) \big(t,\Xbb^{N,i}_t \big) \mathrm{d}t
    +
    \sigma \big(\widehat{G}_\delta * \mu^N  \big) \big(t,\Xbb^{N,i}_t \big) \mathrm{d}W^i_t\;\;\mbox{and}\;\;\mathrm{d}\Xbb^{N,i,2}_t=\lambda (t,\Xbb^{N,i,2}_t)\mathrm{d}t
    +
    \beta (t,\Xbb^{N,i,2}_t)\mathrm{d}B^i_t
\end{align*}
where $\mu^N_\cdot:=\sum_{i=1}^N \delta_{\Xbb^{N,i}_\cdot}$, one has that
\begin{align*}
    \Lim_{N \to \infty} \Lc(\mu^N)=\delta_{\mu^\delta}\;\mbox{{\rm in}}\;\Wc_e.
\end{align*}
In addition, for any $k \ge 1$ and any bounded measurable map $\phi: [0,T] \x \R^{2k} \to \R$
\begin{align*}
    \Lim_{N \to \infty}\E \bigg[\int_0^T \phi(t,\Xbb^{N,1}_t,\cdots,\Xbb^{N,k}_t)\;\mathrm{d}t \bigg]
    =
    \int_{[0,T] \x \R^{2k}} \phi(t,\xbb^1,\cdots,\xbb^k)\;p^\delta(t,\xbb^1)\cdots p^\delta(t,\xbb^k) \;\mathrm{d}\xbb^1\cdots\mathrm{d}\xbb^k\;\mathrm{d}t.
\end{align*}
\end{proposition}

\begin{proof}
    Since, for each $\delta >0$, the map $[0,T] \x \R^d \x C([0,T];\Pc_e(\R^d)) \ni (t,(x_1,x_2),\nu) \to (b,\lambda,\sigma,\beta)(\widehat{G}_\delta * \nu)(t,\xbb) \in \R^4$ is Lipschitz in $(\xbb,\nu)$ uniformly in $t$. Therefore, for each $N \ge 1$, $(\Xbb^{N,1},\cdots,\Xbb^{N,N})$ is uniquely defined. Also, the process $\Xbb^\delta$ is uniquely defined in distribution where
    $\Xbb^\delta:=\Xbb:=(X^1,X^2)$ satisfies: $p(0,\cdot)=\varrho_0(\cdot)$, for $t \in [0,T]$,
\begin{align*}
    \mathrm{d}X^{1}_t
    =
    \overline{b}_\delta(p^\delta) \big(t,\Xbb_t \big) \mathrm{d}t
    +
    \overline{\sigma}_\delta(p^\delta) \big(t,\Xbb_t \big)W_t\;\;\mbox{and}\;\;\mathrm{d}X^{2}_t=\lambda(t,X^2_t) )\mathrm{d}t
    +
    \beta(t,X^2_t)\mathrm{d}B_t
\end{align*}
where $p^\delta(t,\xbb)\mathrm{d}\xbb=\Lc(\Xbb^\delta_t)(\mathrm{d}\xbb).$ Let $(\Ybb^{N,1},\cdots,\Ybb^{N,N})$ be the solution of: $\Ybb^i_0=\Xbb^i_0$,
\begin{align*}
    \mathrm{d}\Ybb^{N,i,1}_t
    =
    \overline{b}_\delta(p^\delta) \big(t,\Ybb^{N,i,1}_t \big) \mathrm{d}t
    +
    \overline{\sigma}_\delta(p^\delta) \big(t,\Ybb^{N,i,1}_t \big)W^i_t\;\;\mbox{and}\;\;\mathrm{d}\Ybb^{N,i,2}_t=\lambda(t,\Ybb^{N,i,2}_t) )\mathrm{d}t
    +
    \beta(t,\Ybb^{N,i,2}_t)\mathrm{d}B^i_t.
\end{align*}
Notice that the sequence $(\Ybb^{N,i})_{1 \le i \le N}$ is i.i.d. with $\Lc(\Ybb^{N,i}_t)=p^\delta(t,\xbb)\mathrm{d}\xbb$.  Since $\int_{\R^d} |\xbb|^r \varrho_0(\mathrm{d}\xbb) < \infty$ for $r >e$, we can apply \cite[Proposition 4.15]{djete2019general} and find: for each $i \ge 1$,
\begin{align*}
    \lim_{N \to \infty}\E^{\P} \bigg[ \sup_{t \in [0,T]} |\Ybb^{N,i}_t-\Xbb^{N,i}_t| \bigg]=0.
\end{align*}
We can conclude that $\Lim_{N \to \infty} \Lc(\mu^N)=\delta_{\mu^\delta}\;\mbox{{\rm in}}\;\Wc_e$ and for each $k \ge 1$, $\Lim_{N \to \infty}\Lc(\Xbb^{N,1},\cdots,\Xbb^{N,k})=\Lc(\Xbb^\delta) \otimes \cdots \otimes \Lc(\Xbb^\delta)$ in $\Wc_e$. Let us show the last result. Let $f^{N}(t,\xbb^1,\cdots,\xbb^k)$ be the density of $\Lc(\Xbb^{N,1}_t,\cdots,\Xbb^{N,k}_t)$ i.e. $\Lc(\Xbb^{N,1}_t,\cdots,\Xbb^{N,k}_t)=f^{N}(t,\xbb^1,\cdots,\xbb^k) \mathrm{d}\xbb^1\cdots\mathrm{d}\xbb^k$. As $m \mathrm{I}_d \le a$ and $|(\ab,\bb)| \le M$, by \Cref{prop:density-integrability}, we have
\begin{align*}
    \sup_{N \ge 1}\int_{[0,T] \x \R^{2k}}|f^{N}(t,\xbb^1,\cdots,\xbb^k)|^{(2k+1)'} \mathrm{d}\xbb^1\cdots\mathrm{d}\xbb^k\;\mathrm{d}t < \infty.
\end{align*}
The sequence $(f^N)_{N \ge 1}$ is relatively compact for the weak * topology i.e. there is $f$ and a sub--sequence $(N_\ell)_{\ell \ge 1}$ s.t. for any $\varphi:[0,T] \x \R^{2k} \to \R$ verifying $\int_0^T |\varphi(t,\xbb^1,\cdots,\xbb^k)|^{2k+1}\mathrm{d}\xbb^1\cdots \mathrm{d}\xbb^k\;\mathrm{d}t < \infty$, we have
\begin{align*}
    &\lim_{\ell \to \infty} \int_{[0,T] \x \R^{2k}} \varphi(t,\xbb^1,\cdots,\xbb^k)f^{N_\ell}(t,\xbb^1,\cdots,\xbb^k)\mathrm{d}\xbb^1\cdots \mathrm{d}\xbb^k\;\mathrm{d}t
    \\
    &=
    \int_{[0,T] \x \R^{2k}} \varphi(t,\xbb^1,\cdots,\xbb^k)f(t,\xbb^1,\cdots,\xbb^k)\mathrm{d}\xbb^1\cdots \mathrm{d}\xbb^k\;\mathrm{d}t.
\end{align*}
It is easy to check that any limit point $f$ satisfies $\Lc(\Xbb^\delta) \otimes \cdots \otimes \Lc(\Xbb^\delta)=f(t,\xbb^1,\cdots,\xbb^k)\mathrm{d}\xbb^1\cdots\mathrm{d}\xbb^k $. Therefore the entire sequence $(f^N)_{N \ge 1}$ converges for the weak * topology towards $f$. Notice that
\begin{align*}
    \Lim_{K \to \infty}\sup_{N \ge 1}\int_{[0,T] \x \R^{2k}} \mathbf{1}_{\{|\xbb^1|+\cdots+|\xbb^k| \ge K \}}f^{N}(t,\xbb^1,\cdots,\xbb^k)\mathrm{d}\xbb^1\cdots \mathrm{d}\xbb^k\;\mathrm{d}t
    \le \Lim_{K \to \infty} \frac{T\sup_{N \ge 1}\sum_{i=1}^k \E[\sup_{t \in [0,T]}|\Xbb^{N,i}_t|]}{K}=0.
\end{align*}
This is enough to take $\varphi$ as a measurable bounded map in the previous convergence.
\end{proof}

\paragraph*{Proof of \Cref{prop:numerical_scheme}} The proof of \Cref{prop:numerical_scheme} is just a combination of \Cref{prop:conver_regularize} and \Cref{prop:regularize:particle}.

\medskip

\bibliographystyle{plain}
\bibliography{McK--Vl_revision_version_arxiv}


\begin{appendix}
\section{Technical results}
We set $r \in (1,\infty)$, $T>0$ and $d \in \N^*$. Let $u:[0,T] \x \R^d \to \R$ be a map, we introduce the linear operator
$$
    \Lc u : \H^{r,1}([0,T] \x \R^d) \to \R
$$
in the following way: for any $\varphi \in \H^{r,1}([0,T] \x \R^d)$,
\begin{align*}
    \Lc u(\varphi)&:=\int_{[0,T] \x \R^d} \partial_t u(t,\xbb)  \varphi(t,\xbb) - a^{i,j}_0(t) \partial_{x_i} u(t,\xbb)\partial_{x_j}\varphi(t,\xbb)\;\mathrm{d}\xbb\;\mathrm{d}t,
\end{align*}
where $\ab_0:=(a_0^{i,j})_{1 \le i,j \le d}:[0,T] \to \Sc^d$ is a Borel map satisfying $0< m \mathrm{I}_d \le \ab \le M \mathrm{I}_d$. The next result is essentially a reformulation of \cite[Theorem 2.4]{krylovVMO} (see also \cite[Theorem 1.2.1]{FK-PL-equations} for the dual norm indicated in the proposition below).
\begin{proposition}\label{prop_app:estimates}
    Let $\fb:=(f^i)_{1\le i \le d}:[0,T] \x \R^d \to \R^{d}$ and $g:[0,T] \x \R^d \to \R$ be Borel maps. The map $u$ is s.t. $u(0,\cdot)=0$ and for any $\varphi \in C^\infty_c([0,T) \x \R^d)$,
    \begin{align*}
        \Lc u(\varphi)
        =
        -\int_{[0,T] \x \R^d} f^i(t,\xbb) \partial_{x_i}\varphi(t,\xbb) +g(t,\xbb) \varphi(t, \xbb)\;\mathrm{d}\xbb\;\mathrm{d}t.
    \end{align*}
    Then, there exists a positive constant $N$ depending only on $T$, $d$, $r$,  $M$ and $m$ s.t.
    \begin{align*}
        \|u\|_{\H^{r,1}([0,T]\x \R^d)} \le N\; \Big[ \|\fb\|_{L^r([0,T] \x \R^d)}^r + \|g\|_{L^r([0,T] \x \R^d)}^r  \Big]^{1/r}
        =N\;\|\Lc u\|_{\H^{r,-1}([0,T]\x \R^d)}.
    \end{align*}
\end{proposition}

\medskip
Now, let $u:[0,T] \x \R^d \to \R$ be a map s.t. there exists $G>0$ satisfying
\begin{align*}
    \int_{\R^d}|u(t,\xbb)| \mathrm{d}\xbb \le G,\;\;\mbox{for all a.e.}\;t \in [0,T].
\end{align*}    
    \begin{proposition}\label{prop:first_sobolev}
        If $1 \le r < d$, one has
        \begin{align*}
            \|u\|^\gamma_{L^{\gamma}([0,T] \x \R^d)} \le C(d,r)^r G^{r/d} \|\nabla u\|_{L^r([0,T] \x \R^d)}^r
        \end{align*}
        where $\gamma:=r \frac{d+1}{d}$ and $C(d,r)$ is a positive constant appearing in the Sobolev embedding Theorem depending only on $r$ and $d$.
    \end{proposition}
    
    \begin{proof}
        As $1 \le r < d$, by Gagliardo--Nirenberg--Sobolev inequality (Sobolev embedding Theorem), there exists $C(d,r)>0$ s.t. for each $t \in [0,T],$
        \begin{align*}
            \|u(t,\cdot)\|_{L^{r^*}(\R^d)} \le C(d,r) \|\nabla u(t,\cdot)\|_{L^r(\R^d)}\;\;\mbox{where}\;r^*:=\frac{dr}{d-r}.
        \end{align*}
        
        Let $a>1$,
        \begin{align*}
            \int_{\R^d} |u(t, \xbb)|^\ell \mathrm{d}\xbb
            =
            \int_{\R^d} |u(t,\xbb)|^{\ell-1/a}|u(t,\xbb)|^{1/a} \mathrm{d}\xbb \le G^{1/a}\bigg[\int_{\R^d} |u(t,\xbb)|^{(\ell-1/a)a'} \mathrm{d}\xbb \bigg]^{1/a'}
            =
            G^{1/a} \|u(t,\cdot)\|_{L^{(\ell-1/a)a'}(\R^d)}^{\frac{(\ell-1/a)a'}{a'}}.
        \end{align*}
        We want $(\ell-1/a)a'=\frac{dr}{d-r}$ and $r=\gamma-1/a$. This leads to $a=d/r$ and $\ell=r \frac{d+1}{d}$. Therefore
        \begin{align*}
            \|u\|^\ell_{L^\ell([0,T] \x \R^d)} \le \int_0^T \|u(t,\cdot)\|^r_{L^{r^*}(\R^d)} \mathrm{d}t \le C(d,r)^r G^{p/r} \|\nabla u\|^r_{L^r([0,T] \x \R^d)}.
        \end{align*}
    \end{proof}
    
    \begin{proposition} \label{prop:second_sobolev}
        If $d < r$, for any $r < \ell \le  r+1$, one has
        \begin{align*}
            \int_0^T\|u(t,\cdot)\|_{C^{0,1-d/r}(\R^d)}^r \mathrm{d}t \le C(d,r)^r \|u\|^r_{\H^{r,1}([0,T] \x \R^d)}
        \end{align*}
        and
        \begin{align*}
            \|u\|^\ell_{L^{\ell}([0,T] \x \R^d)} \le G^{1/s'} C(d,r)^{r/s'}\|u\|^{r/s}_{L^r([0,T] \x \R^d)} \|u\|_{\H^{r,1}([0,T] \x \R^d)}^{r/s'}
        \end{align*}
        where $s:=\frac{1}{r+1-\ell}$.
    \end{proposition}
    
    \begin{proof}
        Since $d < r$, by Morrey's inequality, there exists $C(d,r)$ s.t. for each $t \in [0,T]$
        \begin{align*}
            \|u(t,\cdot)\|_{C^{0,1-d/r}(\R^d)} \le C(d,r) \|u(t,\cdot)\|_{H^{r,1}(\R^d)}.
        \end{align*}
        Let $1 < a <\ell$, for any $s>1$
        \begin{align*}
            \|u\|^\ell_{L^\ell([0,T] \x \R^d)}
            &=
            \int_{[0,T] \x \R^d} |u(t,\xbb)|^{\ell-a}|u(t,\xbb)|^a \mathrm{d}\xbb \mathrm{d}t
            \le \int_{[0,T] \x \R^d} |u(t,\xbb)|^{\ell-a} \mathrm{d}\xbb |u(t,\cdot)|_{\infty}^a \mathrm{d}t
            \\
            &\le \bigg[ \int_0^T \Big| \int_{\R^d} |u(t,\xbb)|^{\ell-a} \mathrm{d}\xbb \Big|^s \mathrm{d}t \bigg]^{1/s} \bigg[ \int_0^T |u(t,\cdot)|^{as'}_{\infty} \mathrm{d}t \bigg]^{1/s'}
            \\
            &\le
            G^{1/s'}\bigg[  \int_0^T  \int_{\R^d} |u(t,\xbb)|^{s(\ell-a-1)}|u(t,\xbb)| \mathrm{d}\xbb \mathrm{d}t \bigg]^{1/s} \bigg[ \int_0^T |u(t,\cdot)|^{as'}_{\infty} \mathrm{d}t \bigg]^{1/s'},
        \end{align*}
        where we used Jensen inequality for the last inequality.
        We want $s(\ell-a-1)+1=r$ and $a\frac{s}{s-1}=r$. Then $s=\frac{1}{r+1-\ell}$. Since we must have $s>1$, this means $r < \ell \le r+1$. We obtain
        \begin{align*}
            \|u\|^\ell_{L^\ell([0,T] \x \R^d)}
            \le 
            G^{1/s'} \|u\|^{r/s}_{L^r([0,T] \x \R^d)} C(d,r)^{r/s'} \|u\|_{\H^{r,1}([0,T] \x \R^d)}^{r/s'}.
        \end{align*}
    \end{proof}

Let $(\Om,\F,\P)$ be a filtered probability space supporting an $\R$--valued $\F$--Brownian motion $B$. We consider bounded maps $[\lambda,\beta]:\R_+ \x \R \to \R \x \R$ s.t. $\beta$ is Lipschitz in $\R$ and $\beta \beta^\top \ge m$ with $m>0$. Let $(Z_t)_{t \ge 0}$ be the $\F$--adapted continuous process satisfying
\begin{align*}
    \mathrm{d}Z_t
    =
    \lambda(t,Z_t)\mathrm{d}t + \beta(t,Z_t)\mathrm{d}B_t.
\end{align*}
We denote by $q$ the density of $\Lc(Z_t)$ i.e. $\Lc(Z_t)(\mathrm{d}z)=q(t,z)\mathrm{d}z$.
\begin{proposition}\label{prop:appendix:integrability}
    For any $(r,\delta) \in (1,\infty) \x (0,1)$, if there is $\alpha>0$ s.t.
    \begin{align*}
        \int_{\R}  {\rm exp}^{\alpha|z_0|^2} q(0,z_0)^r\mathrm{d}z_0<\infty
    \end{align*} 
    then
\begin{align*}
   \int_\R \bigg[\int_0^{T}\big| q(t,z)\big|^{r} \mathrm{d}t \bigg]^{\delta} \mathrm{d}z < \infty,\;\;\mbox{for each}\;T>0.
\end{align*}
\end{proposition}

\begin{proof}
    Let us first assume that $\lambda=0$ and $\beta=\mathrm{I}$. In that case, $q$ is the density of the normal distribution, more specifically
    \begin{align*}
        q(t,z)
        =
        \frac{1}{\sqrt{2 \pi\;t}} \int_{\R} {\rm exp}^{-\frac{1}{2t}|z-z_0|^2} q(0,z_0)\mathrm{d}z_0.
    \end{align*}
    We can verify that
    \begin{align*}
         {\rm exp}^{-\frac{1}{2t}|z-z_0|^2} \le  {\rm exp}^{-\frac{1}{2T}|z-z_0|^2}\;\;\mbox{and}\;\;{\rm exp}^{-\frac{1}{2T}|z-z_0|^2} \le {\rm exp}^{-\frac{1-\theta}{2T}|z|^2} {\rm exp}^{\frac{1-\theta}{2T\theta}|z_0|^2}\;\mbox{for any}\;\theta \in (0,1).
    \end{align*}
    Therefore for any $\theta \in (0,1)$,
    \begin{align*}
        &\int_0^{T}\big| q(t,z)\big|^{r} \mathrm{d}t
        \le 
        \int_0^{T}\bigg| \frac{1}{\sqrt{2 \pi\;t}} \int_{\R} {\rm exp}^{-\frac{1}{2t}|z-z_0|^2} q(0,z_0)\mathrm{d}z_0\bigg|^{r} \mathrm{d}t
        \le 
        \int_0^{T} \frac{1}{\sqrt{2 \pi\;t}} \int_{\R} {\rm exp}^{-\frac{1}{2t}|z-z_0|^2} q(0,z_0)^r\mathrm{d}z_0 \;\mathrm{d}t
        \\
        &\le 
        \int_0^{T} \frac{1}{\sqrt{2 \pi\;t}} \int_{\R} {\rm exp}^{-\frac{1-\theta}{2T}|z|^2} {\rm exp}^{\frac{1-\theta}{2T\theta}|z_0|^2} q(0,z_0)^r\mathrm{d}z_0 \;\mathrm{d}t
        =
        \int_0^{T} \frac{1}{\sqrt{2 \pi\;t}} \;\mathrm{d}t\int_{\R}  {\rm exp}^{\frac{1-\theta}{2T\theta}|z_0|^2} q(0,z_0)^r\mathrm{d}z_0\;{\rm exp}^{-\frac{1-\theta}{2T}|z|^2}.
    \end{align*}
    Then
    \begin{align*}
        \int_\R \bigg[\int_0^{T}\big| q(t,z)\big|^{r} \mathrm{d}t \bigg]^{\delta} \mathrm{d}z
        \le 
        \bigg[\int_0^{T} \frac{1}{\sqrt{2 \pi\;t}} \;\mathrm{d}t\bigg]^\delta \bigg[\int_{\R}  {\rm exp}^{\frac{1-\theta}{2T\theta}|z_0|^2} q(0,z_0)^r\mathrm{d}z_0\bigg]^\delta \;\int_{\R}{\rm exp}^{-\frac{(1-\theta)\delta}{2T}|z|^2}\;\mathrm{d}z.
    \end{align*}
    By taking $\theta$ s.t. $\frac{1-\theta}{2T\theta}<\alpha$ with $(1-\theta)>0$, we can deduce the result. For the general case, by \cite{KUSUOKA2017359} (see \cite{AronsonBounds} for the absence of Dini continuity over the drift), the density $q$ has Gaussian upper bound. By using the same argument, we can conclude.
\end{proof}

\end{appendix}

\end{document}